\documentclass[11pt,a4paper]{amsart} 

\usepackage[T1]{fontenc}
\usepackage{graphicx}
\usepackage{comment}
\usepackage{amssymb}
\usepackage{color}
\usepackage{setspace}
\usepackage{mathtools}
\usepackage{tikz}
\usepackage{tikz-cd}
\usepackage{dsfont}
\usepackage{commath}
\usepackage{enumerate}
\usepackage{mathrsfs}
\usepackage{fancyhdr}
\usepackage{amsthm}
\usepackage{scalerel} 
\usepackage{fancyhdr}
\usepackage{bbm}
\usepackage{hyperref}
\usepackage[normalem]{ulem}
\usepackage{tikz}
\usepackage{subcaption}

\usetikzlibrary{decorations.fractals}

\newtheorem{theorem}{Theorem}[section]

\newtheorem{lemma}[theorem]{Lemma}

\newtheorem{corollary}[theorem]{Corollary}
\newtheorem{proposition}[theorem]{Proposition}

\newtheorem{thm}{Theorem}

\newtheorem{lem}{Lemma}

\theoremstyle{definition}
\newtheorem{definition}[theorem]{Definition}

\theoremstyle{remark}
\newtheorem{remark}[theorem]{Remark}

\numberwithin{equation}{section}

\newcommand{\rn}{\mathbb{R}^n}

\newcommand{\cH}{\mathcal{H}}

\newcommand{\cU}{\mathcal{U}}
\newcommand{\cM}{\mathcal{M}}

\newcommand{\cS}{\mathcal{S}}

\newcommand{\cF}{\mathcal{F}}

\newcommand{\cL}{\mathcal{L}}

\newcommand{\fI}{\mathfrak{I}}
\newcommand{\fL}{\mathfrak{L}}

\newcommand{\N}{\mathbb{N}}
\newcommand{\Z}{\mathbb{Z}}
\newcommand{\St}{\mathbb{S}}
\newcommand{\R}{\mathbb{R}}

\newcommand{\intrn}{\int_{\rn}}

\newcommand{\spt}{{\rm{spt}~}}

\newcommand{\dist}{\operatorname{dist}}

\newcommand{\Span}{\operatorname{span}}
\newcommand{\Id}{\operatorname{Id}}

\newcommand{\DMOs}{\operatorname{DMO_{s}}}
\newcommand{\DMOl}{\operatorname{DMO_{\ell}}}
\newcommand{\DDMOs}{\operatorname{DDMO_{s}}}
\newcommand{\DMO}{\widetilde{\operatorname{DMO}}}

\renewcommand{\tilde}{\widetilde}

\newcommand{\lip}{\operatorname{Lip}}

\newcommand{\Langle}{\left \langle}
\newcommand{\Rangle}{\right \rangle}

\DeclareMathOperator{\divr}{div}

\newcommand{\defeq}{\vcentcolon=}
\newcommand{\eqdef}{=\vcentcolon}

\def\Xint#1{\mathchoice
{\XXint\displaystyle\textstyle{#1}}%
{\XXint\textstyle\scriptstyle{#1}}%
{\XXint\scriptstyle\scriptscriptstyle{#1}}%
{\XXint\scriptscriptstyle\scriptscriptstyle{#1}}%
\!\int}
\def\XXint#1#2#3{{\setbox0=\hbox{$#1{#2#3}{\int}$ }
\vcenter{\hbox{$#2#3$ }}\kern-.6\wd0}}

\def\dashint{\Xint-}

\newcommand{\Tan}{\mathrm{Tan}}

\newcommand{\oalpha}{\mathring{\alpha}}
\newcommand{\obeta}{\mathring{\beta}}

\newcommand{\hi}{\cH^{m}_{i}}
\newcommand{\hg}{\cH^{m}_{\Gamma}}

\newcommand{\muik}{\mu_{i,k}}
\newcommand{\tjl}{T_{j,l}}
\newcommand{\pjl}{\Phi_{j,\ell}}
\newcommand{\kjl}{K_{j,\ell}}

\newcommand{\restr}{\mathbin{\vrule height 1.6ex depth 0pt width
0.13ex \vrule height 0.13ex depth 0pt width 1.3ex}}

\title[Rectifiability and tangents in a rough Riemannian setting]{Rectifiability and tangents in a rough Riemannian setting}
\author{Emily Casey}
\address{Emily Casey,
School of Mathematics, University of Minnesota, Minneapolis, MN 55455, USA}
\email{\href{mailto:ecasey@umn.edu}{ecasey@umn.edu}}

\author{Max Goering}
\address{Max Goering, Department of Mathematics and Computer Science, Technical University of Eindhoven, 5600 MB Eindhoven, The Netherlands}
\email{\href{mailto:m.l.goering@tue.nl}{m.l.goering@tue.nl}}

\author{Tatiana Toro}
\address{Tatiana Toro, Simons Laufer Mathematical Sciences Institute, Berkeley, California, 97420, USA and \newline Department of Mathematics, University of Washington, Seattle, WA 98195, USA  
}
\email{\href{mailto:toro@uw.edu}{toro@uw.edu}; \href{mailto:ttoro@slmath.org}{ttoro@slmath.org} }

\author{Bobby Wilson}
\address{Bobby Wilson, Department of Mathematics, University of Washington, Seattle, WA 98195, USA}
\email{\href{mailto:blwilson@uw.edu}{blwilson@uw.ed}}

\begin{document}

\begin{abstract}
Characterizing rectifiability of Radon measures in Euclidean space has led to fundamental contributions to geometric measure theory. Conditions involving existence of principal values of certain singular integrals \cite{mattila1995rectifiable} and the existence of densities with respect to Euclidean balls \cite{preiss1987geometry} have given rise to major breakthroughs. We study similar questions in a rough elliptic setting where Euclidean balls $B(a,r)$ are replaced by ellipses $B_{\Lambda}(a,r)$ whose eccentricity and principal axes depend on $a$.

Given $\Lambda : \R^{n} \to GL(n,\R)$, we consider the family of ellipses $B_{\Lambda}(a,r) = a + \Lambda(a) B(0,r)$. We characterize $m$-rectifiability in terms of the almost everywhere existence of the densities 
$$
\theta^{m}_{\Lambda(a)}(\mu,a) = \lim_{r \downarrow 0} \frac{\mu(B_{\Lambda}(a,r))}{r^{m}} \in (0, \infty).
$$

We characterize $m$-rectifiable measures in terms of the existence of the principal values-- and even under the weaker assumptions that
$$
\lim_{\epsilon \downarrow 0} \int_{B_{\Lambda}(a,\epsilon R) \setminus B_{\Lambda}(a, \epsilon r)} \frac{\Lambda(a)^{-1}(y-a)}{|\Lambda(a)^{-1}(y-a)|^{m+1}} d \mu(y) = 0 \quad \forall 0 < r < R
$$
when $0 < \theta^{m}_{*}(\mu,a) < \infty$ almost everywhere. 

We apply the second result to characterize $(n-1)$-rectifiable measures in $\R^{n}$ in terms of the behavior of the gradient of the single layer potential to the PDE $L_{A} u = - \divr(A \nabla u)$ under weak continuity assumptions on $A$.
\end{abstract}

\maketitle

\tableofcontents

\section{Introduction}

In this paper, we study two classical questions from geometric measure theory: Does rectifiability of a measure follow from its density properties? Does rectifiability of a measure follow from the existence of principal values of singular integrals?

The origins of Geometric Measure Theory can be traced back to the 1920s and 1930s when Besicovitch began studying the density question for $1$-dimensional sets in the plane, \cite{besicovitch1928fundamental,besicovitch1938fundamental}. A modern formulation of Besicovitch's results is that if we let $m=1$, $n=2$, and $\mu = \cH^{m} \restr E$ be a Radon measure for some Borel $E \subset \R^{n}$ then whenever
\begin{equation} \label{e:bp}
0 < \mu(\R^{n}) < \infty \qquad \text{and} \qquad 0 < \lim_{r \downarrow 0} \frac{\mu \left( B(x,r) \right)}{\omega_{m} r^{m}} < \infty \quad \mu- \text{a.e. } x,
\end{equation}
it follows $E$ is $m$-rectifiable. In \cite{morse1944phi}, it was shown that when $m=1$, $n=2$, and $\mu$ is any Radon measure, \eqref{e:bp} implies rectifiability of $\mu$. The extension  $n \ge 2$ was provided in \cite{moore1950density}. Federer proved \cite{federer1947phi} a general converse to Besicovitch's question, i.e., that is if $\mathcal{H}^m\restr E$ is an $m$-rectifiable measure then it has positive and finite density almost everywhere. In \cite{marstrand1961hausdorff} the first step to proving a converse for $m$-dimensional sets with $m \ge 2$ was made, proving that $2$-dimensional sets in $\R^{n}$ with density one at almost every point are rectifiable. In \cite{marstrand1964density}, Marstrand showed that if the $s$ density of a measure exists on a set of positive measure then $s$ is an integer. Finally the density question for sets in $\R^{n}$ was resolved in \cite{mattila1975hausdorff}, where Mattila proved that if $\mu = \cH^{m} \restr E$ has density $1$ at almost every point, then $E$ is rectifiable. Preiss ultimately resolved the density question for general Radon measures in Euclidean space in \cite{preiss1987geometry}, see Theorem \ref{t:preiss}. The introduction of \cite{preiss1987geometry} is also a great source for a detailed history of this problem and brief description of the difficulties that needed to be overcome for each subsequent generalization.

    Another fundamental problem in geometric measure theory is understanding the relationship between the regularity of a set or measure and the behavior of singular integral operators on that set or measure. In the quantitative setting, David and Semmes \cite{david1991singular,david1993analysis} showed that the $L^{2}$-boundedness of all singular integral operators of Calderon-Zygmund type is equivalent to uniform rectifiability of Ahlfors regular measures. They conjectured that the $L^{2}$-boundedness of the Riesz transform should be sufficient to imply rectifiability of Ahlfors regular measures. In \cite{mattila1995cauchy,mattila1995rectifiable} a qualitative version of this conjecture was shown to be true: existence of principal values of $m$-dimensional  Riesz transform implies rectifiability of measures under reasonable density assumptions. The conjecture of David and Semmes was resolved in the codimension one case in \cite{nazarov2014uniform}. Since then, there has been success in extending this codimension 1 quantitative characterization to the setting of other singular integrals which arise as the gradient of fundamental solutions to divergence form elliptic PDEs with the "frozen coefficient method" \cite{kenig2011layer,conde2019failure,prat2021l2,molero2021l2}. We discuss some of the most relevant new results to the current article in Section \ref{s:ipv}.

To study these problems we introduce a generalization of tangent measures called $\Lambda$-tangents. This is in line with the generalization of Preiss' tangent measures to metric groups, see \cite{mattila2005measures}. However, without a metric preserving group action, our methods fall outside those previously used.

\subsection{Densities and rectifiability}

In the seminal work \cite{preiss1987geometry}, Preiss characterized $m$-rectifiable  measures in terms of the existence of positive and finite densities, see Sections \ref{s:tanmeas} and \ref{s:dcones} for definitions and notation.

\begin{theorem}[\cite{preiss1987geometry}] \label{t:preiss}
Let $\mu$ be a Radon measure on $\R^{n}$ and $0 < \theta^{m}_{*}(\mu,a)$ for $\mu$ a.e. $a$ in $\R^{n}$. There exists a dimensional constant $\delta_{n} > 0$ so that the following are equivalent.
\begin{enumerate}
\item $\mu$ is $m$-rectifiable.
\item For $\mu$ a.e. $a$, any of the following hold:
\begin{enumerate}
\item[i)]$ \displaystyle
0 < \theta^{m}(\mu,a) < \infty.
$
\item[ii)] $\Tan(\mu,a) \subset \cM_{n,m}$,
the set of $m$-dimensional flat Radon measures on $\R^{n}$.
\item[iii)] $\theta^{m}_{*}(\mu,a) < \infty$ and
$\Tan(\mu,a) \subset \cM_{n}$, the space of flat measures on $\R^{n}$.
\item[iv)] 
$$
\frac{\theta^{m,*}(\mu,a)}{\theta^{m}_{*}(\mu,a)} - 1 < \delta_{n}.
$$
\end{enumerate}
\end{enumerate}
\end{theorem}
The first main theorem of this article, proven in Section \ref{s:densities}, shows that in the Riemannian setting (even with very rough metric) a complete analogue of Preiss' theorem holds. See Section \ref{s:ltan} for definitions of $\theta^{m}_{\Lambda}$ and $\Tan_{\Lambda}$.

\begin{theorem} \label{t:main2}
Let $\mu$ be a Radon measure on $\R^{n}$ and $\Lambda: \R^{n} \to GL(n,\R)$. If $\theta^{m}_{*}(\mu,a) > 0$ for $\mu$ a.e. $a \in \R^{n}$ and $\delta_{n}$ is as in Theorem \ref{t:preiss}, the following are equivalent:
\begin{enumerate}
\item $\mu$ is $m$-rectifiable. 
\item For $\mu$ almost every $a$, any of the following hold:
\begin{enumerate}
\item[i)] 
$ \displaystyle
0 <  \theta^{m}_{\Lambda}(\mu,a) < \infty.
$
\item[ii)] $\Tan_{\Lambda}(\mu,a) \subset \cM_{n,m}$, the set of $m$-dimensional flat Radon measures on $\R^{n}$.
\item[iii)] $\theta^{m}_{*}(\mu,a) < \infty$ and $\Tan_{\Lambda}(\mu,a) \subset \cM_{n}$, the space of flat measures on $\R^{n}$. 
\item[iv)]
$$
\frac{\theta^{m,*}_{\Lambda}(\mu,a)}{\theta^{m}_{\Lambda,*}(\mu,a)}  - 1 < \delta_{n}.
$$ 
\end{enumerate}
\end{enumerate}
\end{theorem}

\begin{remark} Depending upon eigenvalues of $\Lambda(a)$, there exist positive finite constants $c_{a}, C_{a}$ so that $B_{\Lambda}(0,c_{a}) \subset B(0,1) \subset B(0,C_{a})$. Hence, in Theorem \ref{t:main2} the hypothesis $\theta^{m}_{*}(\mu,a) > 0$ is equivalent to assuming $\theta^{m}_{\Lambda,*}(\mu,a) > 0$. 
\end{remark}

\begin{remark}
Since Theorem \ref{t:main2} holds for arbitrary $\Lambda : \R^{n} \to GL(n,\R)$ condition (i) says that the rectifiability of $\mu$ is equivalent to the existence at $\mu$ almost every $a$ of some choice of $\Lambda(a)$ so that the density $\theta^{m}_{\Lambda}(\mu,a)$ exists.
\end{remark}

There has been recent work in the literature extending the results of \cite{mattila1975hausdorff} to other settings. It is extended to some homogeneous groups in \cite{julia2022sets} and to finite-dimensional strictly convex Banach spaces by the third author in \cite{wilson2023density}. In the codimension $1$ Heisenberg and parabolic settings \cite{merlo2022geometry,merlo2022density} show that the existence of appropriate densities for measures implies rectifiability of the measure. The study of density question for measures in the Heisenberg group was started by \cite{chousionis2015marstrand} which demonstrates that Marstrand's density theorem holds for the Heisenberg group and the study of uniform measures in the Heisenberg group was initiated in \cite{chousionis2020uniform}. It is also known that locally $2$-uniform measures in $\R^{3}$  with respect to the density $\| \cdot \|_{\ell^{\infty}}$ are rectifiable, \cite{lorent2003rectifiability}.

In early drafts of this work, we used $\Lambda$-tangents to prove the equivalences of (1) and (i)-(iii) in Theorem \ref{t:main2}. While writing this paper, Bernd Kirchheim suggested an alternate proof of the equivalence of (1) and (i). His suggestion could be modified to prove the equivalence of (1) and (iv),  cf. Theorem \ref{t:decomp}. We chose to include the original proof of the equivalence of (1) and (i) for completeness, cf. Theorem \ref{t:lambdadensity}.

\subsection{Principal values and rectifiability} \label{s:ipv}

It is well known that if $K$ is an $-m$-homogeneous odd kernel sufficiently smooth\footnote{The notion of sufficiently smooth was weakened in \cite{Mas13}.} away from its singularity at $0$, then whenever $\mu$ is an $m$-rectifiable measure, the limit
$$
\lim_{\epsilon \downarrow 0} \int_{|y-x| \ge \epsilon} K(y-x) d \mu(y)  \in \R^n
$$
for $\mu$ almost every $x$, see \cite[Theorem 20.28]{mattila1999geometry} and references therein. Our next main theorem extends the classical result in two ways: first it considers a family of such kernels $\{K_{x}\}$, with respect to which one can still deduce rectifiability almost everywhere. Similar extensions have been considered in the works of \cite{puliatti2022gradient,prat2021l2,molero2021l2}, but are only stated for specific continuous families of kernels relating to a PDE and require additional $L^{2}$-boundedness assumptions. Second, it allows one to consider principal values with cut-offs defined by any norm, potentially depending on the point $x$.

\begin{theorem} \label{t:existenceofpv}
    If $\mu$ is a finite $m$-rectifiable measure on $\R^{n}$, then there exists an $M \in \N$ and a set $A \subset \R^{n}$ so that $\mu \left(\R^{n} \setminus A \right) = 0$ with the following property: For all $x \in A$, any norm $\| \cdot \|$, and any $-m$-homogeneous odd kernel $K \in C^{M}(\R^{n} \setminus \{0\} ; \R^{n})$ the principal value 
    \begin{equation} \label{e:pvexists}
        \lim_{\epsilon \downarrow 0} \int_{\|y-x\| \ge \epsilon} K(y-x) d \mu(y) \in \R^{n}.
    \end{equation}
\end{theorem}

The reason we consider $\R^{n}$-valued kernels in Theorem \ref{t:existenceofpv} is due to the following connection to PDEs. Consider the setting of a symmetric elliptic matrix $A \in \R^{n \times n}$ and the associated operator $L_{A} \defeq - \divr (A \nabla  \cdot )$. For dimensions $n \ge 2$, the fundamental solution has gradient given by
\begin{equation} \label{e:gradfs}
\nabla_{1} \Theta(x,y;A) = c_{n} \frac{A^{-1}(y-x)}{\det(A)^{1/2} \Langle A^{-1}(y-x), y-x \Rangle^{n/2}} = c_{n} \frac{(\Lambda^{-1})^{2}(y-x)}{\det(\Lambda) |\Lambda^{-1} (y-x)|^{n}},
\end{equation}
where $\Lambda$ is the unique positive definite matrix satisfying $\Lambda^{2} = A$. See, for instance, \cite{mitrea2013distributions}. Given a Radon measure $\mu$, the principal value of the gradient of the single layer potential associated to $L_{A}$ at $x$ is given by
\begin{align} \label{e:gradslp}
T_{A}\mu(x)  & = \lim_{\epsilon \downarrow 0} \int_{|\Lambda^{-1}(y-x)| \ge \epsilon } \nabla_{1} \Theta(x,y; A) d \mu(y) = \lim_{\epsilon \downarrow 0} \int_{|\Lambda^{-1}(y-x)| \ge \epsilon} \frac{\Lambda^{-2} (y-x)}{|\Lambda^{-1}(y-x)|^{n}} d \mu(y).
\end{align}
For the remainder of this paper, $\Lambda : \R^{n} \to GL(n,\R)$, is a matrix valued mapping denoted as $a \mapsto \Lambda(a)$ and $A : \R^{n} \to \R^{n \times n}$ is a symmetric\footnote{The requirement that $A$ is symmetric is born of convenience, not of necessity. Standard methods extend our results to non-symmetric $A$, see e.g., \cite{molero2021l2}.} uniformly elliptic matrix valued function. Given $m \in \{1, ..., n-1\}$, define
    \begin{align} \label{e:fixedpv}
        T^{m}_{\Lambda}\mu(x) \defeq \lim_{\epsilon \downarrow 0} \int_{|\Lambda(x)^{-1}(y-x)| \ge \epsilon } \frac{ \Lambda(x)^{-1} (y-x)}{|\Lambda(x)^{-1}(y-x)|^{m+1}} d \mu(y).
    \end{align}

Observe that
    \begin{align*}
        \Lambda(x)^{-1} T^{n-1}_{\Lambda} \mu(x) = T_{A}\mu(x), 
    \end{align*}
and therefore, in the setting of \eqref{e:gradslp}, the existence of $T^{n-1}_{\Lambda}\mu(x)$ and $T_{A}\mu(x)$ are equivalent. In this article, we use $T^{m}_{\Lambda}\mu(x)$ as it is more convenient in the geometric setting. Theorem \ref{t:fixedpv}(1) states that given lower density bounds on $\mu$, the a.e. existence of $T^{m}_{\Lambda} \mu(x)$ implies that a.e. tangents to $\mu$ are flat. Theorem \ref{t:fixedpv}(2) states a seemingly weaker condition which also implies almost all tangents are flat.

Since Theorem \ref{t:existenceofpv} holds for any norm, it in particular holds in the particular case of $\|y-x\|=|\Lambda^{-1}(x)(y-x)|$.

\begin{theorem} \label{t:fixedpv}
    Suppose $\Lambda : \R^{n} \to GL(n,\R)$ is a measurable function and $\mu$ is a finite Borel measure satisfying $\theta^{m}_{*}(\mu,x) > 0$. If for $\mu$-a.e. $x \in \rn$,
        \begin{enumerate}
            \item $T^{m}_{\Lambda}\mu(x)$ exists, or
            \item for all $0 < r < R < \infty$
            $$
                \lim_{\epsilon \downarrow 0}  \int_{\epsilon r \le |\Lambda(x)^{-1}(y-x)| \le \epsilon R} \frac{\Lambda(x)^{-1}(y-x)}{|\Lambda(x)^{-1}(y-x)|^{m+1}} d \mu(y) = 0,
            $$
        \end{enumerate}
        then $\Tan(\mu,x) \subset \cM_{n}$ for $\mu$-a.e. $x$.
\end{theorem}

We emphasize that there is no finiteness assumption on the density of $\mu$ in Theorem \ref{t:fixedpv}. Additionally, the coefficients $\Lambda$ need not satisfy any continuity assumptions nor have any uniformly controlled eccentricity. 

In the case where $\Lambda \equiv \Id_{n}$, Theorem \ref{t:fixedpv}(1) was proven in \cite{mattila1995rectifiable}. In fact, Theorem \ref{t:fixedpv}(2) was implicit in the proof, but has never before been explicitly stated, even when $\Lambda \equiv \Id_{n}$.  In \cite{mattila1995rectifiable}, the fact that $\Tan(\mu,x) \subset \cM_{n}$ almost everywhere is ultimately the consequence of a (doubly) rotationally-symmetric condition for the tangent measures. In the setting of Theorem \ref{t:fixedpv} it is not clear that tangent measures satisfy this type of symmetry. The novelty of our approach is that, taking guidance from what would occur on a Riemannian manifold, we introduce a notion of anisotropic tangent measures called $\Lambda$-tangents. They absorb the anisotropy at the level of $\mu$ to recover the same symmetry condition for $\Lambda$-tangents that was used in \cite{mattila1995rectifiable}. After showing this implies a.e. $\Lambda$-tangents to $\mu$ are flat, we recover a.e. flatness of tangents to $\mu$.

Theorems \ref{t:existenceofpv} and \ref{t:fixedpv} play a critical role in proving the following new characterizations of rectifiable measures.

\begin{theorem} \label{t:pvchars}
    Suppose $\mu$ is a finite Radon measure on $\R^{n}$ such that $0 < \theta^{m}_{*}(\mu,x) < \infty$ for $\mu$-a.e. $x \in \R^{n}$ and $\Lambda : \R^{n} \to GL(n,\R)$ is a measurable matrix-valued mapping. Then, the following are equivalent:

    \begin{enumerate}
        \item $\mu$ is $m$-rectifiable.
        \item For $\mu$-a.e. $x \in \R^{n}$, 
        $$
            \lim_{\epsilon \downarrow 0} \int_{|\Lambda(x)^{-1}(y-x)| \ge \epsilon} \frac{\Lambda(x)^{-1}(y-x)}{|\Lambda(x)^{-1}(y-x)|^{m+1}} d \mu(y) \in \R^{n}.
        $$
        \item For $\mu$-a.e. $x \in \R^{n}$ and all $0 < r < R < \infty$,
        $$
            \lim_{\epsilon \downarrow 0} \int_{\epsilon r \le |\Lambda(x)^{-1}(y-x)| \le \epsilon R} \frac{\Lambda(x)^{-1}(y-x)}{|\Lambda(x)^{-1}(y-x)|^{m+1}} d \mu(y) = 0.
        $$
    \end{enumerate}
\end{theorem}

\begin{remark}
    The proof Theorem \ref{t:pvchars} shows that (1) $\implies$ (2) $\implies$ (3) $\implies$ (1). This leaves open the question of whether there is any direct proof that (3) implies (2), without passing through the rectifiability of $\mu$. Such a proof could have further applications to the geometry of fractional dimensional sets or principal values defined with respect to other kernels. We emphasize that (2) is apriori stronger than (3), as it is equivalent to
    \begin{equation} \label{e:unboundedratio}
\lim_{r< R\downarrow 0} \int_{r \le |\Lambda(x)^{-1}(y-x)| \le R} \frac{\Lambda(x)^{-1}(y-x)}{|\Lambda(x)^{-1}(y-x)|^{m+1}} d \mu(y) = 0.
\end{equation}
\end{remark}

When $m=n-1$, we can additionally assume that $A : \R^{n} \to \R^{n \times n}$ is in $\DDMOs \cap \DMOl$ or $\DMO$, see Section \ref{s:dmo}, and relate Theorem \ref{t:fixedpv} to the elliptic equation 
$$
L_Au \defeq -\divr( A \nabla u ) = 0.
$$
Denote the fundamental solution to the equation by $\Gamma_A$, that is, $L_A \Gamma_{A} (\cdot, y)= \delta_y.$
 When $A$ has sufficiently nice varying coefficients, the expectation is that $\nabla_{1} \Gamma_{A}(x,y)$ is close to $\nabla_{1} \Theta(x,y; A(x))$ given by \eqref{e:gradfs}, but not equal.\footnote{Here $\nabla_1$ is used to denote taking the gradient in the first component only. This is a necessary distinction because $\Theta\left(x, y ; A(x)\right)$ has multiple entries that depend on $x$.} Still, there is no a priori reason that the existence of principal values of the singular integrals defined with respect to $\nabla_{1} 
\Gamma_A(x,y)$ and $\nabla_1 \Theta(x,y; A(x))$ are equivalent.

However, roughly speaking, estimates from \cite[Lemma 3.12 and 3.13]{molero2021l2} show that if $A \in \DMO$ then the principal values in \eqref{e:gradslp} and \eqref{e:starstar} converge in an $L^{1}(\mu)$ sense, see Lemma \ref{l:integrated} for the formal statement. 
This produces the following application of Theorem \ref{t:pvchars}:

\begin{theorem} \label{t:dmo}
    Suppose $n \ge 3$ and $A \in \DMOs \cap \DMOl$ is symmetric and uniformly elliptic with constant $\Lambda_{0}$, and that $\mu$ is a finite Radon measure on $\R^{n}$ satisfying $0 < \theta^{n-1}_{*}(\mu,x) \le \theta^{n-1,*}(\mu,x) < \infty$ for almost every $x$. Then the following are equivalent:
    \begin{enumerate}
        \item $\mu$ is $(n-1)$-rectifiable
        \item If additionally $A \in \DDMOs$, for $\mu$-.a.e $x \in \R^{n}$, 
        \begin{equation} \label{e:starstar}
        \lim_{\epsilon \downarrow 0} \int_{|\Lambda(x)^{-1}(y-x)| \ge \epsilon} \nabla_{1} \Gamma_{A}(x,y) d \mu(y) \in \R^{n}
        \end{equation}
        \item For all $0 < r < R$ and $\mu$-a.e. $x \in \rn$,
        \begin{equation}\label{e:boundeddmoversion}
            \lim_{\epsilon \downarrow 0} \int_{\epsilon r < |\Lambda(x)^{-1}(y-x)| \le \epsilon R} \nabla_{1} \Gamma_{A}(x,y) d \mu(y) = 0.
        \end{equation}
        \end{enumerate}
        Any of the above imply
        \begin{enumerate}
        \item[(4)] If additionally $A \in \DDMOs$, $\mu$-.a.e $x \in \R^{n}$,  $$ \lim_{\epsilon \downarrow 0} \int_{|y-x| \ge \epsilon} \nabla_{1} \Gamma_{A}(x,y) d \mu(y) \in \R^{n}.$$
    \end{enumerate}
\end{theorem}

That (1) implies (4) in Theorem \ref{t:dmo} was first stated in \cite[Proposition 1.5]{molero2021l2}, under an additional $L^{2}$-boundedness assumption which we remove. Lemma \ref{l:shapechange} played a critical role in removing this assumption and we believe the ideas behind this change of shape lemma are of independent interest.

\begin{lemma} \label{l:shapechange}
 Let $\mu$ be a finite Borel measure on $\R^{n}$ and $K$ be an odd homogeneous function of degree $-m$ satisfying $\|K\|_{C^{1}(\St^{n-1})} < \infty$. If $0 < \theta^{m}_{*}(\mu,x) \le \theta^{m,*}(\mu,x) < \infty$ and $\lim_{r \to 0} \oalpha_{\mu}(x,r) = 0$ (see \eqref{e:oalpha}), then for any norm $\| \cdot \|$, 
\begin{equation}\label{e:change of shape}
        \lim_{\epsilon \downarrow 0} \int_{\|y-x\| \geq \epsilon} K(y-x) d\mu(y) \in \R^{n} \iff \lim_{\epsilon \downarrow 0} \int_{|y-x| \geq \epsilon} K(y-x) d \mu(y) \in \R^{n}.
\end{equation}
    In the case that either principal value exists, then both share the same value.
\end{lemma}

We state Lemma \ref{l:shapechange} with the hypotheses we use to prove the theorem. However, it turns out the hypothesis that $\theta^{m,*}(\mu,x) < \infty$ is redundant. Indeed, we record several new characterizations of rectifiable measures in Theorem \ref{t:centeredchar}. In particular Theorem \ref{t:centeredchar}(2c) states that $\mu$ is $m$-rectifiable if $0 < \theta^{m}_{*}(\mu,x)$ and $\lim_{r \to 0} \oalpha_{\mu}(x,r) = 0$ for almost every $x$.

\begin{remark} \label{r:pde}
Given a positive definite matrix $A \in \R^{n \times n}$, we consider the Finsler $p$-Laplacian corresponding to the norm $x \mapsto \langle A x, x \rangle^{1/2}$. When $p=1+ \frac{n-1}{m}$, that is,
\begin{equation} \label{e:highercodimop}
L^{m}_{A}(\cdot) \defeq - \divr \left( \langle A \nabla \cdot , \nabla \cdot \rangle^{\frac{n-1}{2m}} A \nabla  \, \cdot  \right) 
\end{equation}
We note that the function $y \mapsto \tilde{\Theta}^{m}( x,y; A) = \langle A^{-1} (y-x), y-x \rangle^{\frac{1-m}{2}}$ solves $L^{m}_{A} \widetilde{\Theta}^{m}(x, y; A) = c_{0} \delta_{x}$ for some constant $c_{0}$ depending on $m,n,A$. Therefore, defining $\Theta^{m}(x, y; A) = c_{0}^{-1} \tilde{\Theta}^{m}(x, y; A)$ is the fundamental solution for $L^{m}_{A}$. A computation shows that for some $c_{1} = c_{1}(m,n,A)$,
    \begin{align*}
        c_{1} \nabla_{1} \Theta^{m}(x,y; A) =  \frac{A^{-1}(y-x)}{\langle A^{-1}(y-x), y-x \rangle^{\frac{m+1}{2}}} = \Lambda^{-1}  \frac{ \Lambda^{-1}(y-x)}{|\Lambda^{-1}(y-x)|^{m+1}},
    \end{align*}
where $\Lambda^{2} = A$. When $A$ has variable entries, let $\Gamma^{m}_A(\cdot,y)$ denote the function such that $L^{m}_{A} \Gamma^{m}_A(\cdot,y) = \delta_{y}$. In analogy to the way Theorem \ref{t:dmo} is proven, we suspect that anytime $A$ has sufficient regularity to ensure that a Finsler $(1 + \frac{n-1}{m})$ analog of Lemma \ref{l:integrated} holds, then under appropriate density assumptions the results of Theorem \ref{t:pvchars} can be extended to higher codimension.
\end{remark}

 At first, one may expect that the quantitative nature of ``the $L^{2}$-boundedness of a singular integral operator" might be stronger than assuming that principal values exist. However, for measures that are not absolutely continuous with respect to the Lebesgue measure, this is a difficult question. Intuitively, this is because the existence of a principal values depends on some sort of local symmetry of the measure (or a sufficiently small density). This intuition was recently formalized for measures that satisfy upper density assumptions \cite{jaye2020problem}. There, it is shown that for a measure with an $L^{2}$ -bounded Calderon-Zygmund-type singular integral operator, the existence of principal values is equivalent to either the density of the measure being zero or the measure being symmetric in terms of a transport distance to a family of symmetric measures.  Previous proofs that $L^{2}$-boundedness of the Riesz transform implies existence of principal values relies on the fact that $L^{2}$-boundedness implies rectifiability and then proceed with a careful extension of Calderon-Zygmund estimates to Lipschitz graphs \cite[Chapter 20]{mattila1999geometry}. 

In \cite{tolsa2008principal}, it is shown that if $0 < \theta^{m,*}(\mu,x) < \infty$ almost everywhere, then the existence of principal values with respect to the Riesz transform still characterizes rectifiability.  Furthermore, a related square function for the center of mass has been used to characterize rectifiable measures \cite{mayboroda2009finite,villa2022square} and extend some results to the $\Omega$-symmetric setting \cite{villa2021omega}. Further results on rectifiability and principal values in various settings can be found in \cite{verdera1992weak, mattila1994existence, huovinen1997singular, jaye2020small, jaye2022huovinen}.

\section*{Acknowledgments}  This work was completed while E.C. was partially supported by NSF DMS Focus Research Grant 1853993. M.G. was partially funded by NSF grant FRG-1853993, and by the European Research Council (ERC) under the European Union’s Horizon Europe research and innovation programme (grant agreement No 101087499).  T.T. was partially supported by NSF grants DMS-1954545 and DMS-1928930, and B.W. was funded by NSF grant DMS-1856124 and NSF CAREER Fellowship, DMS-2142064. The authors thank David Bate for his interest in this project and many interesting discussions, Bernd Kirchheim for very helpful conversations, and Pertti Mattila and Michele Villa for their helpful feedback on an early draft. 

\section{Background and preliminaries} 
\subsection{Tangent measures and $d$-cones} \label{s:tanmeas} 

Whenever we say $\mu$ is a Radon measure, we mean that it is a Radon outer measure; see \cite{evans2018measure}. We write $B(x,r) = \{ |y-x| \le r \}$ and $U(x,r) = \{|y-x| < r\}$. If $x = 0$ we may simply write $B_{r}$ and $U_{r}$. We say that a Radon measure $\mu$ on $\R^{n}$ is $m$-rectifiable if $\mu \ll \cH^{m}$ and there exist countably many Lipschitz maps $f_{i} : \R^{m} \to \R^{n}$ so that $\mu \left( \R^{n} \setminus \cup_{i} f_{i} (\R^{m}) \right) = 0$.

When we write $A \lesssim_{a,b} B$ this means that there exists a constant $C$ depending on $a,b$ so that $A \le C B$. We write $A \sim_{a,b} B$ to mean $A \lesssim_{a,b} B$ and $B \lesssim_{a,b} A$.

Whenever $E \subset \R^{n}$ and $r > 0$, we let $rE = \{ r x : x \in E \}$. For each $a \in \R^{n}$ and $r > 0$, define the translation and scaling map
$$
T_{a,r}(y) = \frac{y-a}{r} \qquad \forall ~ y \in \R^{n}.
$$
Given a Radon measure $\mu$ on $\R^{n}$ and a Borel $T : \R^{n} \to \R^{n}$, denote by $T[\mu]$ the image measure of $\mu$ by $T$, namely $T[\mu](E) = \mu(T^{-1}(E))$. In particular, $T_{a,r}[\mu]$ is defined by
$$
T_{a,r}[\mu](E) = \mu(T_{a,r}^{-1}(E)) = \mu(a + rE)
$$
for all $E \subset \R^{n}$. 

\begin{definition}[Tangent measures]
		Let $\mu$ be a Radon measure on $\R^{n}$. We write $\nu \in \Tan(\mu,a)$ and say that $\nu$ is a tangent measure to $\mu$ at $a$ if $\nu$ is a nonzero Radon measure and there exists $c_{i} > 0$ and $r_{i} \downarrow 0$ so that
		$$
		c_{i} T_{a,r_{i}}[\mu] \xrightharpoonup{*} \nu,
		$$
		where $\xrightharpoonup{*}$ denotes convergence in the weak-$*$ sense. In addition, we write $\Tan[\mu]$ for the weak-$*$ closure of $\cup_{a \in \spt \mu} \Tan(\mu,a)$.
	\end{definition}
For a compact set $K \subset \R^{n}$, and two Radon measures $\mu,\nu$ we define
\begin{equation} \label{e:F}
F_{K}(\mu,\nu) = \sup \left\{ \left|\int f d(\mu-\nu) \right|\mid \lip(f) \le 1,  ~ f \in C_{c}(K) \right\}.
\end{equation}
If $K = B_{r}$, we simply write $F_{r}( \cdot, \cdot)$. We recall, see \cite[Lemma 14.13]{mattila1999geometry} that for a sequence of Radon measures $\{\mu_{k}\}$ and a Radon measure $\mu$,
\begin{equation} \label{e:allradii}
\mu_{k} \xrightharpoonup{*} \mu \iff \lim_{k \to \infty} F_{r}(\mu_{k},\mu) = 0 \quad \forall r > 0.
\end{equation} 
It is well-known, see \cite[Proposition 1.12]{preiss1987geometry}, that
\begin{equation*}
F(\mu,\nu) \defeq \sum_{\ell=1}^{\infty} 2^{-\ell} \min \{1, F_{\ell}(\mu,\nu)\}
\end{equation*}
defines a metric on the space of Radon measures. Moreover, $F$ generates the topology of weak-$*$ convergence. We denote $F(\mu) = F(\mu,0)$.

\begin{proposition} \label{p:continuity}
Let $\mu$ be a Radon measure on $\R^{n}$ and $T, T_{i} : \R^{n} \to \R^{n}$ be proper homeomorphisms, that is, homeomorphisms such that $T^{-1}(K)$ is compact whenever $K$ is compact. If $\mu_{i} \xrightharpoonup{*} \mu$ and $T_{i}, T_{i}^{-1}$ converge uniformly in compact subsets to $T, T^{-1}$, respectively, then $T_{i}[\mu_{i}] \xrightharpoonup{*} T[\mu]$.
\end{proposition}

\begin{proof} 
Fix $f \in C_{c}(\R^{n})$ and let $K = \spt f$. Since $f$ has compact support, $f$ is uniformly continuous. Let $\omega$ denote its modulus of continuity. Since $T_{i}^{-1} \to T^{-1}$ locally uniformly, if $F_{i} = T^{-1}(K) \cup T_{i}^{-1}(K)$, then $F_{i} \subset F$ for some fixed compact set $F$. Since $T_{i} \to T$ locally uniformly, $\delta_{i} \defeq  \|T_{i} - T\|_{L^{\infty}(F)} \xrightarrow{i \to \infty} 0$.  Therefore,
\begin{align*}
 \lim_{i \to \infty} \bigg| \int &f d(T_{i}[\mu_{i}] - T[\mu]) \bigg| =  \lim_{i \to \infty}\left| \int f \circ T_{i} d \mu_{i} - f \circ T d \mu \right| \\
 &\le \limsup_{i} \int |f \circ T_{i} - f \circ T | d \mu_{i} + \left| \int f \circ T d(\mu_{i} - \mu) \right| \\
 & \le\limsup_{i} \omega(\delta_{i}) \mu_{i}(F_{i}).
\end{align*}
Since $\mu_{i} \xrightharpoonup{*} \mu$ and $F_{i} \subset F$, we know $\limsup_{i} \mu_{i}(F_{i}) \le \limsup_{i} \mu_{i}(F)\leq \mu(F) < \infty$. The proposition follows since $\limsup_{i} \omega(\delta_{i}) = 0$.
\end{proof}

The next theorem originates in \cite[Theorem 2.12]{preiss1987geometry}, but our presentation follows \cite[Theorem 14.16]{mattila1999geometry}.

\begin{theorem} \label{t:tan2tan}
Let $\mu$ be a Radon measure on $\rn$. Then at $\mu$ almost all $a \in \rn$ every $\nu \in \Tan(\mu,a)$ has the following two properties:
\begin{enumerate}
\item $T_{x,r}[\nu] \in \Tan(\mu,a)$ for all $x \in \spt \nu, r > 0$.
\item $\Tan(\nu,x) \subset \Tan(\mu,a)$ for all $x \in \spt \nu$.
\end{enumerate}
\end{theorem}

A useful tool for quantifying the properties of tangent measures is their distance to $d$-cones. 

\begin{definition}[Cones, $d$-cones, and basis]
A collection of nonzero Radon measures $\cM$ is called a \emph{cone} if $\mu \in \cM \implies c \mu \in \cM$ for all $c > 0$. A cone of Radon measures is called a  $d$-\emph{cone} if $\mu \in \cM \implies T_{0,r}[\mu] \in \cM$ for all $r > 0$. The \emph{basis} of a d-cone is the collection of $\mu \in \cM$ so that $F_{1}(\mu) = 1$. We let $\cM_{B}$ denote the basis of $\cM$. A $d$-cone $\cM$ is said to have a \emph{closed} (respectively \emph{compact}) \emph{basis} if the basis is closed (respectively compact) with respect to the weak-$*$ topology.
\end{definition}

\begin{proposition}{\cite[Proposition 2.2]{preiss1987geometry}} \label{p:compactbasis}
If a $d$-cone $\cM$ of Radon measures has closed basis, then $\cM$ has a compact basis if and only if for every $\lambda \ge 1$ there is a $\tau = \tau(\lambda) > 1$ so that
\begin{equation} \label{e:compactchar}
F_{\tau r}(\mu) \le \lambda F_{r}(\mu) \quad \forall \mu \in \cM \quad \forall r > 0.
\end{equation}
In this case, $0 \in \spt \mu$ for all $\mu \in \cM$.
\end{proposition}

Let $\cM$ be a $d$-cone and $\nu$ a Radon measure in $\R^{n}$. If $s> 0$ and $0< F_{s}(\nu) < \infty$ we define the \emph{distance} between $\nu$ and $\cM$ at scale $s$ by
\begin{equation} \label{e:ds}
d_{s}(\nu,\cM) = \inf \left\{ F_{s} \left( \frac{\nu}{F_{s}(\nu)}, \mu \right) \mid \mu \in \cM ~ \text{and} ~ F_{s}(\mu) = 1 \right\}.
\end{equation}
If $F_{s}(\nu) \in \{0, \infty\}$, we define $d_{s}(\nu,\cM) = 1$.

\begin{proposition} \cite[Remark 2.1 and 2.2]{kenig2009boundary} \label{p:kptrs}
If $\mu, \nu$ are Radon measures, 
\begin{equation} \label{e:fscaling}
F_{r}(\mu,\nu) = r F_{1}(T_{0,r} [\mu], T_{0,r}[\nu]).
\end{equation}

If $\cM$ is a $d$-cone and $\nu$ a Radon measure,
\begin{enumerate}
\item[i)] $d_{s}(\nu, \cM) \le 1$ for all $s > 0$.
\item[ii)] $d_{s}(\nu, \cM) = d_{1} \left( T_{0,s}[\nu],\cM\right)$ for all $s > 0$.
\item[iii)] If $\nu_{i} \xrightharpoonup{*} \nu$ and $F_{s}(\nu) > 0$, then $d_{s}(\nu,\cM) = \lim_{i \to \infty} d_{s}(\nu_{i},\cM)$.
\end{enumerate}
\end{proposition}

The ideas behind this next theorem originate in \cite[Theorem 2.6]{preiss1987geometry}, but our presentation is a combination of both \cite[Theorem 2.6]{preiss1987geometry} and \cite[Theorem 2.1]{kenig2009boundary}.

\begin{theorem} \label{t:kpt} 
 Suppose $\cF$ is a closed $d$-cone with compact basis, $\mu$ is a Radon measure, and $r_{0} > 0$. 

(1) If there exists $\widetilde{\nu} \in \Tan(\mu,a) \cap \cF$, $0 < \epsilon < 1$, and $\nu \in \Tan(\mu,a)$ so that $0 < \epsilon < d_{r_{0}}(\nu,\cF)$, then there exists $\nu_{\epsilon} \in \Tan(\mu,a)$ satisfying 
$$
\begin{cases}
d_{r_{0}}(\nu_{\epsilon},\cF) = \epsilon \\
d_{r}(\nu_{\epsilon},\cF) \le \epsilon & r > r_{0}.
\end{cases}
$$

(2) Suppose $\cM$ is a $d$-cone with closed basis and the property
\begin{equation} \tag{P} \label{p}
\begin{cases}
\exists \epsilon_{0} > 0 \text{ such that } \forall ~ \epsilon \in (0, \epsilon_{0}) ~ \text{ there exists no } \nu \in \cM  \\
\text{ satisfying } d_{r}(\nu,\cF) \le \epsilon ~ \forall r \ge r_{0} > 0 \text{ and } d_{r_{0}}(\nu,\cF) = \epsilon.
\end{cases}
\end{equation}
Whenever $a \in \rn$ is so that
\begin{equation*} 
\Tan(\mu,a) \subset \cM \text{ and } \Tan(\mu,a) \cap \cF \neq \emptyset,
\end{equation*}
then $\Tan(\mu,a) \subset \cF$.
\end{theorem}

\subsection{Remarkable $d$-cones} \label{s:dcones}
Several specific examples of $d$-cones will play an important role in this article. We introduce here the space of flat $m$-dimensional measures in $\R^{n}$,
$$
\cM_{n,m} = \{ c \cH^{m} \restr V : V \in G(n,m) \text{ and } 0 < c < \infty \},
$$
where $G(n,m)$ is the space of $m$-dimensional planes in $\R^{n}$. We denote the space of flat measures in $\R^{n}$,
$$
\cM_{n} = \bigcup_{m=0}^{n} \cM_{n,m}.
$$
We also consider the space of uniform measures on $\R^{n}$, 
$$
\cU(\R^{n}) = \{ \nu : 0 \in \spt \nu ~ \text{ and } ~ \nu(B(x,r)) = \nu(B(y,r)) \quad \forall x, y \in \spt \nu, \quad \forall r > 0 \},
$$
and the space of $m$-uniform measures 
\begin{align*}
\cU^{m}(\R^{n}) = \{ \nu \in \cU(\R^{n}) : \exists c > 0 \text{ so that } \nu(B(x,r)) = c r^{m} \, ~ \forall x \in \spt \nu, \, ~  \forall r > 0 \}.
\end{align*}
The next lemma is a remark in \cite[Section 3.7(2)]{preiss1987geometry}.
\begin{lemma} \label{l:compactbases}
The following $d$-cones have compact basis: $\cM_{n}$, $\cM_{n,m}$, $\cU^{m}(\R^{n})$, and $\cU(\R^{n})$. 
\end{lemma}

\subsubsection{Uniform measures}
In this section we recall some information about uniform measures. While all the ideas originate in \cite{preiss1987geometry},\footnote{See, for instance \cite[Proposition 2.11]{preiss1987geometry}} our presentation is heavily influenced by \cite[Section 6]{de2008rectifiable}.

\begin{lemma}{\cite[Lemma 3.9]{preiss1987geometry}} \label{l:unifimpliesflattan} If $\nu$ is a uniform measure on $\R^{n}$, then $\cM_{n} \cap \Tan(\nu,x) \neq \emptyset$ for $\nu$ almost every $x \in \R^{n}$.
\end{lemma}

\begin{proposition}{\cite[Proposition 6.16]{de2008rectifiable}}\label{p:c6.16}
If $\nu$ is $m$-uniform then there exists some $m$-uniform $\lambda$ so that, for any sequence $\{r_{i}\}$ with $r_{i} \to \infty$, 
$$
\lim_{i \to \infty} r_{i}^{-m} T_{0,r_{i}}[\nu] = \lambda.
$$
\end{proposition}

\begin{definition}
For $\nu \in \cU^{m}(\R^{n})$ we define $\Tan_{\infty}(\nu) = \{\lambda\}$ where $\lambda$ is the measure from Proposition \ref{p:c6.16}. We call $\lambda$ the tangent at infinity. Moreover, we say that $\nu$ is flat at infinity if $\lambda \in \cM_{n,m}$.
\end{definition}

\begin{proposition}{\cite[Propositions 6.18, 6.19]{de2008rectifiable}} \label{p:c6.18} There exists a constant $\epsilon_{0} = \epsilon(m,n)$ so that if $\nu \in \cU^{m}(\R^{n})$, $\{\lambda\} = \Tan_{\infty}(\nu)$, and
$$
d_{1}(\lambda,\cM_{n,m}) \le \epsilon_{0},
$$
then $\lambda \in \cM_{n,m}$. Moreover, in this case $\nu = \lambda$.
\end{proposition}

We now show that when $\cF = \cM_{n,m}$ and $\cM = \cU^{m}(\R^{n})$ property \eqref{p} holds.

\begin{lemma} \label{l:hasP}
There exists $\epsilon_{0} = \epsilon(m,n) > 0$ so that for all $\epsilon \in (0, \epsilon_{0}]$ there exists no $\mu \in \cU^{m}(\R^{n})$ satisfying
\begin{equation} \label{e:notp}
\begin{cases}
d_{r}(\mu, \cM_{n,m}) \le \epsilon & \forall r \ge 1 \\
d_{1}(\mu,\cM_{n,m}) = \epsilon.
\end{cases}
\end{equation}
\end{lemma}

\begin{proof}
Let $\epsilon_{0}$ be as in Proposition \ref{p:c6.18}. Suppose $\mu \in \cU^{m}(\R^{n})$ satisfies \eqref{e:notp}. Let $\lambda = \Tan_{\infty}(\mu)$. By Propositions \ref{p:kptrs} and \ref{p:c6.16}, 
$$
d_{1}(\lambda, \cM_{n,m}) = \lim_{j \to \infty} d_{1}( 2^{-jm} T_{0,2^{j}}[\mu], \cM_{n,m}) = \lim_{j \to \infty} d_{2^{j}}(\mu, \cM_{n,m}) \le \epsilon.
$$
So, Proposition \ref{p:c6.18} implies $\lambda, \mu \in \cM_{n,m}$. This contradicts \eqref{e:notp}.
\end{proof}

\subsubsection{Symmetric measures}
In this section we define the $d$-cone of symmetric measures and review some of their properties. The information in this section is contained within \cite{mattila1995rectifiable}, but is included here in a condensed fashion for the readers' convenience.

	\begin{definition}\cite[Definition 3.4]{mattila1995rectifiable}
		Let $\nu$ be a nonzero locally finite measure over $\mathbb{R}^{n}$. A point $x \in \mathbb{R}^{n}$ is said to be a point of symmetry of $\nu$ if 
			\begin{align*}
				\int_{B(x, r)} \langle z-x, y \rangle \, d\nu(z)=0
			\end{align*}
		for every $y \in \mathbb{R}^{n}$ and every $r>0$. The measure $\nu$  is said to be symmetric if every point in $\spt \nu$ is a point of symmetry. We denote the $d$-cone of all symmetric measures on $\rn$ whose support contains $\{0\}$ by $\cS_{n}$.
	\end{definition}

\begin{lemma}\cite[Lemma 3.5]{mattila1995rectifiable} \label{l:symchar}
		Let $\nu$ be a nonzero locally finite measure over $\mathbb{R}^{n}$, $s>0$, and $x \in \mathbb{R}^{n}$. Then the following three conditions are equivalent.
			\begin{enumerate}
				\item $x$ is a point of symmetry of $\nu$.
				\item There exists an $m \in \{1, \dots, n\}$ so that
					\begin{align*}
						\int_{r\leq |x-z| \leq R} \frac{x-z}{|x-z|^{m+1}} \,d\nu(z) = 0
					\end{align*}
					for all $0< r< R < \infty$.
				\item For all continuous $g : \mathbb{R} \to \mathbb{R}$ with compact support in $\mathbb{R} \setminus \{0\}$, 
					\begin{align*}
					\int_{\mathbb{R}^{n}}(x-z)g(|x-z|)\,d\nu(z) = 0.
					\end{align*}
					
			\end{enumerate}
	\end{lemma}
The next lemma states $\cS_{n}$ has two properties which are the hypothesis (iv) and conclusion (b) of \cite[Lemma 3.2]{mattila1995rectifiable}. The fact that $\cS_{n}$ satisfies the hypotheses of that lemma is verified across \cite[Lemma 3.6, 3.9, 3.11]{mattila1995rectifiable}.

\begin{lemma} \label{l:symprops}
$\cS_{n}$ has the following properties
\begin{enumerate}
\item If $\nu \in \cS_{n}$, then $\Tan[\nu] \cap \cM_{n} \neq \emptyset$.
\item There is $\epsilon_{0} > 0$ such that whenever $\nu \in \cS_{n}$ satisfies
$$
\limsup_{r \to \infty} d_{r} \left( \nu, \cM_{n,m} \right) < \epsilon_{0}
$$
for some $m = 0, 1, \dots, n$, then the linear span of $\spt \nu$ has dimension at most $m$.
\item Suppose $d = \dim V$, $V = \Span \spt \nu$, and $\nu \in \cS_{n}$. Then either there exists some $c > 0$ so that $\nu = c \cH^{d} \restr V$ or else $\Tan[\nu] \cap \cup_{i=1}^{d-1} \cM_{n,i} \neq \emptyset$.
\end{enumerate}
\end{lemma}

\subsection{Square functions and rectifiability}

Given a radon measure $\mu$ on $\R^{n}$, two common tools to detect the local $m$-dimensional flatness and rectifiability of sets and measures are the (homogeneous) $\alpha$- and $\beta$- numbers, defined respectively by
\begin{equation} \label{e:alpha}
\alpha_{\mu}^{m}(x,r) = r^{-(m+1)} \inf_{\sigma \in \cM_{n,m}} F_{B(x,r)}(\mu, \sigma),
\end{equation}
where $F$ is as in \eqref{e:F}, and
\begin{equation} \label{e:beta}
\beta^{m}_{\mu,2}(x,r)^{2} = \inf_{L} \frac{1}{r^{m}} \int_{B(x,r)} \left( \frac{\dist(y,L)}{r} \right)^{2} d \mu(y),
\end{equation}
where the infimum is taken over all affine $m$-dimensional planes in $\R^{n}$. It will be convenient to define
\begin{equation} \label{e:G}
\beta_{\mu,2}^{m}(x,r;L)^{2} =  \frac{1}{r^{m}} \int_{B(x,r)} \left( \frac{\dist(y,L)}{r} \right)^{2} d \mu(y).
\end{equation}

Our techniques will depend upon studying a centered-version of the $\alpha$-numbers defined by
\begin{equation} \label{e:oalpha}
\oalpha_{\mu}^{m}(x,r) = r^{-(m+1)} \inf_{\substack{\sigma \in \cM_{n,m} \\ x \in \spt \sigma}} F_{B(x,r)}(\mu,\sigma).
\end{equation}

To aid in developing properties for centered $\alpha$-numbers we consider the the centered $\beta$-numbers defined by
\begin{equation} \label{e:obeta}
\obeta^{m}_{\mu,2}(x,r)^{2} = \inf_{L \ni x} \beta_{\mu,2}^{m}(x,r;L)^{2}.
\end{equation}

For the remainder of the paper, we will suppress the dependence on $m$ and $2$, and merely write $\alpha_{\mu}, \beta_{\mu}, \oalpha_{\mu}$, and $\obeta_{\mu}$ in place of $\alpha_{\mu}^{m}, \beta_{\mu,2}^{m}, \oalpha_{\mu}^{m}$, and $\obeta_{\mu,2}^{m}$.

The following theorem gathers in one place much of the literature on characterizing rectifiability of Radon measures in terms of $\alpha_{\mu}$ and $\beta_{\mu}$. Here, we only include the final results in $\R^{n}$ and none of the many great works that helped develop the necessary theory for these end products. See the references within the results we mention for a more complete history.

\begin{theorem} \label{t:characterizations}
    If $\mu$ is a Radon measure on $\R^{n}$, then the following are equivalent:
    \begin{enumerate}
        \item $\mu$ is $m$-rectifiable.\\
        \item For $\mu$- a.e. $x$, any of the following hold:
        \begin{enumerate}
            \item $\displaystyle\int_{0}^{1} \left(\frac{\inf_{\sigma \in \cM_{n,m}} F_{B(x,r)}(\mu,\sigma)}{r \mu(B(x,3r))} \right)^{2} \frac{dr}{r} < \infty$ and \\
            $\displaystyle\int_{0}^{1} \left( \inf_{L} \frac{1}{\mu(B(x,3r))} \int_{B(x,r)} \left(\frac{\dist(y,L)}{r} \right)^{2} d \mu(y) \right) \frac{dr}{r} < \infty$, \\
            \item $\limsup_{r \to 0} \frac{\mu(B(x,2r))}{\mu(B(x,r))} < \infty$ and \\
            $\int_{0}^{1} \left( \frac{\inf_{\sigma \in \cM_{n,m}} F_{B(x,r)}(\mu,\sigma)}{r \mu(B(x,r))} \right)^{2} \frac{dr}{r} < \infty$
            \item $0 < \theta^{m}_{*}(\mu,x) \le \theta^{m,*}(\mu,x) < \infty$ and $\int_{0}^{1} \alpha_{\mu}^{m}(x,r)^{2} \frac{dr}{r} < \infty$,
            \item $0 < \theta^{m,*}(\mu,x)$, $\theta^{m}_{*}(\mu,x) < \infty$, and 
            $
            \int_{0}^{1} \beta_{\mu}^{m}(x,r)^{2} \frac{dr}{r} < \infty,
            $
            \item $0 < \theta^{m}_{*}(\mu,x) \le \theta^{m,*}(\mu,x) < \infty$ and $\int_{0}^{1} \obeta_{\mu}^{m}(x,r)^{2} \frac{dr}{r} < \infty$.
        \end{enumerate}
    \end{enumerate}
\end{theorem}
The equivalence with (1) and (2a) comes from \cite{dabrowski2021sufficient}. In combination with \cite{dkabrowski2020necessary}, a characterization in terms of transport numbers is also given for generic Radon measures in $\R^{n}$. (1) and (2b) comes from \cite{azzam2020characterization}. These first two results are the strongest results in this list due to the lack of assumptions on the density of $\mu$, but all integrals in these statements are analogs of the $
\alpha$ and $\beta$-numbers, re-scaled to deal with the lack of assumptions on the density. The equivalence of (1) and (2c) is because of the equivalence of (2b) and (2c) given the extra density assumptions. That (1) is equivalent to (2d) is \cite{edelen2016quantitative}. Finally the equivalence of (1) and (2e) is due to the equivalence of (2d) and (2e) under the additional density assumptions \cite{kolasinski2016estimating}. In Theorem \ref{t:centeredchar}, we use many of the equivalences in Theorem \ref{t:characterizations} to prove new characterizations of rectifiable measures in terms of $\oalpha, \alpha$.

The following proposition is verified as in the proof of \cite[Lemma 3.4]{dabrowski2021sufficient}. We note the only difference between Proposition \ref{p:centeredalpha} and \cite[Lemma 3.4]{dabrowski2021sufficient} is that we used the centered $\beta$-numbers in place of the $\beta$-numbers in order to recover an estimate on the centered $\alpha$-numbers. Since the crux of the original proof is that the angle between $L_{\alpha}$ and $L_{\beta}$ is bounded by $\beta_{\mu}(x_{0},r_{2}) + \alpha_{\mu}(x_{0},r_{2})$ it is not surprising that one can recover the same bound with $\obeta_{\mu}$ in place of $\beta_{\mu}$ since $\beta_{\mu} \le \obeta_{\mu}$.

\begin{proposition} \label{p:centeredalpha}
    Suppose that $\mu$ is a Radon measure on $\R^{n}$, $x_{0} \in \spt \mu$, and that $0 < r_{1} \le \frac{9}{10} r_{2} < \infty$. Further suppose $\mu(B(x_{0},r_{1})) \approx \mu(B(x_{0},r_{2})) \approx r_{1}^{m} \approx r_{2}^{m}$. Let $L_{\beta}$ be the plane containing $x_{0}$ so that  $\beta_{\mu}(x_{0},r_{2};L_{\beta}) = \obeta_{\mu}(x_{0},r_{2})$. Similarly let $\sigma = c_{\alpha} \cH^{n} \restr L_{\alpha}$ be the flat measure minimizing $F_{B(x_{0},r_{2})}(\mu,\sigma)$. Suppose further that $L_{\alpha} \cap B(x_{0}, \frac{9}{10}r_{1}) \neq \emptyset$. Then,
    $$
        \frac{F_{B(x_{0},r_{1})}(\mu, c_{\alpha} \cH^{m} \restr L_{\beta})}{r_{1}^{m+1}} \lesssim \obeta_{\mu}(x_{0},r_{2}) + \alpha_{\mu}(x_{0},r_{2}).
    $$

    In particular,
    $$ 
        \oalpha_{\mu}(x_{0},r_{1}) \lesssim \obeta_{\mu}(x_{0},r_{2}) + \alpha_{\mu}(x_{0},r_{2})
    $$
\end{proposition}

The following proposition is a slight modification of \cite[Lemma 7.1]{jaye2020small}. 
\begin{lemma}[Small Annuli] \label{l:smallanulli}
    Suppose that $\| \cdot \|$ is a norm on $\R^{n}$, $\mu$ a Radon measure and $x \in \R^{n}$ is so that  $\theta^{m}_{\mu}(x, r) < \infty$. Suppose $\Lambda_{0}$ is so that
    $$
        B_{\| \cdot \|}(x,r) \subset B(x,\Lambda_{0}r), 
    $$
    where $B_{\|\cdot\|}(x,r)=\{y\in \mathbb{R}^n: \|x-y\|<r\}$. Then, for all $0 < \delta < 1/2$ and $r>0$, 
    \begin{equation} \label{e:smallannuliconclusion}
        \frac{ \mu(B_{\| \cdot \|}(x, r(1+\delta)) \setminus B_{\| \cdot \|}(x, r))}{r^{m}} \lesssim_{\Lambda_{0},m} \frac{\oalpha_{\mu}(B(x,2\Lambda_{0} r))}{  \delta} + \theta_{\mu}^m(x,2 \Lambda_{0} r)\delta, 
    \end{equation}
    where $\theta_{\mu}^m(x,r)=\frac{\mu(B(x,r))}{r^m}$.
\end{lemma}

\begin{proof}
Let $\sigma \in \cM_{n,m}$ be a minimizing measure for $\oalpha_{\mu}^m(x,2\Lambda_{0}r)$, which exists since the space of Radon measures is weak-$*$ compact. Then $\sigma = c \cH^{m} \restr V$ some $V\in G(n,m)$. We first claim that $c\leq \frac{2^{m+2}}{\omega_m}\frac{\mu(B(x,\Lambda_{0}r))}{r^m}$.

    Let $\Phi_{B(x,2\Lambda_{0}r)}(y):=(2\Lambda_{0} r-|y-x|)_+$. Since $\Phi_{B(x,2\Lambda_{0}r)}(y)\geq \Lambda_{0} r$ on $B(x, \Lambda_{0}r)$, we have 
    \begin{align*}
        c\left(\Lambda_{0}r\right)^{m+1}\omega_m
        &= c\Lambda_{0}r\mathcal{H}^m(B(x,\Lambda_{0}r)\cap V)\\
        &\leq \int_{B(x,\Lambda_{0}r)}\Phi_{B(x,2\Lambda_{0}r)}(y)d(c\mathcal{H}^m\restr V)\\
        &\leq \int_{B(x,2\Lambda_{0}r)}\Phi_{B(x,2\Lambda_{0}r)}(y)d(c\mathcal{H}^m\restr V)\\
        &=F_{B(x,2\Lambda_{0}r)}(\sigma).
    \end{align*}
    Then from the triangle inequality, and the infimizing property of $\sigma$,
    \begin{align*}
        c\left(\Lambda_{0}r\right)^{m+1}\omega_m&\leq F_{B(x,2\Lambda_{0}r)}(\sigma, \mu)+F_{B(x,2\Lambda_{0}r)}(\mu)\\
        &\leq \lim_{\tau \to 0} F_{B(x,2\Lambda_{0}r)}(\tau \mathcal{H}^m\restr V, \mu)+F_{B(x,2\Lambda_{0}r)}(\mu)\\
        &\leq (4\Lambda_{0}r)\mu(B(x,2\Lambda_{0}r)). 
    \end{align*}
    Thus,
    \begin{equation}\label{e: c in small annuli}
        c\leq \frac{2^{m+1}}{\omega_m}\frac{\mu(B(x,2\Lambda_{0}r))}{(2\Lambda_{0}r)^m} = \frac{2^{m+1}}{\omega_{m}} \theta^{m}_{\mu}(x, 2\Lambda_{0}r).
    \end{equation}
     Fix a cut-off function $\psi_{r,\delta}$ satisfying
    \begin{equation*} \begin{cases} 
    \chi_{B_{\|\cdot\|}(x,r(1+\delta))} - \chi_{B_{\|\cdot\|}(x,r)} \leq \psi_{r,\delta} \leq \chi_{B_{\|\cdot\|}(x,r(1+2\delta))}\\
    \spt \psi_{r,\delta} \subset B_{\|\cdot\|}(x,r(1+2\delta)) \setminus B_{\|\cdot\|}(x,r(1-\delta)) \\
    \lip \psi_{r,\delta} = (r \delta)^{-1}.
    \end{cases}
    \end{equation*}

    Note, $\spt \psi_{r,\delta} \subset B(x,2\Lambda_{0}r)$. For an $m$-plane $P$ through the origin, define
    $$
        C_{P,\| \cdot \|} = \cH^{m} \left(P \cap B_{\| \cdot \|}(0,1) \right)
    $$
    and observe
    $$
        \begin{cases}
            \cH^{m} \left(P \cap B_{\| \cdot \|}(0,s) \right) = C_{P,\| \cdot\|} s^{m} \\
            C_{P,\|\cdot\|} \le \cH^{m} \left( P \cap B(0,\Lambda_{0}) \right) = \omega_{m} \Lambda_{0}^{m} .
        \end{cases}
    $$
        Then, recalling $\sigma = c \cH^{m} \restr V$ is chosen to achieve the minimium in $\oalpha_{\mu}(x,2 \Lambda_{0}r)$,
    \begin{align*}
        \mu \bigg( B_{\|\cdot\|}(x, r(1+\delta)) & \setminus B_{\|\cdot\|}(x,r )\bigg) 
        \le \int \psi_{r,\delta}(y) d \mu(y) \\
        & = \int \psi_{r,\delta}(y) d (\mu - \sigma)(y) + \int \psi_{r,\delta}(y) d \sigma(y) \\
        & \lesssim \frac{r^{m+1}}{r \delta} \oalpha_{\mu}^m(B(x,2\Lambda_{0}r)) + \int \psi_{r,\delta}(y) d (c\mathcal{H}^m\restr V(y)) \\
        & \le \frac{r^{m+1}}{r \delta} \oalpha_{\mu}^m(B(x,2\Lambda_{0}r))\\
        &\hspace{1cm}+ c\mathcal{H}^m\left( V\cap \left(B_{\|\cdot\|}(x,r(1+2\delta))\setminus B_{\|\cdot\|}(x,r(1-\delta))\right)\right)\\
        & =  \frac{r^{m}}{ \delta} \oalpha_{\mu}^m(B(x,2\Lambda_{0}r))+ c C_{P, \| \cdot \|} \left( (r(1+2 \delta))^{m} - (r(1-\delta))^{m} \right) \\
        &\lesssim \frac{r^m}{\delta}\oalpha_{\mu}^m(B(x,2\Lambda_{0}r))+c\left(\sup_{P}C_{\|\cdot\|,P}\right)r^m\delta \\
        &\le C(m, \Lambda_{0}) \left( \frac{r^m}{\delta}\oalpha_{\mu}^m(B(x,2\Lambda_{0}r))+c r^{m} \delta \right), 
    \end{align*}
    where we take $P= V-x$. 
    Using (\ref{e: c in small annuli}) we can bound $c r^{m}$ by $C(m) \theta^{m}(x,2\Lambda_{0}r)$, hence verifying \eqref{e:smallannuliconclusion}.
\end{proof}

\section{$\Lambda$-Tangents} \label{s:ltan}
Consider a mapping $\Lambda : \R^{n} \to GL(n,\R)$ and the ellipse
$$
B_{\Lambda}(a,r) = a +  \Lambda(a)B(0,r),
$$
whose eccentricity depends on the point $a$. For a Radon measure $\mu$ we define the $m$-dimensional upper and lower $\Lambda$-densities of $\mu$ by
\begin{equation} \label{e:gdensity}
\theta^{m,*}_{\Lambda}(\mu,a) = \limsup_{r \downarrow 0} \frac{ \mu \left(B_{\Lambda}(a,r)\right)}{r^{m}}  \text{ and } \theta^{m}_{\Lambda,*}(\mu,a) = \liminf_{r \downarrow 0}  \frac{ \mu \left( B_{\Lambda}(a,r) \right)}{r^{m}}.
\end{equation}
In the case these two quantities agree, their common value is the $m$-dimensional $\Lambda$-density, denoted $\theta^{m}_{\Lambda}(\mu,a)$. When $\Lambda = \Id$, we suppress the dependence on $\Lambda$ and recover the usual densities with respect to Euclidean balls $\theta^{m}(\mu, a)$, $\theta^{m,*}(\mu,a)$, and $\theta^{m}_{*}(\mu,a)$.

From a PDE perspective, one would assume the mapping $\Lambda$ should be uniformly elliptic. At the level of rectifiability, geometry is more flexible and allows us to only require that for each $a$ the matrix $\Lambda(a)$ is invertible. Invertibility is necessary since the geometry of a measure near $a$ can be lost if $\Lambda(a)$ collapses $\rn$ into a lower dimensional space.

We next define the rescaling
$$
T^{\Lambda}_{a,r}(y) = \Lambda(a)^{-1} \left( \frac{y-a}{r} \right),
$$
and denote image measures under this rescaling by $T_{a,r}^{\Lambda}[\mu]$. That is,
$$
T^{\Lambda}_{a,r}[\mu](E) = \mu(a + r \Lambda(a) E).
$$
In particular
$$
T^{\Lambda}_{a,r}[\mu](B_{1})) = T_{a,r}[\mu](B_{\Lambda}(0,1)) = \mu\left(B_{\Lambda}(a,r)\right).
$$

\begin{definition}[$\Lambda$-tangents]
If $\mu$ is a Radon measure, we define
\begin{equation} \label{e:lambdatan}
\Tan_{\Lambda}(\mu,a) = \left\{\nu  \text{ Radon s.t. } \nu = \lim_{i} c_{i} T^{\Lambda}_{a,r_{i}}[\mu]  : c_{i} > 0, \, r_{i} \downarrow 0, \, \nu \neq 0 \right\}.
\end{equation}
\end{definition}

\begin{remark}
Given $\nu \in \Tan_{\Lambda}(\mu,a)$ with $c_{i} T^{\Lambda}_{a,r_{i}}[\mu] \xrightharpoonup{*} \nu$ it is easy to check that $c T_{0,r}[\nu] = \lim_{i} c c_{i} T^{\Lambda}_{a, rr_{i}}[\mu]$ for any $c , r > 0$. In particular $\Tan_{\Lambda}(\mu,a)$ is a $d$-cone.
\end{remark}

We will prove that $\Lambda$-tangents have a property that implies tangents to $\Lambda$-tangents are $\Lambda$-tangents.

\begin{theorem} \label{t:tan2ltan}
Let $\mu$ be a Radon measure on $\rn$ and $\Lambda : \rn \to GL(n,\R)$. Then at $\mu$ almost all $a \in \rn$ every $\nu \in \Tan_{\Lambda}$ has the following two properties:
\begin{enumerate}
\item $T_{x,r}[\nu] \in \Tan_{\Lambda}(\mu,a)$ for all $x \in \spt \nu, r > 0$.
\item $\Tan(\nu,x) \subset \Tan_{\Lambda}(\mu,a)$ for all $x \in \spt \nu$.
\end{enumerate}
\end{theorem}

One can directly prove Theorem \ref{t:tan2ltan} by making several modifications to the original proof of Theorem \ref{t:tan2tan}. Some of these modifications are showcased in the proof of Theorem \ref{t:uniffromdensity}. Instead, we will make use of the following lemma, where $\Lambda(a)_{\sharp} \nu$ will be used to denote the image measure $T[\nu]$ when $T(x) = \Lambda(a)x$.

\begin{lemma} \label{l:iso}
Let $\mu$ be a Radon measure on $\rn$ and $\Lambda : \rn \to GL(n,\R)$. For a Radon measure $\nu$ the following are equivalent:
\begin{enumerate}
\item $\nu \in \Tan_{\Lambda}(\mu,a)$
\item $\Lambda(a)_{\sharp} \nu \in \Tan(\mu,a)$
\item $\nu \in \Tan( (\Lambda(a)^{-1})_{\sharp} \mu,  \Lambda(a)^{-1} a)$
\end{enumerate}
\end{lemma}
Lemma \ref{l:iso} provides geometric intuition about $\Lambda$-tangents. The equivalence of (1) and (2) says that any $\Lambda$ tangent could equivalently be generated by applying a fixed linear transformation to a Euclidean tangent measure. The equivalence of (1) and (3) says $\Lambda$-tangents are a Euclidean tangent of a linear transformation of the original measure.  Each perspective serves its own purpose:

The equivalence of (1) and (2) says that $\Lambda(\cdot)_{\sharp}$ is an isomorphism between $\Tan_{\Lambda}(\mu, \cdot)$ and $\Tan(\mu, \cdot)$. Therefore, any statement about $\Tan(\mu, \cdot)$ that holds almost everywhere has an equivalent statement for $\Tan_{\Lambda}(\mu, \cdot)$ that holds almost everywhere, after unwinding what effect the isomorphism $\Lambda(\cdot)_{\sharp}$ has. This will be used to prove Theorem \ref{t:tan2ltan}.

The equivalence of (1) and (3) states that the $d$-cone $\Tan((\Lambda(a)^{-1})_{\sharp}\mu, \Lambda(a)^{-1} a)$ and  the $d$-cone $\Tan_{\Lambda}(\mu,a)$  are the same. Hence, properties about tangent measures derived from the fact that $\Tan(\mu,\cdot)$ forms a $d$-cone are also valid for $\Lambda$-tangents. In this case, no unwinding of the effects of an isomorphism is required.  This will be used to prove Theorem \ref{t:jams}.

\begin{proof}[Proof of Lemma \ref{l:iso}]
To prove the equivalence of (1) and (2) observe
$$
\Lambda(a) T^{\Lambda}_{a,r}(y) = T_{a,r}(y).
$$
Therefore, by Proposition \ref{p:continuity}, $\nu = \lim_{i} c_{i} T_{a,r_{i}}^{\Lambda}[\mu] \in \Tan_{\Lambda}(\mu,a)$ if and only if 
$$
\Lambda(a)_{\sharp} \nu =  \lim_{i}   c_{i} T_{a,r_{i}}[\mu] =  \lim_{i} c_{i} T_{a,r_{i}}[\mu] \in \Tan(\mu,a).
$$
To prove the equivalence of (1) and (3) observe
\begin{equation} \label{e:doublecomp}
T^{\Lambda}_{a,r}(y) = T_{\Lambda(a)^{-1}  a, r} \left( \Lambda(a)^{-1} y \right).
\end{equation}
So Proposition \ref{p:continuity} guarantees $\nu = \lim_{i} c_{i} T_{a,r_{i}}^{\Lambda}[\mu]$ if and only if
$$  
\nu = \lim_{i} c_{i} T_{\Lambda(a)^{-1}a,r_{i}}[(\Lambda(a)^{-1})_{\sharp} \mu] \in \Tan(( \Lambda(a)^{-1})_{\sharp} \mu, \Lambda(a)^{-1} a).
$$
\end{proof}

\begin{proof}[Proof of Theorem \ref{t:tan2ltan}]
Let $A \subset \rn$ be the set of full measure satisfying the conclusion of Theorem \ref{t:tan2tan}. Fix some $a \in A$ and $\nu \in \Tan_{\Lambda}(\mu,a)$ and $x \in \spt \nu$. By Lemma \ref{l:iso}, $\nu_{0} \defeq \Lambda(a)_{\sharp} \nu \in \Tan(\mu,a)$. On the other hand, $x \in \spt \nu \implies \Lambda(a) x \in \spt  \nu_0$. Since $a \in A$, it follows $\nu_{1} \defeq T_{\Lambda(a)x,r} [\nu_0] \in \Tan(\mu,a)$ and a final application of Lemma \ref{l:iso} implies $(\Lambda(a)^{-1})_{\sharp} \nu_{1} \in \Tan_{\Lambda}(\mu,a)$. To confirm $T_{x,r}[\nu] \in \Tan_{\Lambda}(\mu,a)$ we check $(\Lambda(a)^{-1})_{\sharp} \nu_{1} = T_{x,r}[\nu]$.  Indeed, from the identity $T_{x,r}(y) = \Lambda(a)^{-1} \circ T_{\Lambda(a) x,r} \circ \Lambda(a)(y)$, it follows
$$
(\Lambda(a)^{-1})_{\sharp} \nu_{1} = \left( \Lambda(a)^{-1} \circ T_{\Lambda(a) x, r} \circ \Lambda(a) \right)_{\sharp} \nu = T_{x,r}[\nu].
$$
\end{proof}

We now state and prove the analog of Theorem \ref{t:kpt}(1) for $\Lambda$-tangents.

\begin{theorem} \label{t:jams} 
Fix $\Lambda : \R^{n} \to GL(n,\R)$ and $r_{0} > 0$. Suppose $\cF$ is a closed $d$-cone with compact basis and $\mu$ is a Radon measure. 

If there exists $\widetilde{\nu} \in \Tan_{\Lambda}(\mu,a) \cap \cF$, $0 < \epsilon < 1$, and $\nu \in \Tan_{\Lambda}(\mu,a)$ so that $0 < \epsilon < d_{r_{0}}(\nu,\cF)$, then there exists $\nu_{\epsilon} \in \Tan_{\Lambda}(\mu,a)$ satisfying 
\begin{equation} \label{e:deps}
\begin{cases}
d_{r_{0}}(\nu_{\epsilon},\cF) = \epsilon \\
d_{r}(\nu_{\epsilon},\cF) \le \epsilon & r > r_{0}.
\end{cases}
\end{equation}

\end{theorem}

\begin{proof}
Let $\nu, \tilde{\nu} \in \Tan_{\Lambda}(\mu,a)$. It follows $\nu, \tilde{\nu} \in \Tan( (\Lambda(a)^{-1})_{\sharp} \mu, \Lambda(a)^{-1}a)$ due to Lemma \ref{l:iso} (1) and (3). Therefore, Theorem \ref{t:kpt}(1) implies there exists $\nu_{\epsilon} \in \Tan( (\Lambda(a)^{-1})_{\sharp} \mu, \Lambda(a)^{-1}a)$ satisfying \eqref{e:deps}. By Lemma \ref{l:iso}(1) and (3), $\nu_{\epsilon} \in \Tan_{\Lambda}(\mu,a)$ proving Theorem \ref{t:jams}.
\end{proof}

The next corollary is a slight extension of Theorem \ref{t:kpt}(2) in the setting of $\Lambda$-tangents. It is a succinct summary of how \cite{mattila1995rectifiable} proves that symmetric tangents a.e. implies flat tangents a.e.

\begin{corollary} \label{c:jams}
Suppose $\cF= \cup_{i=1}^{\infty} \cF_{i}$, each $\cF_{i}$ is a $d$-cone with compact basis and there exists $\epsilon_{i} >0$ and $\cM$  a $d$-cone with closed basis so that for each $i$: $\cF_{i} \subset \cM$ and  
\begin{equation} \label{e:pi} \tag{$P{_i}$}
\begin{cases}
\exists \epsilon_{i} > 0, \, R_{i} > 0 \text{ such that } \forall \, \epsilon \in (0, \epsilon_{i}) \text{ there exists no } \nu \in \cM \setminus \cup_{j=1}^{i-1} \cF_{j} \\
\text{ satisfying } d_{r}(\nu, \cF_{i}) \le \epsilon ~ \forall \, r \ge R_i > 0 ~ \text{ and } d_{R_i}(\nu, \cF_{i}) = \epsilon. 
\end{cases}
\end{equation}
If $a \in \rn$ is so that $\Tan_{\Lambda}(\mu,a) \subset \cM$  and $\Tan_{\Lambda}(\mu,a) \cap \cF \neq \emptyset$ then $\Tan_{\Lambda}(\mu,a) \subset \cF$.
\end{corollary}

\begin{proof}
Suppose $\Tan_{\Lambda}(\mu,a) \subset \cM$ and $\Tan_{\Lambda}(\mu,a) \cap \cF \neq \emptyset$. Let $i$ be the smallest integer so that $\Tan_{\Lambda}(\mu,a) \cap \cF_{i} \neq \emptyset$. Then in particular, $\Tan_{\Lambda}(\mu,a) \cap \cup_{j < i} \cF_{i} = \emptyset$.

We will show that $\Tan_{\Lambda}(\mu,a) \subset \cF_{i}$. Indeed, suppose not. Then by Theorem \ref{t:jams}(1) applied to $\cF_{i}$ and $\Tan_{\Lambda}(\mu,a)$,  for $0 < \epsilon < \min \{1, \epsilon_{i}\}$, there exists $\nu_{\epsilon}$ so that
$$
\begin{cases}
d_{r_{0}}(\nu_{\epsilon},\cF_i) = \epsilon \\
d_{r}(\nu_{\epsilon},\cF_i) \le \epsilon & r > r_{0}.
\end{cases}
$$
but this contradicts Property \eqref{e:pi}
\end{proof}

The next lemma provides information about the measure of balls centered at the origin for $\Lambda$-tangents. This is the crucial starting point for Theorem \ref{t:uniffromdensity} as well as for showing the equivalence of studying the density question with arbitrary weights $c_{i}$ or in the special case $c_{i} = cr_{i}^{-m}$.

\begin{lemma}\label{l:density2tan}
Suppose $\nu \in \Tan_{\Lambda}(\mu,a)$ and $c_{i} T^{\Lambda}_{a,r_{i}}[\mu] \xrightharpoonup{*} \nu$. Then 
\begin{equation} \label{e:dquant1}
1 \le \frac{ \limsup_{i} c_{i} r_{i}^{m}}{\liminf_{i} c_{i} r_{i}^{m}} \le \frac{ \theta^{m,*}_{\Lambda}(\mu,a)}{\theta^{m}_{\Lambda,*}(\mu,a)},
\end{equation}
and for all $R > 0$,
\begin{equation} \label{e:tanmeas1}
\limsup_{i} c_{i} r_{i}^{m} \theta^{m}_{\Lambda,*}(\mu,a) \le \frac{\nu \left(B_R \right)}{R^{m}} \le \liminf_{i} c_{i} r_{i}^{m} \theta_{\Lambda}^{m,*}(\mu,a).
\end{equation}
In particular, if $0 < \theta^{m}_{\Lambda,*}(\mu,a) \le \theta^{m,*}_{\Lambda}(\mu,a) < \infty$,
\begin{equation} \label{e:dqual1}
0 < \liminf_{i \to \infty} c_{i} r_{i}^{m}  \le \limsup_{i \to \infty} c_{i} r_{i}^{m} < \infty.
\end{equation}
\end{lemma}

The following is a quick corollary of Lemma \ref{l:density2tan}.
\begin{corollary} \label{c:density}
If $\nu \in \Tan_{\Lambda}(\mu,a)$ and $0 < \theta^{m}_{\Lambda,*}(\mu,a) \le \theta^{m,*}_{\Lambda}(\mu,a) < \infty$, then for all $C \ge 1$
$$
\frac{\nu(B(0,CR))}{\nu(B(0,R))} \le \frac{\theta_{\Lambda}^{m,*}(\mu,a)}{\theta^{m}_{\Lambda,*}(\mu,a)} C^{m}.
$$

If $\theta^{m}_{\Lambda}(\mu,a)$ exists, $\nu \in \Tan_{\Lambda}(\mu,a)$ and $\nu = \lim_{i} c_{i} T^{\Lambda}_{a,r_{i}}[\mu]$, then 
$$
\theta^{m}_{\Lambda}(\mu,a) \lim_{i} c_{i} r_{i}^{m} = \nu(B_{1}).
$$
In particular, $\nu = \lim_{i} \tilde{c_{i}} T^{\Lambda}_{a,r_{i}}[\mu]$ where $\tilde{c_{i}} = \frac{\nu(B_1)}{\theta^{m}_{\Lambda}(\mu,a) r_{i}^{m}}$.

\end{corollary}
\begin{proof}[Proof of Lemma \ref{l:density2tan}]
Note that for any $R > 0$,
$$
\infty > \theta^{m,*}_{\Lambda}(\mu,a) = \limsup_{r \downarrow 0} \frac{ T_{a,r}^{\Lambda}[\mu](B_R)}{(rR)^{m}} \ge \limsup_{i \to \infty} \frac{ T^{\Lambda}_{a,r_{i}}[\mu](B_R)}{(r_{i}R)^{m}}.
$$
Since $\nu$ is a Radon measure, for almost every $R > 0$, $\nu( \partial B(0,R)  ) = 0$. Choosing such $R$,
\begin{align*}
\nonumber \limsup_{i \to \infty} \frac{T^{\Lambda}_{a,r_{i}}[\mu](B_R)}{(r_{i}R)^{m}} &= \limsup_{i \to \infty} \frac{1}{c_{i} (r_{i} R)^{m}} c_{i} T_{a,r_{i}}^{\Lambda}[\mu] \left( B_R \right) \\
 & = \nu \left( B_R \right) \limsup_{i \to \infty} \frac{1}{c_{i} (r_{i}R)^{m}}.
\end{align*}
Since $0 \in \spt \nu$ this implies
\begin{equation}
\label{e:dquant2}
0 < R^{-m} \nu \left(B_R \right) \le \theta^{m,*}_{\Lambda}(\mu,a) \liminf_{i \to \infty} c_{i} r_{i}^{m}.
\end{equation}
Similarly, for any such $R$, it follows
\begin{align*}
\nonumber 0 & < \theta^{m}_{\Lambda,*}(\mu,a) = \liminf_{r \downarrow 0} \frac{ T^{\Lambda}_{a,r}[\mu](B_R)}{(rR)^{m}} \le \liminf_{i \to \infty} \frac{c_{i} T^{\Lambda}_{a,r_{i}}[\mu](B_R)}{c_{i} (r_{i} R)^{m}}  \\
\label{e:dquant3} & = \frac{\nu( B_R)}{\limsup_{i} c_{i} (r_{i} R)^{m}}.
\end{align*}
Since $\nu$ is Radon, this implies
\begin{equation} \label{e:102}
\theta^{m}_{\Lambda,*}(\mu,a) \limsup_{i} c_{i}r_{i}^{m} \le \frac{\nu(B_{R})}{R^{m}} < \infty.
\end{equation}
Combining \eqref{e:dquant2} and \eqref{e:102} confirms \eqref{e:dqual1} and \eqref{e:dquant1}.  In fact, \eqref{e:dquant2} and \eqref{e:102} also verifies \eqref{e:tanmeas1} for all $R$ so that $\nu \left( \partial B(0,R) \right) = 0$. To prove \eqref{e:tanmeas1} for general $R$, note
$$
\left(\frac{s}{R} \right)^{m} \frac{\nu(B_{s})}{s^{m}} \le \frac{\nu(B_{R})}{R^{m}} \le \left( \frac{S}{R} \right)^{m} \frac{\nu(B_{S})}{S^{m}}
$$
Choosing any sequence of $s_{i} \le R \le S_{i}$ so that $\nu(\partial B_{S_{i}}) = 0 = \nu(\partial B_{s_{i}})$ and $s_{i} \uparrow R, S_{i} \downarrow R$ confirms \eqref{e:tanmeas1} for general $R$.
\end{proof}

\begin{lemma} \label{l:ai}
Let $\mu$ be a Radon measure and $A \subset \R^{n}$. If $\mu(A) >0$, for $\mu$ a.e. $a \in A$
\begin{equation} \label{e:density1}
\lim_{r \downarrow 0} \frac{\mu(A \cap B(a,r))}{\mu(B(a,r))} = 1.
\end{equation}
Moreover, for any such $a$, if $\nu \in \Tan_{\Lambda}(\mu,a)$ and $\nu = \lim_{i} c_{i} T^{\Lambda}_{a,r_{i}}[\mu]$, then for any $x \in \spt \nu$, there exists $a_{i} \in A$ with
\begin{equation} \label{e:aix}
\lim_{i \to \infty} \Lambda(a)^{-1} \left( \frac{a_{i} -a}{r_{i}} \right) = x.
\end{equation}
\end{lemma}

\begin{proof}[Proof of Lemma \ref{l:ai}]
By \cite[Theorem 2.9.11]{federer2014geometric}, for any measure $\mu$ and $A \subset \R^{n}$,
$$
\mu \left( A \setminus \{ x : \liminf_{r \downarrow 0} \frac{\mu(A \cap B(x,r))}{\mu(B(x,r))} = 1 \} \right) = 0.
$$ 
Thus \eqref{e:density1} holds for $\mu$ a.e. $a \in \R^{n}$. Now suppose $a \in A$ satisfies \eqref{e:density1} but \eqref{e:aix} fails. Then there exist $\nu \in \Tan_{\Lambda}(\mu,a)$, $x \in \spt \nu$ with $x\neq 0$, a subsequence $\{i_{k}\}$, and $0<\delta <|\Lambda(a)x|$, so that
\begin{equation} \label{e:acomp}
\dist( B(a + r_{i_{k}} \Lambda(a) x, r_{i_{k}} \delta) , A) > 0.
\end{equation}
Without loss of generality we suppose \eqref{e:acomp} holds for the original sequence. Since $\mu$ is Radon, for any sets $E,F$ with $\dist (E,F) > 0$, it follows $\mu(E \cup F) = \mu(E) + \mu(F)$. Therefore, \eqref{e:density1} and \eqref{e:acomp} imply
\begin{equation} \label{e:103}
1 = \lim_{i \to \infty} \frac{\mu \left(B(a ,2 r_{i}|\Lambda(a)x|) \cap A\right)}{\mu(B(a, 2 r_{i}|\Lambda(a) x|))} \le 1 - \liminf_{i \to \infty} \frac{\mu(B ( a + r_{i} \Lambda(a) x, \delta r_{i}))}{\mu(B(a,2r_{i}|\Lambda(a)x|))}.
\end{equation}
We will show \eqref{e:103} is a contradiction by producing a nonzero lower bound on the final term. Indeed, following the convention that $B(x,r)$ and $U(x,r)$ are respectively the closed and open balls around $x$ of radius $r$,
\begin{align}
\nonumber \liminf_{i \to \infty} & \frac{\mu(B ( a + r_{i} \Lambda(a) x, \delta r_{i} ))}{\mu(B(a,2r_{i}|\Lambda(a)x|))} \ge \liminf_{i \to \infty} \frac{ c_{i} T_{a,r_{i}}[\mu] \left( U (\Lambda(a) x, \delta ) \right)}{c_{i} T_{a,r_{i}}[\mu] \left(B(0,2 |\Lambda(a)x|)\right)} \\
\nonumber & = \liminf_{i \to \infty} \frac{ c_{i} T_{a,r_{i}}^{\Lambda}[\mu] \left( \Lambda(a)^{-1} U(\Lambda(a) x, \delta  ) \right)}{c_{i} T_{a,r_{i}}^{\Lambda}[\mu] \left( \Lambda(a)^{-1}B(0, 2|\Lambda(a)x| ) \right)} \\
\label{e:104} & \ge \frac{\nu \left( \Lambda(a)^{-1} U(\Lambda(a) x, \delta  )\right)}{\nu \left( \Lambda(a)^{-1} B(0, 2|\Lambda(a)x|) \right)} > 0.
\end{align}
The reason the final term is positive is that $\Lambda(a)^{-1} U(\Lambda(a) x, \delta )$ is an open neighborhood of $x \in \spt \nu$.  Now \eqref{e:103} and \eqref{e:104} yield a contradiction, confirming \eqref{e:aix}.
\end{proof}

\section{Rectifiability from existence of densities} \label{s:densities}
In this section we prove Theorem \ref{t:main2}. In Theorem \ref{t:uniffromdensity}, we show that almost everywhere existence of $\Lambda$-densities implies $\Lambda$-tangents are uniform almost everywhere and theorem \ref{t:lambdadensity} shows that existence of $\Lambda$-densities implies rectifiability. In Theorem \ref{t:decomp} we switch gears and instead of using $\Lambda$-tangents, decompose the measure $\mu$ into countably many pieces to show that a small $\Lambda$-density gap also implies the measure is rectifiable. We then put together all the pieces to prove the equivalences in Theorem \ref{t:main2}.

\begin{theorem} \label{t:uniffromdensity}
Suppose $\Lambda : \R^{n} \to GL(n,\R)$ and for $\mu$ almost every $a$ that $\theta^{m}_{\Lambda}(\mu,a)$ exists. Then for $\mu$ almost every $a$, and every $\nu \in \Tan_{\Lambda}(\mu,a)$, 
$$
\nu(B(x,r)) = \nu(B(0,1)) r^{m} \qquad \forall x \in \spt \nu.
$$
\end{theorem}

Before beginning the proof, we note that one can identify $GL(n,\R)$ with a subset of $\R^{n \times n}$, and we recall that the eigenvalues of a matrix depend continuously upon the coefficients. Therefore, by considering only elements of $GL(n,\R)$ with rational coefficients, given any $\epsilon > 0$ we can cover $GL(n,\R)$ with countably many sets $\{U^{\epsilon}_{i}\}_{i \in \N}$ so that for all $i \in \N$,

\begin{equation} \label{e:ui}
 B_{M}(0,(1-\epsilon)r) \subset  B_{\tilde{M}}(0,r) \subset B_{M}(0,(1+\epsilon)r) \quad \forall M, \tilde{M} \in U_{i}^{\epsilon} \quad \forall i \in \N.
\end{equation}

\begin{proof}
By Corollary \ref{c:density} if $\theta^{m}_{\Lambda}(\mu,a)$ exists, then,
\begin{equation} \label{e:nuorigin}
\nu(B(0,r)) = \nu(B(0,1)) r^{m} \qquad \forall \nu \in \Tan_{\Lambda}(\mu,a).
\end{equation}
In fact, by Theorem \ref{t:tan2ltan} and another application of Corollary \ref{c:density}, we know that for almost every $a$, and all $\nu \in \Tan_{\Lambda}(\mu,a)$
\begin{equation} \label{e:tanorigin}
\nu\left( B(x,r) \right) = T_{x,1}[\nu](B(0,r)) = T_{x,1}[\nu](B(0,1)) r^{m} \quad \forall x \in \spt \nu.
\end{equation}

So the theorem follows from showing that for almost every $a$, 
\begin{equation} \label{e:unifcond1}
T_{x,1}[\nu](B(0,1)) = \nu(B(0,1)) \qquad \forall \nu \in \Tan_{\Lambda}(\mu,a) \quad \forall x \in \spt \nu.
\end{equation}
Indeed, briefly assuming \eqref{e:unifcond1}, Theorem \ref{t:uniffromdensity} follows from \eqref{e:nuorigin} and \eqref{e:tanorigin} that
\begin{equation*} 
\nu \left(B(x,r)\right) = \nu(B(0,1)) r^{m} = \nu(B(0,r)) \quad \forall \nu \in \Tan_{\Lambda}(\mu,a) \quad \forall x \in \spt \nu.
\end{equation*}

Define $E \subset \rn$ as the set of points $a$ so that,
$$
 \exists \nu_{a} \in \Tan_{\Lambda}(\mu,a) \text{ and } \exists x_a \in \spt \nu_{a} \text{ so that } \nu_{a}(B(x_a,1)) \neq \nu_{a}(B(0,1)) .
$$
Assume that $\mu(E) > 0$.  Consequently, for some $k$ large enough,
\begin{align} \nonumber
E(k) &= \bigg\{ a \in B(0,k) : \exists \nu_{a} \in \Tan_{\Lambda}(\mu,a) ~ \exists x_{a} \in \spt \nu_{a},  \\
\label{e:trainek} & \text{ so that } \frac{\nu_{a}(B(x_a,1))}{\nu_a(B(0,1))} \not \in ((1+ k^{-1})^{-1}, 1 + k^{-1})  \bigg \}
\end{align}
has positive measure. Fix such a $k_{0}$. We will reach a contradiction by showing that in fact
\begin{equation} \label{e:upperdensbound}
\frac{\nu_{a}(B(x_{a},1))}{\nu_{a}(B(0,1))} < 1 + k_{0}^{-1}.
\end{equation}
The proof that
$$
\frac{\nu_{a}(B(x_{a},1))}{\nu_{a}(B(0,1))} > (1 + k_{0}^{-1})^{-1}
$$
follows by applying \eqref{e:upperdensbound} to $\tilde{\nu_{a}}= T_{x_{a},1}[\nu]$ with the point $-x_{a} \in \spt \tilde{\nu_{a}}$. 

Let $A$ be the set of all $a \in \rn$ such that $\theta^{m}_{\Lambda}(\mu,a) \in (0,\infty)$ and on $A$ define the function
$$
F_{i}(a) =   \sup_{r < 2^{-i}}  \frac{\mu(B_{\Lambda}(a,r))}{\theta^{m}_{\Lambda}(\mu,a) r^{m}}.
$$
Note that for $\mu$ almost every $a$, $1 \le F_{i+1}(a) \le F_{i}(a)$ and $\lim_{i \to \infty} F_{i}(a) = 1$. In particular, since $\mu(E(k)) > 0$, for $\epsilon_{0} > 0$ there exists $i_{0}$ large enough so that
$$
E_{i_0}= \{ a \in E(k_0) : 0 \le F_{i_0}(a) - 1 < \epsilon_{0} \} 
$$
has $\mu(E_{i_0}) >0$. It follows that 
\begin{equation} \label{e:traineps0}
\frac{\frac{\mu(B_{\Lambda}(a,r))}{r^{m}}}{\theta^{m}_{\Lambda}(\mu,a)} < 1 + \epsilon_{0} \qquad \forall r < 2^{-i_0} \quad \forall a \in E_{i_0}
\end{equation}
For $\epsilon_{1} > 0$, let $\{U_{j}^{\epsilon_{1}}\}_{j \in \N}$ be a countable cover of $GL(n,\R)$ as in \eqref{e:ui}. Define $A_{j}= \{ a \in E_{i_0} : \Lambda(a) \in U_{i}^{\epsilon_{1}} \}$. Since $\cup_{j} A_{j}$ covers $E_{i_0}$ there exists some $j_0$ so that $\mu(A_{j_0}) > 0$.  For $\epsilon_2 > 0$, define for $k \in \Z$
$$
A^{k}_{j_{0}} = \left\{ a \in A_{j_{0}} : \theta^{m}_{\Lambda}(\mu,a) \in [(1+\epsilon_{2})^{k}, (1 + \epsilon_{2})^{k+1}) \right\}.
$$
If $a_1,a_2  \in A^{k}_{j_{0}}$ for some $k$, it follows 
\begin{equation} \label{e:traineps2}
(1+\epsilon_{2})^{-1} \le \frac{\theta^{m}_{\Lambda}(\mu,a_1)}{\theta^{m}_{\Lambda}(\mu,a_2)} \le (1+\epsilon_{2}).
\end{equation}
Since $(0, \infty) = \cup_{k \in \Z} [ (1+ \epsilon_{2})^{k}, (1+\epsilon_{2})^{k+1})$, there exists some $k$ with $\mu(A^{k}_{j_{0}}) > 0$ and we denote this set by $A$.

By Lemma \ref{l:ai} almost every $a \in A$ satisfies the density condition \eqref{e:density1}. Fix such an $a$. Let $\nu_{a} \in \Tan(\mu,a)$ and $x_{a} \in \spt \nu_{a}$ be as in \eqref{e:trainek}. By Corollary \ref{c:density}, suppose without loss of generality that
\begin{equation} \label{e:nuawlog}
\nu_{a} = \frac{\nu_{a}(B_{1})}{\theta^{m}_{\Lambda}(\mu,a)}  \lim_{i} r_{i}^{-m} T_{a,r_{i}}^{\Lambda}[\mu].
\end{equation}
By \eqref{e:aix} of Lemma \ref{l:ai}, there exists $\{a_{i}\} \subset A$ so that
\begin{equation} \label{e:trainaix}
\lim_{i \to \infty} \Lambda(a)^{-1} \left( \frac{a_{i} -a}{r_{i}} \right) = x_{a}.
\end{equation}

Since $\nu_{a}$ is a uniform Radon measure, $\nu_{a}(\partial B(0,1)) = \nu_{a}(\partial B(x_a,1)) = 0$.  Applying Proposition \ref{p:continuity} twice, once with the choices 
$\mu= \nu_{a}$, $\mu_{i}= T_{a,r_{i}}^{\Lambda}[\mu]$, $T =  T_{x_{a},1}$, and $T_{i}= T_{\Lambda(a)^{-1} \left( \frac{a_{i}-a}{r_{i}} \right),1}$  and again with the choices $\mu= \nu_{a}$, $\mu_{i}= T_{a,r_{i}}^{\Lambda}[\mu]$, $T =  T_{x_{a},1}$, and $T_i$ the constant sequence $T_{i}=T_{x_{a},1}$, then using  \eqref{e:nuawlog}, \eqref{e:trainaix}, and Corollary \ref{c:density}, it follows 
\begin{equation} \label{e:trainxfer}
\lim_{i \to \infty} \frac{T_{x_a,1} \circ T^{\Lambda}_{a,r_{i}}[\mu](B(0,1))}{r_{i}^{m}}  = \theta^{m}_{\Lambda}(\mu,a) = \lim_{i \to \infty} \frac{ T_{\Lambda(a)^{-1}\left( \frac{a_{i}-a}{r_{i}} \right),1} \circ T^{\Lambda}_{a,r_{i}}[\mu](B(0,1))}{r_{i}^{m}}.
\end{equation}
Let $\Lambda_{a}$ denote the constant matrix-valued function from $\rn \to GL(n,\R)$ given by $y \mapsto \Lambda(a)$. A computation shows
$$
T^{\Lambda_{a}}_{a_{i},r_{i}}(y) = T_{\Lambda(a)^{-1} \left( \frac{a_{i} - a}{r_{i}} \right), 1} \circ T^{\Lambda}_{a,r_{i}}(y).
$$
Therefore \eqref{e:trainxfer} implies
$$
\frac{\nu_{a}(B(x_a,1))}{\nu_{a}(B_{1})} = \frac{1}{\theta^{m}_{\Lambda}(\mu,a)} \lim_{i \to \infty} \frac{ \mu (B_{\Lambda_{a}}(a_{i},r_{i}))}{r_{i}^{m}}.
$$
Now \eqref{e:ui} and $a \in A_{j_0}$ ensures
$$
\frac{ \nu_{a}(B(x_a,1))}{\nu_{a}(B(0,1))} \le \frac{(1+\epsilon_{1})^{m}}{\theta^{m}_{\Lambda}(\mu,a)}  \limsup_{i \to \infty} \frac{\mu \left( B_{\Lambda}(a_{i},(1+\epsilon_{1})r_{i}) \right)}{(1+\epsilon_{1})^{m} r_{i}^{m}}.
$$
When $i$ is large enough that $(1+\epsilon_{1}) r_{i} < 2^{-i_{0}}$, \eqref{e:traineps0} and \eqref{e:traineps2} imply
\begin{align*}
\frac{ \nu_{a}(B(x_a,1))}{\nu_{a}(B(0,1))}  & < \frac{(1+\epsilon_{1})^{m}}{\theta^{m}_{\Lambda}(\mu,a)}  (1 + \epsilon_{0}) \limsup_{i \to \infty} \theta^{m}_{\Lambda}(\mu,a_{i}) \\
& \le (1 + \epsilon_{1})^{m} (1+\epsilon_{0}) (1+\epsilon_{2}).
\end{align*}
For $\epsilon_{0}, \epsilon_{1}, \epsilon_{2}$ small enough, this is less than $ 1 + k_{0}^{-1}$, verifying \eqref{e:upperdensbound} and reaching a contradiction.
\end{proof}

\begin{theorem} \label{t:lambdadensity}
If $\mu$ is a Radon measure on $\R^{n}$ and $\Lambda: \rn \to GL(n,\R)$ are such that
$0 < \theta^{m}_{\Lambda}(\mu,a) < \infty$ for $\mu$ almost every $a$, then $\Tan(\mu,a) \subset \cM_{n,m}$ for almost every $a$. In particular, $\mu$ is $m$-rectifiable.    
\end{theorem}

\begin{proof}
By Theorems \ref{t:tan2ltan} and \ref{t:uniffromdensity} for almost every $a$, and all $\nu \in \Tan_{\Lambda}(\mu,a)$, $\Tan[\nu] \subset \Tan_{\Lambda}(\mu,a) \subset \cU^{m}(\rn)$. Lemma \ref{l:unifimpliesflattan} implies $\Tan[\nu] \cap \cM_{n} \neq \emptyset$.

 Moreover, since $\nu \in \cU^{m}(\R^{n})$ it follows in fact that whenever $\nu_{a} \in \Tan[\nu] \cap \cM_{n}$ then $\nu_{a} \in \cM_{n,m}$. By Lemma  \ref{l:compactbases} and \ref{l:hasP}, we can apply Lemma \ref{t:jams} to $\cF = \cM_{n,m}$ and $\cM = \cU^{m}(\R^{n})$ to conclude that for almost every $a$,
$\Tan_{\Lambda}(\mu,a) \subset \cM_{n,m}$. By \cite[Theorem 5.6]{preiss1987geometry} this implies $\mu$ is $m$-rectifiable.
\end{proof}

\begin{theorem} \label{t:decomp}
Let $\delta_{n}$ be the dimensional constant in Theorem \ref{t:preiss}.  Suppose $\mu$ is a Radon measure on $\R^{m}$, with the following properties at $\mu$ almost every $a$:
$\theta^{m}_{*}(\mu,a) > 0$ and there exists a $\Lambda(a) \in GL(n,\R)$ so that
\begin{equation} \label{e:densitygap}
\frac{ \theta^{m,*}_{\Lambda}(\mu,a)}{\theta^{m}_{\Lambda,*}(\mu,a)} - 1 < \delta_{n}.
\end{equation}
Then $\mu$ is $m$-rectifiable.
\end{theorem}

The idea of the proof relies on the equivalence of (1) and (3) in Lemma \ref{l:iso}. We use the Lebesgue-Besicovitch differentiation theorem to decompose  the measure $\mu$ into countably many pieces $\mu_{i}$, so that each $\mu_{i}$ has the following two properties: (a) $\mu_{i}$ almost everywhere $\theta^{m}_{\Lambda}(\mu_{i},a)$ exists, and (b) $\Lambda$ has small oscillation on $\mu_{i}$. Together these two properties will imply that a linear transformation of $\mu_{i}$, denoted by $\nu_{i}$, has small density gap, i.e., $\frac{\theta^{m,*}(\nu_{i},a)}{\theta^{m}_{*}(\nu_{i},a)} - 1$ is small $\nu_{i}$ almost everywhere. Then Theorem \ref{t:preiss}(iv) will imply $\nu_{i}$, and consequently $\mu_{i}$, is rectifiable. This type of proof cannot be used to prove Theorem \ref{t:fixedpv} because the cancellation present in the definition of the principal value does not behave well when decomposing a measure into small pieces.

\begin{proof}
Fix some $\Lambda : \rn \to GL(n,\R)$ so that for $\mu$ a.e. $a$, 
\begin{equation} \label{e:kroom}
A_{k} = \left\{ a ~ \bigg| ~ \frac{ \theta^{m,*}_{\Lambda}(\mu,a)}{\theta^{m}_{\Lambda,*}(\mu,a)} - 1 < (1-2^{-k})\delta_{n} \right\}.
\end{equation}

For $k > 2$ and $\epsilon_{k} > 0$ to be chosen later, decompose $GL(n,\R)$ into countably many neighborhoods $\{U_{i}^{\epsilon_k}\}_{i \in \N}$ as in \eqref{e:ui}, so that 
$$
  M^{\prime} B(0,(1-\epsilon_{k})) \subset M B(0,1) \subset   M^{\prime} B(0,(1+\epsilon_k)) \qquad \forall M, M^{\prime} \in U^{i}_{\epsilon_k}.
$$
Define $E_{i,k} = \{ a \in A_{k} : \Lambda(a) \in U_{i}^{\epsilon_{k}} \}$. Since
$$
\mu \left( \rn \setminus \cup_{i,k} E_{i,k} \right) = 0,
$$
rectifiability of $\mu$ follows from confirming $\mu \restr E_{i,k}$ is rectifiable for each $i,k$.

\noindent Fix some $i,k \in \N$. Suppose $M \in U_{i}^{\epsilon_{k}}$ and define
$$
\mu_{M} = (M^{-1})_{\sharp}( \mu \restr E_{i,k}).
$$
Since $M \in U_{i}^{\epsilon_{k}}$, 
\begin{equation} \label{e:contain}
B_{\Lambda}(a,(1-\epsilon_{k})r) \subset B_{M}(a,r) \subset B_{\Lambda}(a,(1+\epsilon_{k} )r).
\end{equation}
Since $M$ is bilipschitz, $\mu_{M}$ is rectifiable if and only if $\mu \restr E_{i,k}$ is rectifiable. The Lebesgue Besicovitch differentiation theorem ensures that for $\mu$ a.e. $a \in E_{i,k}$,
\begin{equation} \label{e:samedens}
\begin{cases}\theta^{m,*}_{\Lambda}(\mu \restr E_{i,k}, a) = \theta^{m,*}_{\Lambda}(\mu,a)  \\
\theta^{m}_{\Lambda,*}(\mu \restr E_{i,k},a) = \theta^{m}_{\Lambda,*}(\mu,a).
\end{cases}
\end{equation}
In particular, \eqref{e:contain} implies
$$
\frac{\mu(B_{\Lambda}(a, (1-\epsilon_{k})r))}{ r^{m}} \le \frac{\mu_{M}(B(M^{-1} a, r))}{r^{m}} \le \frac{\mu(B_{\Lambda}(a,(1+\epsilon_{k})r))}{r^{m}}
$$
so that \eqref{e:samedens} and \eqref{e:kroom} respectively guarantee
\begin{align*}
\frac{\theta^{m,*}(\mu_{M},M^{-1}a)}{\theta^{m}_{*}(\mu_{M},M^{-1}a)} -1 & \le \frac{(1+\epsilon_{k})^{m}}{(1-\epsilon_{k})^{m}} \frac{ \theta^{m,*}_{\Lambda}(\mu,a)}{\theta^{m}_{\Lambda,*}(\mu,a)} - 1 \\
& < \frac{ (1+ \epsilon_{k})^{m}}{(1-\epsilon_{k})^{m}} (1+(1-2^{-k})\delta_{n}) - 1.
\end{align*}
Therefore, if $\epsilon_{k}$ is chosen small enough so that
\begin{equation} \label{e:fixepsk}
\frac{ (1+ \epsilon_{k})^{m}}{(1-\epsilon_{k})^{m}} (1+(1-2^{-k})\delta_{n}) - 1 < \delta_{n}
\end{equation}
then the measure $\mu_{M}$ satisfies Theorem \ref{t:preiss}(iv) and consequently is $m$-rectifiable. Thus $\mu \restr E_{i,k}$ is rectifiable and since $\mu \left( \rn \setminus \cup_{i,k} E_{i,k} \right) = 0$, this implies $\mu$ is $m$-rectifiable.
\end{proof}

We now put together all the pieces to prove Theorem \ref{t:main2}.

\begin{proof}
    For (1) $\iff$ (ii), note that by Lemma \ref{l:iso}, $(\Lambda(a)^{-1})_{\sharp} \Tan(\mu,a) \subset \cM_{n,m}$ if and only if $\Tan_{\Lambda}(\mu,a) \subset \cM_{n,m}$. But the prior condition is equivalent to $\Tan(\mu,a) \subset \cM_{n,m}$, so Theorem \ref{t:preiss} now verifies (1) $\iff$ (ii). The equivalence (1) $\iff$ (iii) follows similarly. That (i) $\implies$ (1) and (iv) $\implies$ (1) are respectively Theorem \ref{t:lambdadensity} and Theorem \ref{t:decomp}.

Clearly (i) $\implies$ (iv), so it suffices to show (1) $\implies$ (i). Since by $m$-rectifiable, we in particular mean $\mu \ll \cH^{m}$, it follows that there exist countably many $m$-dimensional $C^{1}$ embedded manifolds $\Sigma_{i}$ so that $\mu \left( \R^{n} \setminus \cup_{i} \Sigma_{i} \right) = 0$. Without loss of generality, each $\Sigma_{i}$ has a global chart $\varphi_{i}: \Sigma_{i} \to \R^{m}$. Fix  $(\Sigma_{i}, \varphi_{i})$ and define $\nu$ on $\R^{m}$ as the image measure $(\varphi_{i}){\sharp} \mu$. Since $\varphi_{i}$ is a $C^{1}$ diffeomorphism onto its image, $\nu \ll \cL^{m} \restr \varphi_{i}(\Sigma_{i})$. Hence, by Radon-Nikodym and additionally the Lebesgue differentiation theorem in the form of \cite[Theorem 3.21]{folland1999real}, for $\nu$ almost every $x$, 
\begin{align*}
\frac{d \nu}{d\cL^{m}}(x) = \lim_{r \to 0} \frac{\nu(E_{r}(x))}{\cL^{m}(E_{r}(x))} \in (0, \infty)
\end{align*}
for $\nu$ almost every $x$ and for any family of sets $E_{r}(x)$ shrinking nicely to $\{x\}$. In particular when $E_{r}(x) = \varphi( B_{\Lambda}( \varphi^{-1}(x),r))$.  By the Lebesgue differentiation theorem in the form of \cite[Theorem 2.9.11]{federer2014geometric},
\begin{equation} \label{e:lebesguecase}
\lim_{r \to 0} \frac{ \mu(B_{\Lambda}(a,r) \cap  \Sigma_{i})}{\mu(B_{\Lambda}(a,r))} = 1
\end{equation}
for $\mu$ almost every $a \in \Sigma_{i}$. Consequently, for almost every $a \in \Sigma_{i}$,
\begin{align*}
\lim_{r \to 0} \frac{\mu(B_{\Lambda}(a,r))}{r^{m}} & = \lim_{r \to 0} \frac{\mu(B_{\Lambda}(a,r) \cap \Sigma_{i})}{r^{m}} \\
& = \left( \lim_{r \to 0} \frac{\cH^{m}(E_{r}(\varphi(a))}{r^{m}} \right) \left( \lim_{r \to 0} \frac{ \mu \restr \Sigma_{i}(B_{\Lambda}(a,r))}{r^{m}} \frac{r^{m}}{\cH^{m}(E_{r}(\varphi(a))} \right)    \\
& = \omega_{m} (J \varphi)(a) \lim_{r \to 0} \frac{ \mu \restr \Sigma_{i} \left( \varphi^{-1}(E_{r}(a))\right)}{\cL^{m}(E_{r}(\varphi(a))}\\
& = \omega_{m} (J \varphi)(a) \lim_{r \to 0} \frac{\nu(E_{r}(x))}{\cL^{m}(E_{r}(x))} \in (0, \infty),
\end{align*}
where the final conclusion of positive and finite is justified for almost every $a$ since  $\varphi$ is a $C^{1}$ diffeomorphism and \eqref{e:lebesguecase}. Thus for any $\Sigma_{i}$ and $\mu$ almost every $a \in \Sigma_{i}$ we have shown $\theta^{m}_{\Lambda}(\mu,a) \in (0, \infty)$. Since $\mu(\R^{n} \setminus \cup_{i} \Sigma_{i}) = 0$ this proves (i).
\end{proof}

\section{Principal values and rectifiability} \label{s:pvandrect}
In Section \ref{s:pvandrect} we will characterize rectifiable measures in terms of the existence of principal values. In Section \ref{s:existenceofpv} we prove Theorem \ref{t:existenceofpv}, which says that if $\mu$ is a finite, $m$-rectifiable Radon measure on $\R^{n}$ then for any family of sufficiently nice Calderon-Zygmund kernels $\{K_{x}\}$, the principal values defined with respect to these kernels exist $\mu$-a.e. In Section \ref{s:fixedpv} we prove Theorem \ref{t:fixedpv}, which is the converse to Theorem \ref{t:existenceofpv} in the special case that $K_{x}(y-x)$ takes the form $\frac{\Lambda(x)^{-1}(y-x)}{|\Lambda(x)^{-1}(y-x)|^{m+1}}$. In Section \ref{s:dmo} we apply the methods and results of the previous two sections to prove Theorem \ref{t:dmo}.

\subsection{Proof of Theorem \ref{t:existenceofpv}} \label{s:existenceofpv}

Before fully diving into the proof of Theorem \ref{t:existenceofpv}, we first prove Lemma \ref{l:shapechange}.

\begin{proof}[Proof of Lemma \ref{l:shapechange}:]
Fix $\mu,x$ as in the statement of the lemma.
We first note that if $\mu_{1} = \mu \restr B(x,r_{0})$ since
$$
\int_{r_0\geq |y-x|\geq \epsilon}K(y-x)d\mu_{1}(y)=\int_{|y-x|\geq \epsilon}K(y-x)d\mu(y)-\int_{|y-x|\geq r_0}K(y-x)d\mu(y),
$$
the existence of principal values at $x$ are equivalent for $\mu$ and $\mu_{1}$.  In particular, we may without loss of generality replace $\mu$ with $\mu_{1}$ where $r_{0}$ is chosen small enough that $0 < 2^{-1} \theta^{m}_{*}(\mu,x) r^{m} \le \mu(B(x,r)) \le 2 \theta^{m,*}(\mu,x) r^{m} < \infty$ for all $0 < r < r_{0}$. We now suppose
$$
\displaystyle \lim_{\epsilon \downarrow 0} \int_{|y-x| \ge \epsilon} K(y-x) d \mu(y)\in \mathbb{R}^n.
$$ 

This proof can be broken into 3 steps:  verifying \eqref{e:000}, \eqref{e:00}, and \eqref{e:01}. We only verify these steps in the case when $\oalpha_{\mu}(x,\epsilon) > 0$ for all $\epsilon > 0$, because if exists an $\epsilon_{0} > 0$ so that $\oalpha_{\mu}(x,\epsilon_{0}) = 0$ we may again without loss of generality replace $\mu$ with $\mu \restr B(x,\epsilon_{0})$ in which case the existence of the principal values follows readily from the oddness of $K$.

\eqref{e:000} is precisely \cite[Proposition 3.32]{orponen2023sub} which guarantees that for any non-decreasing $\phi:[0,\infty) \to [0,1]$ with $0\not \in \spt \phi$ and $\lim_{s\to \infty}\phi(s)=1$:
\begin{equation} \label{e:000}
    \lim_{\epsilon \downarrow 0} \int_{|y-x| \ge \epsilon} K(y-x) d \mu(y)=\lim_{\epsilon\downarrow 0}\int_{\mathbb{R}^n}\phi_{\epsilon}(|y-x|)K(y-x)d\mu(y), 
\end{equation}
where $\phi_{\epsilon}(s)=\phi(\frac{s}{\epsilon})$. For our purpose, we only need to consider $\phi$ from the following family of smooth cut-offs parameterized by $\eta>0$:
$$
\phi_{\epsilon}^{\eta}(s) = 
\begin{cases} 0 & s \le \epsilon\\ 
\frac{\eta}{\epsilon}(s-\epsilon) & \epsilon\le s \le \epsilon(1+\frac{1}{\eta}) \\
1 & s \ge \epsilon(1+\frac{1}{\eta}) \end{cases}.
$$
We claim that for all $\delta > 0$, there exists $\epsilon_{0} >0$ so that for all $0<\epsilon<\epsilon_0$ there exists $\eta:=\eta(\epsilon)$ such that 
\begin{equation} \label{e:00}
   \left|\int_{\mathbb{R}^n}\phi^{\eta}_{\epsilon}(|y-x|)K(y-x)d\mu(y)-\int_{\mathbb{R}^n} \phi^{\eta}_{\epsilon}(\|y-x\|)K(y-x)d\mu(y)\right| < \frac{\delta}{2}. 
\end{equation}

We first observe, that the comparability of norms implies there exists some $\Lambda_{0}$ so that
$$
B_{\| \cdot \|}(x, \Lambda_{0}^{-1} \epsilon) \subset B(x,\epsilon) \subset B_{\| \cdot \|}(x, \Lambda_{0} \epsilon).
$$
Therefore $\| \cdot \|$ is $\Lambda_{0}$-Lipschitz. 
On the other hand, if 
$$
\psi^{\eta}_{\epsilon} ( y ) \defeq \phi^{\eta}_{\epsilon}(| y-x|) - \phi^{\eta}_{\epsilon}(\| y-x\|)
$$
then whenever $\eta \ge 2$, it follows
\begin{equation} \label{e:supportconstraint}
\spt \left(\psi^{\eta}_{\epsilon} \right) \subseteq B(x, 2 \Lambda_{0} \epsilon) \setminus B(x, \Lambda_{0}^{-1} \epsilon) \eqdef A_{\Lambda_{0}}(\epsilon).
\end{equation}
Moreover, by the triangle inequality and chain rule for Lipschitz functions
\begin{align*}
\lip \left(\psi^{\eta}_{\epsilon} \right) & \le \lip \left( \phi^{\eta}_{\epsilon}(| \cdot |) \right) + \lip \left(\phi^{\eta}_{\epsilon}(\| \cdot \| ) \right) \\
& \le \lip \left(\phi^{\eta}_{\epsilon} \right) \left(1 + \Lambda_{0} \right) = \frac{\eta}{\epsilon} \left(1 + \Lambda_{0} \right).
\end{align*}
On the other hand, since $K \in C^{1}(\R^{n} \setminus \{0\})$ is $-m$-homogeneous,
$$
\|K(\cdot-x)\|_{L^{\infty}(A_{\Lambda_{0}}(\epsilon))} \lesssim (\Lambda_{0}^{-1} \epsilon)^{-m} \quad \text{and} \quad \|\nabla K(\cdot-x)\|_{L^{\infty}(A_{\Lambda_{0}}(\epsilon)} \lesssim (\Lambda_{0}^{-1} \epsilon)^{-(m+1)} 
$$
Using $\eta \ge 2$, the product rule, non-negativity of $\phi^{\eta}_{\epsilon}$, and \eqref{e:supportconstraint} we deduce
\begin{align*}
    \lip & \left( \psi^{\eta}_{\epsilon} K \right) \le \|\phi^{\eta}_{\epsilon}\|_{L^{\infty}} \| \nabla K \|_{L^{\infty}(A_{\Lambda_{0}}(\epsilon))} + \lip(\psi^{\eta}_{\epsilon}) \|K\|_{L^{\infty}(A_{\Lambda_{0}}(\epsilon))} \\
    & \lesssim (\Lambda_{0}^{-1} \epsilon)^{-(m+1)} + \eta \epsilon^{-1} \left( 1 + \Lambda_{0} \right) \left( \Lambda_{0}^{-1} \epsilon \right)^{-m} \\
    & \lesssim \eta \epsilon^{-(m+1)}.
 \end{align*}
 In particular, there exists some constant $C$ independent of $\eta, \epsilon$ so that
 $$
        \frac{ \epsilon^{m+1} \psi^{\eta}_{\epsilon} K}{C\eta} \in \lip_{1}(B(x, 2 \Lambda_{0} \epsilon)).
 $$
 Meanwhile, since $\psi_{\epsilon}^{\eta}K$ is an odd function, whenever $\sigma \in \cM_{n,m}$ and $x \in \spt \sigma$ it follows
 \begin{align*}
     \int_{\mathbb{R}^n}\psi^{\eta}_{\epsilon}(y-x) K(y-x)d\sigma(y) = 0.
 \end{align*}
Therefore,
\begin{align}\label{e:firstchoice of eps0}
\begin{split}
\left| \int_{\R^{n}} \psi^{\eta}_{\epsilon}(y)K(y-x) d \mu(y) \right| & = \left| \int_{\R^{n}} \psi^{\eta}_{\epsilon}(y)K(y-x) d (\mu- \sigma)(y) \right| \\
& \le \frac{C\eta}{\epsilon^{m+1}} \left(2 \Lambda_{0} \epsilon \right)^{m+1} \oalpha_{\mu}(x, 2 \Lambda_{0} \epsilon) \\
& \lesssim \eta \oalpha_{\mu}(x,2 \Lambda_{0} \epsilon),
\end{split}
\end{align}
where the suppressed constant is independent of $\eta$ and $\epsilon$. Choosing $\eta = \max\left\{2,\frac{1}{\sqrt{ \oalpha_{\mu}(x,2 \Lambda_{0} \epsilon)}} \right\}$ and taking $\epsilon$ sufficiently small implies \eqref{e:00}. It remains to show that, 
\begin{equation} \label{e:01}
\left| \intrn \phi^{\eta}_{\epsilon}(\| y-x \|) K(y-x) d \mu(y) - \int_{\|y-x\| \ge \epsilon} K(y-x) d \mu(y) \right| \le \frac{\delta}{2}.
\end{equation}
To verify \eqref{e:01}, we consider the function
$$
\Psi^{\eta}_{\epsilon} (y) \defeq \phi^{\eta}_{\epsilon}(\| y-x \|) - \chi_{B_{\|\cdot\|}(x,\epsilon)^c}(y).
$$
From the definition of $\phi^{\eta}_{\epsilon}$, it follows
$$
\spt \left( \Psi^{\eta}_{\epsilon} \right) = \overline{ B_{\| \cdot \|}(x, (1+\eta^{-1})\epsilon) \setminus B_{\| \cdot \|}(x, \epsilon)} \eqdef A^{\eta}(\epsilon).
$$
On the other hand, $\| K( \cdot-x)\|_{L^{\infty}(A^{\eta}(\epsilon))} \lesssim \epsilon^{-m}$. So,
\begin{align}
\nonumber    \bigg|  \intrn & \phi^{\eta}_{\epsilon}(\| y-x \|) K(y-x) d \mu(y) - \int_{\|y-x\| \ge \epsilon} K(y-x) d \mu(y) \bigg| \\
 \label{e:annularestimate}   & \lesssim \epsilon^{-m} \mu \left(A^{\eta}(\epsilon) \right).
\end{align}
Suppose $2 \Lambda_{0} \epsilon \le r_{0}$, defined at the beginning of this proof and that $\epsilon_{0}$ is small enough that $\sqrt{\oalpha_{\mu}(x,\epsilon)} \le \frac{1}{2}$ for all $0 < \epsilon \le \epsilon_{0}$. By Lemma \ref{l:smallanulli},
\begin{equation} \label{e:smallannulusindeed}
\frac{\mu \left( A^{\eta}(\epsilon) \right)}{\epsilon^{m}} \lesssim_{\Lambda_{0},m} \eta \oalpha_{\mu}(B(x,2 \Lambda_{0} \epsilon)) + \frac{\theta^{m}_{\mu}(x,\Lambda_{0} \epsilon) }{\eta} \le (1 + 2\theta^{m,*}_{\mu}(x)) \sqrt{ \oalpha_{\mu}(x, 2 \Lambda_{0} \epsilon)}.
\end{equation}
By choosing $\epsilon_0$ small enough so that the bounds in \eqref{e:firstchoice of eps0} and \eqref{e:smallannulusindeed} are small compared to $\delta$, combining \eqref{e:annularestimate} with \eqref{e:smallannulusindeed} verifies \eqref{e:01}.  Since $\delta > 0$ is arbitrary, combining \eqref{e:000}, \eqref{e:00}, and \eqref{e:01} verifies the forward implication of \eqref{e:change of shape} with equality of the limiting values. The reverse implication is proven similarly.

\end{proof}

To motivate our next result, we consider a family of kernels parametrized by $x\in \mathbb{R}^n$. In particular, we consider the family of kernels $K_x(y-z):=K(y-z)=\frac{\Lambda(x)^{-1}(y-z)}{|\Lambda(x)^{-1}(y-z)|^{m+1}}$. Then for any given $x$, we know from standard arguments that $\displaystyle\lim_{\epsilon\downarrow 0}\int_{|y-z|\geq \epsilon}\frac{\Lambda(x)^{-1}(y-z)}{|\Lambda(x)^{-1}(y-z)|^{m+1}}d\mu(y)$ exists for almost every $z$, but not necessarily when $z=x$. So, in order to prove that the principal values exist for almost every $x$ when the kernels are of the form $K_{x}(y-x)$, we use the following lemma.

\begin{lemma}\label{l: pv exist geo hiccup}
If $\mu$ is a finite $m$-rectifiable measure on $\R^{n}$, then there exists an $M \in \N$ and a set $A \subset \R^{n}$ satisfying $\mu \left(\R^{n} \setminus A \right) = 0$ with the following property: For all $x \in A$ and any $-m$-homogeneous odd kernel $K \in C^{M} \left( \R^{n} \setminus \{0\} ; \R^{n} \right)$, the principal value
\begin{equation} \label{e: pv geo hiccup}
\lim_{\epsilon \downarrow 0} \int_{|y-x| \ge \epsilon} K(y-x) d \mu(y) \in \R^{n}.
\end{equation}
\end{lemma}

In light of Lemma \ref{l:shapechange}, Theorem \ref{t:existenceofpv} is now an immediate consequence of Lemma \ref{l: pv exist geo hiccup}.

Before proceeding with the proof, we remark that our Lemma \ref{l: pv exist geo hiccup} looks a lot like \cite[Theorem 3.1]{puliatti2022gradient}. In private communication with the author of \cite{puliatti2022gradient}, we have confirmed that the extra strength of \cite[Theorem 3.1]{puliatti2022gradient} stems from a miss-citation which also simplified their proof. Therefore, while parts of our proof may look similar to the presentation of \cite[Theorem 3.1]{puliatti2022gradient}, we have additional technicalities to overcome. Some of the key ideas of the proof are already present in \cite{MT99}.

The novelties are the additional decompositions of our base measure $\mu$ into Lipschitz graphs, allowing the use of well-known $L^{2}$-boundedness of singular integral operators on Lipschitz graphs, \cite{david1991singular}.   Due to the technical nature of the proof, we first present a sketch of it.

Inspired by \cite{MT99} (see also \cite{puliatti2022gradient}), we first enumerate the $-m$-homogeneous extensions of the odd spherical harmonics. These will play the role of a countable basis for all nice odd kernels. By \cite[Theorem 20.28]{mattila1999geometry}, the principal value exists almost everywhere for every basis kernel simultaneously. This set of full measure is the set $A$. We then expand our nice kernels in terms of the basis kernels as in \eqref{e:expansion} and the goal is to show that the $L^{2}$-convergence in that equation is strong enough to preserve the existence of principal values at every point in $A$.

Establishing appropriate convergence for the partial sums of the basis kernels comes down to having quantitative $L^{2}$-bounds for the operators induced by the basis kernels.  
However, even for the basis kernels, uniform rectifiability of the measure is necessary to obtain this quantitative $L^2$-boundedness information. Thus a decomposition of the measure is necessary. We take advantage of the fact that our measure is $m$-rectifiable and consider the representations of $\mu$ as $\mu = \mu_{i} + \sigma_{i}$, where $\mu_{i} = \mu \restr \Gamma_{i}$ and $\sigma_{i} = \mu - \mu_{i}$, and $\Gamma_{i}$ is one of the Lipschitz graphs carrying $\mu$. Then, we treat $\mu_{i}$ as in \cite{puliatti2022gradient} and $\sigma_{i}$ as in the "singular part" from \cite[Theorem 20.28]{mattila1999geometry}.

The quantitative bounds for $\mathcal{H}^m\restr \Gamma_i$ with respect to the basis kernels are due to \cite{david1991singular}. To relate this information back to the existence of principal values of $\mu$, we need to know that the Radon-Nikodym derivative $\frac{d \mu}{d \cH^{m} \restr \Gamma_{i}} \in L^{2}(\cH^{m} \restr \Gamma_{i})$. This is a technicality that forces an additional decomposition of $\mu$ depending upon its density bounds, but this has no meaningful impact on the ideas outlined above.

\begin{proof}  
    Consider the orthonormal basis for $L^2(\mathbb{S}^{n-1})$ given by the surface spherical harmonics, $\{\varphi_{j,\ell}\}_{j\geq 1, 1\leq \ell\leq N_j}$, of degree $j$, where $N_j=O(j^{n-2})$ for large $j$, c.f., \cite[2.12]{atkinsonandhan}. We can write $K$ in terms of this basis, by exploiting the $-m$-homogeneity of $K$, 
    \begin{equation} \label{e:expansion}
        K(z)=|z|^{-m}K\left(\frac{z}{|z|}\right)=\sum_{j\geq 1}\sum_{\ell=1}^{N_j}\langle K, \varphi_{j,\ell}\rangle_{L^2(\mathbb{S}^{n-1})}|z|^{-m}\varphi_{j,\ell}\left(\frac{z}{|z|}\right).
    \end{equation}
    Let $\kjl \defeq \langle K, \varphi_{j,\ell}\rangle_{L^2(\mathbb{S}^{n-1})}$ and observe that $\kjl = 0$ for all even $j$ since  $K$ is an odd function and $\kjl$ is even. Thus, 
    \begin{equation*}
         K(z)=\sum_{\substack{j\geq 1\\j\,\text{odd}}}\sum_{\ell=1}^{N_j}\kjl \pjl(z), 
    \end{equation*}
    where $\pjl(z)=|z|^{-m}\varphi_{j,\ell}(\frac{z}{|z|})$ is an odd Calderon-Zygmund kernel, smooth on $\R^{n} \setminus \{0\}$. Therefore existence in \eqref{e: pv geo hiccup} is equivalent to existence of
    \begin{equation}\label{e: pv spherical harmonic geo hiccup}
        \lim_{\epsilon\downarrow 0} \int_{|y-x|\ge\epsilon}\sum_{\substack{j\geq 1\\j\,\text{odd}}}\sum_{\ell=1}^{N_j}\kjl \pjl(y-x)d\mu(y). 
    \end{equation}

To verify \eqref{e: pv spherical harmonic geo hiccup}, we will first apply Fubini theorem to show that for all $\epsilon > 0$,
    \begin{equation} \label{e: pv Fubini no limit}
    \int_{|y-x|\ge\epsilon}\sum_{\substack{j\geq 1\\j\,\text{odd}}}\sum_{\ell=1}^{N_j}\kjl \pjl(y-x)d\mu(y) = \sum_{\substack{j\geq 1\\j\,\text{odd}}}\sum_{\ell=1}^{N_j} \int_{|y-x|\ge\epsilon} \kjl \pjl(y-x)d\mu(y).
    \end{equation}
    We will then perform a further decomposition of $\mu$ into Lipschitz graphs to show that there exists a $C(x)$ so that
    \begin{equation} \label{e:independentofeps}
    \sum_{\substack{j\geq 1\\j\,\text{odd}}}\sum_{\ell=1}^{N_j} \left| \int_{|y-x|\ge\epsilon}  \kjl \pjl(y-x) d\mu(y) \right|  \le C(x)
    \end{equation}
    is independent of $\epsilon$. This will additionally allow us to apply the Dominated Convergence Theorem to interchange the limit with the double sum and determine that the limit in \eqref{e: pv spherical harmonic geo hiccup} exists if and only if the limit
    \begin{equation}\label{e: pv Fubini}
        \sum_{\substack{j\geq 1\\j\,\text{odd}}}\sum_{\ell=1}^{N_j}\kjl \lim_{\epsilon\downarrow 0} \int_{|y-x|\ge\epsilon}\pjl(y-x)d\mu(y) \in \R^{n} 
    \end{equation}
    exists. So, we first verify \eqref{e: pv Fubini no limit}. Let us recall that $H^s(\mathbb{S}^{n-1})$, $s\in \mathbb{R}$, is the completion of $C^{\infty}(\mathbb{S}^{n-1})$ with respect to the norm 
    \begin{equation*}
    \|u\|_{H^s({\mathbb{S}^{n-1}})}=\left(\sum_{j\geq 1}\sum_{\ell=1}^{N_j}\left(j+\frac{n-2}{2}\right)^{2s}|\langle u, \varphi_{j,\ell}\rangle_{L^2(\mathbb{S}^{n-1})}|^2\right)^{1/2},
\end{equation*}
and the Sobolev embedding theorem ensures $H^s(\mathbb{S}^{n-1})$ continuously embeds into $C(\mathbb{S}^{n-1})$ for $s>\frac{n-1}{2}$\footnote{For a more thorough introduction to this space see \cite[Section 3.8]{atkinsonandhan}, or \cite[Lemma 3.1]{puliatti2022gradient} and references therein.}. Choose $s=\frac{n-1}{2}+1$. Then, 
\begin{equation}\label{e:bound using Sobolev norm}
   \|\varphi_{j,\ell}\|_{C^J(\mathbb{S}^{n-1})}\lesssim_n \sum_{i=0}^{J}\|\varphi_{j,\ell}\|_{H^{s+i}(\mathbb{S}^{n-1})}=\sum_{i=0}^J\left(j+\frac{n-2}{2} \right)^{\frac{n-1}{2}+1+i} \lesssim_{n} 2^{J} j^{\frac{n-1}{2} + 1 + J}.
\end{equation}

    So, for $|y-x| \ge$ $\epsilon>0$ we have
\begin{equation*}
    |\pjl(y-x)|\leq \epsilon^{-m}\|\varphi_{j,\ell}\|_{L^{\infty}(\mathbb{S}^{n-1})}\lesssim_n 
    \epsilon^{-m}\left(j+\frac{n-2}{2} \right)^{\frac{n-1}{2}+1}, 
\end{equation*}
and for $j$ large enough depending on $n$, 
\begin{equation}\label{e: Phi bound}
    |\pjl(y-x)|\lesssim_n \epsilon^{-m} j^{\frac{n-1}{2}+1}.
\end{equation}

    Moreover, for each $x\in \mathbb{R}^n$ and $\epsilon>0$ the smoothness of $K$ on $\mathbb{S}^{n-1}$, Green's theorem,  $\Delta_{\mathbb{S}^{n-1}}\varphi_{j,\ell}=\lambda_{j}\varphi_{j,\ell}$, $\|\varphi_{j,\ell}\|_{L^{2}(\St^{n-1})}=1$, and $\lambda_{j} = j(j+n-2)$ implies that  for every $r\in \mathbb{N}$, 
\begin{align}\label{e:Stein}
\nonumber    \kjl 
&= \frac{ \left| \int_{\St^{n-1}} K (\Delta^{r}_{\St^{n-1}} \varphi_{j,\ell}) d \cH^{n-1} \right|}{\lambda_{j}^{r}}\\
    &\le \frac{\left| \int_{\mathbb{S}^{n-1}}(\Delta_{\mathbb{S}^{n-1}}^r)K(z)\varphi_{j,\ell}d\mathcal{H}^{n-1} \right|}{j^{2r}}\lesssim_{n} \|K\|_{C^{2r}(\St^{n-1})} j^{-2r}.
\end{align}
For more details, see for instance \cite[(3.1.4)]{steinbook}.

Recalling $N_j=O(j^{n-2})$, and gathering the powers of $j$ in \eqref{e: Phi bound} and \eqref{e:Stein} gives 
\begin{equation*}
    \int_{|y-x| \ge \epsilon}\sum_{\substack{j\geq 1\\j\,\text{odd}}}\sum_{\ell=1}^{N_j} \left| \kjl \pjl(y-x) \right| d\mu(y)\lesssim_{n} \int_{|y-x|\ge \epsilon} C\sum_{\substack{j\gg 
    1\\j\,\text{odd}}}j^{-2r+\frac{n-1}{2}+1+n-2} d\mu(y),  
\end{equation*}
where $C = \epsilon^{-m} \|K\|_{C^{2r}(\St^{n-1})}$. Furthermore, if  $r\geq \frac{3n+1}{4}$ in (\ref{e:Stein}), we obtain 
\begin{equation*}
    \int_{|y-x| \ge \epsilon}\sum_{\substack{j\geq 1\\j\,\text{odd}}}\sum_{\ell=1}^{N_j} \left| \kjl \pjl(y-x) \right| d\mu(y)\lesssim_{n} \int_{|y-x|\ge \epsilon} C\sum_{\substack{j\gg 
    1\\j\,\text{odd}}}j^{-2} d\mu(y)< \infty. 
\end{equation*}
Thus verifying \eqref{e: pv Fubini no limit}.  We now turn our attention to proving \eqref{e:independentofeps}. To this end, we would like to say that 
$$
\tjl \mu(x) \defeq \lim_{\epsilon\downarrow 0}\int_{|y-x|\ge \epsilon}\pjl(y-x)d\mu(y)
$$ 
is $L^2(\mu)\to L^2(\mu)$ bounded by a polynomial in $j$, see \cite[(3.5)]{puliatti2022gradient} and the discussion therein following (3.6).  In combination with \eqref{e:Stein}, we would verify  (\ref{e:independentofeps}). But, $\mu$ is $m$-rectifiable, which is not sufficient to obtain a quantitative $L^{2}$-bound on $\tjl \mu$. So, to achieve the quantitative polynomial bounds, we use the $m$-rectifiability of $\mu$ to decompose it into the underlying Lipschitz graphs.

    Denote by $\Gamma_i$, the Lipschitz graphs such that $\mu\left(\mathbb{R}^n\setminus \bigcup_i \Gamma_i\right)=0$, and define
\begin{equation*}
    A_k:=\{x\in \mathbb{R}^n: 2^{-k}\leq \theta^m(\mu, x)\leq 2^k \}. 
\end{equation*}
Observe that $\mu\left(\mathbb{R}^n\setminus \bigcup_k A_k\right)=0$. Let $\muik=\mu\restr(\Gamma_i\cap A_k)$, and $\sigma_{i,k}=\mu-\muik$. Denoting $\hi  = \cH^{m} \restr \Gamma_{i}$ it is sufficient to show that for each $i,k$ 

\begin{equation}\label{e: pv on Lip graph}
    \sum_{j\geq 1}\sum_{\ell=1}^{N_j}\kjl \lim_{\epsilon\downarrow 0}\int_{|y-x|\ge \epsilon}\pjl(y-x)f_{i,k}(y)d\hi(y) \in \R^{n}, 
\end{equation}
and
\begin{equation}\label{e: pv off Lip graph}
   \sum_{j\geq 1}\sum_{\ell=1}^{N_j}\kjl \lim_{\epsilon\downarrow 0}\int_{|y-x| \ge  \epsilon}\pjl(
   y-x)d\sigma_{i,k}(y) \in \R^{n}
\end{equation}
    each exist $\hi$-a.e. $x \in A_{k}$, where $f_{i,k}$ is the Radon-Nikodym derivative $\frac{\displaystyle d (\mu \restr A_{k})}{d \hi}$. Indeed, both \eqref{e:independentofeps} and \eqref{e: pv Fubini} follow from \eqref{e: pv on Lip graph} and \eqref{e: pv off Lip graph}  since $\mu \ll \cH^{m}$ and $\mu(\R^{n}) < \infty$ imply that for any bounded Borel function $g$, it holds $\int g d \mu = \int g f_{i,k} d \hi + \int g d \sigma_{i,k}$. 

From now on, let $i,k$ be fixed and denote $\Gamma:=\Gamma_i$, $\hg \defeq \mathcal{H}^m\restr\Gamma$, $f := f_{i,k}$, and $\sigma := \sigma_{i,k}$.

We first handle (\ref{e: pv on Lip graph}). For any $g \in L^{2}(\hg)$, define
$$
g\mapsto \tjl \hg g(\cdot)=\lim_{\epsilon\downarrow 0}\int_{|y-\cdot|\ge \epsilon}\pjl(y-\cdot)g(y)d\hg(y).
$$
Since $\hg$ is a uniformly rectifiable measure, it follows from a theorem of David and Semmes, see \cite[Theorem 2.3]{conde2019failure}, that there exists  $J=J(m)$ such that for any ball $B$ centered on $\Gamma$
\begin{equation*}   \|\tjl \hg\|_{L^2(\hg \restr B)\to L^2(\hg \restr B)}\lesssim  \|\pjl \restr \mathbb{S}^{n-1}\|_{C^{J}(\mathbb{S}^{n-1})}=\|\varphi_{j,\ell}\|_{C^{J}(\mathbb{S}^{n-1})},
\end{equation*}
where the suppressed constants depend on $m,n$ and the Lipschitz constant of the function defining $\Gamma$.
Thus, from (\ref{e:bound using Sobolev norm}), 
    \begin{equation}\label{e:L^2 boundedness}
\| \tjl \hg\|_{L^2(\hg \restr B )\to L^2(\hg \restr B)}\lesssim  2^{J} j^{\frac{n-1}{2} + 1 + J},
\end{equation}
again with suppressed dependencies on $m,n$, and the Lipschitz constant of the function defining $\Gamma$.

Using the triangle inequality and the Lebesgue density theorem applied to the function $T_{j,l}\hg f$ for $\hg$-.a.e. $x\in A_k$   we find an $r(x)$ so that 
 \begin{align*}
      \left| \sum_{j\geq 1}\sum_{\ell=1}^{N_j}\kjl \tjl \hg f(x) \right|
    &\leq \sum_{j\geq 1}\sum_{\ell=1}^{N_j}|\kjl \tjl \hg f(x)|\\
    & \hspace{-1in} \leq \sum_{j\geq 1}\sum_{\ell=1}^{N_j} |\kjl | \left|\frac{2}{\hg(B_{r(x)})}\int_{B_{r(x)}} \tjl \hg f(x)d\hg\right|=:I.
    \end{align*}

Note, since $f$ is the Radon-Nikodym derivative $\frac{\displaystyle d (\mu\restr A_k)}{d \hg}$, then by the definition of $A_{k}$, $f(x) \le 2^{k}$ for all $x\in A_k$. Denote $B \defeq B_{r(x)}$. Therefore, for $\hg$-a.e. $x \in A_k$, 
$\|f\|_{L^{2}(\hg|_{B})} \lesssim 2^{k} \sqrt{ \hg(B)}$ where we denote by $B \defeq B_{r(x)}$. Therefore, using Cauchy-Schwarz and \eqref{e:L^2 boundedness},

    \begin{align*}
    I&\leq \frac{2}{\sqrt{\hg(B)}} \sum_{j\geq 1}\sum_{\ell=1}^{N_j}| \kjl | \|\tjl \hg f\|_{L^2(\hg|_{ B} )}\\
    & \hspace{1cm} \leq  \frac{2\|f \|_{L^{2}(\hg |_{B})}}{\sqrt{\hg(B)}} \sum_{j\geq 1}\sum_{\ell=1}^{N_j}|\kjl | \|\tjl \hg\|_{L^2(\hg |_{B}) \to L^2(\hg |_{B})}  \\
    &  \hspace{1cm} \lesssim_{J} 2^{k} \sum_{j \ge 1} \sum_{\ell=1}^{N_{j}} |\kjl  | j^{\frac{n-1}{2} + 1 + J}  \\
    &  \hspace{1cm} \le  2^{k}\|K\|_{C^{2r}(\St^{n-1})} \sum_{j \ge 1} N_{j}    j^{\frac{n-1}{2} + 1 + J-2r}
    \end{align*}
    where the suppressed constants depend on $J,m,n$, and the Lipschitz character of $\Gamma$. Since $N_{j} \lesssim j^{n-2}$, choosing $r\geq \frac{3n+1+2J}{4}$ in \eqref{e:Stein} guarantees that $-2r + \frac{n-1}{2} + 1 +J + n-2 \le - 2$, and then we deduce from the previous estimate that for $\hg$-a.e., $x \in A_{k}$,
    \begin{equation}\label{e:001}
        \sum_{j\geq 1}\sum_{\ell=1}^{N_j}\left|\kjl \tjl \hg f(x) \right| \lesssim 2^{k} \sum_{j \ge 1} j^{-2} < \infty
    \end{equation}
    
    with allowable constants depending on $k$, verifying \eqref{e: pv on Lip graph} holds. We remark that this proves \eqref{e:independentofeps}
    for the part of the measure $\mu \restr \Gamma_{i} \cap A_{k}$ and almost every $x \in \Gamma_{i} \cap A_{k}$, and \eqref{e:independentofeps} will follow by proving a similar bound for $\sigma = \mu - \mu \restr (\Gamma_{i} \cap A_{k})$.

We now show (\ref{e: pv off Lip graph}). This argument is heavily influenced by standard arguments, see \cite[Theorem 20.27]{mattila1999geometry}\footnote{see also \cite{MM94},\cite{V92}}. We include the full proof for completeness. We first show that for any $\alpha>0$ there exists a $\beta>0$ such that there exists $D_{\alpha}\subset\Gamma$ for which $\hg(D_{\alpha})<\alpha$ and 
\begin{equation}\label{e:prop of D_alpha}
    \left| T_{\delta}\sigma(x) -  T_{\epsilon}\sigma(x)\right|\leq \alpha \quad \text{for}\quad x\in \Gamma\setminus D_{\alpha}, \, \delta, \epsilon\in (0,\beta), 
\end{equation}
where 
$$
T_{\delta}\sigma(x)=\sum_{j\geq 1}\sum_{\ell=1}^{N_j}\kjl \int_{|y-x|> \delta} \pjl(y-x) d\sigma(y).
$$

We also remark that, in proving this bound on the measure of $D_{\alpha}$, in \eqref{e:002} we prove that outside a set of measure $\alpha$, the contribution to \eqref{e:independentofeps} from $\sigma$ is less than $\alpha$. In particular, a large part of confirming the validity of \eqref{e:independentofeps} is done in the verification of $\hg(D_{\alpha}) < \alpha$. The justification of \eqref{e:independentofeps} will occur after verifying \eqref{e:prop of D_alpha}.

Fix $\gamma>0$ small, to be chosen later. Let $U_{\gamma}$ be an open neighborhood of $\Gamma \cap A_{k}$ such that $\sigma(U_{\gamma}\setminus (\Gamma\cap A_k))=\sigma(U_{\gamma})<\gamma$. Such a set exists because $\sigma(\Gamma\cap A_k)=0$ and as a Radon measure $\sigma$ is outer-regular. Let $F$ be an arbitrary compact subset of $\Gamma\cap A_k$. 

Since $F\subset U_{\gamma}$ compact there is some $\beta>0$ such that $d(F, \mathbb{R}^n\setminus U_{\gamma})>\beta$.
Let $\tau=\sigma \restr U_{\gamma}$. Then, for $\delta,\epsilon\in (0,\beta)$ and all $x \in F$
$$
T_{\delta}\sigma(x)-T_{\epsilon}\sigma(x)=T_{\delta}\tau(x)-T_{\epsilon}\tau(x).
$$
Indeed, assuming $\delta >  \epsilon$, for any $y$ such that $\epsilon<|y-x|<\delta$, we have $d(y, F)<\delta<\beta$, and thus $y\in U_{\gamma}$.
Define 
\begin{equation*}
T^*\tau(x)=\sup_{\epsilon>0}|T_{\epsilon}\tau(x)| \quad \text{and} \quad    D_{\alpha}=\{x\in F: T^*\tau(x)>\alpha/2\}.
\end{equation*}
For $\delta,\epsilon\in (0,\beta)$ and $x\in  F \setminus D_{\alpha}$, 
\begin{equation*}
    |T_{\delta}\sigma(x)-T_{\epsilon}\sigma(x)|=|T_{\delta}\tau(x)-T_{\epsilon}\tau(x)|\leq 2T^*\tau(x)\leq \alpha.
\end{equation*}
Now we want to show that 
$$
\hg(D_{\alpha})=\mathcal{H}^m(D_{\alpha})\leq \alpha.
$$
We begin estimating.
\begin{align}
\nonumber    \mathcal{H}^m(D_{\alpha})
    &=\mathcal{H}^m(\{x\in F: T^*\tau(x)>\alpha/2\})\\
\nonumber    &= \mathcal{H}^m\left(\left\{x\in F: \sup_{\epsilon>0}\left|\sum_{j\geq 1}\sum_{\ell=1}^{N_j}\kjl \int_{|y-x|\ge \epsilon} \pjl (y-x) d\tau(y)\right|>\alpha/2\right\}\right)\\
\label{e:002}    &\leq \mathcal{H}^m\left(\left\{x\in \Gamma: \sum_{j\geq 1}\sum_{\ell=1}^{N_j}\left|\kjl \right| \sup_{\epsilon>0}\left|\int_{|y-x|\geq \epsilon} \pjl(y-x
    )d\tau(y)\right|>\alpha/2\right\}\right)\\
\nonumber    & \le \mathcal{H}^m\left(\left\{x\in \Gamma: \sum_{j\geq 1}\sum_{\ell=1}^{N_j}j^{-2r} \tjl^{*} \tau(x) \overset{\eqref{e:Stein}}{>} \frac{\alpha \cdot c_{n}}{2\|K\|_{C^{2r}(\St^{n-1})}}\right\}\right),
\end{align}
for any $r\in \mathbb{N}$, where 
$$
\tjl^{*} \tau(x)=\sup_{\epsilon>0}\left|\int_{|y-x|\geq \epsilon}\pjl(y-x)d\tau(y)\right|.
$$
Denoting $\lambda = \frac{\alpha\cdot c_n}{2 \|K\|_{C^{2r}(\St^{n-1})}}$, and using  $\sum_{j \ge 1} \sum_{\ell=1}^{N_{j}} c_{j} j^{-2}  =1/2$  when $c_{j} \approx \frac{1}{ N_{j}} = O(j^{2-n})$, it follows
\begin{equation*}
    \left\{x\in \Gamma: \sum_{j\geq 1}\sum_{\ell=1}^{N_j}j^{-2r} \tjl^{*} \tau(x)>\lambda \right\}\subseteq \bigcup_{j\geq 1} \bigcup_{\ell=1}^{N_j}\left\{x\in \Gamma: j^{-2r} \tjl^{*} \tau(x)>\lambda \frac{c_j}{j^2}\right\}.
    \end{equation*}

Continuing the estimate from above we have 
\begin{align*}
    \mathcal{H}^m(D_{\alpha})
    &\le \mathcal{H}^m\left(\left\{x\in \Gamma: \sum_{j\geq 1}\sum_{\ell=1}^{N_j}j^{-2r} \tjl^{*} \tau(x)>\lambda \right\}\right)\\
    &  \leq  \sum_{j\geq 1} \sum_{\ell=1}^{N_j}\mathcal{H}^m\left(\left\{x\in \Gamma: j^{-2r} \tjl^{*} \tau(x)>\lambda \frac{c_j}{j^2}\right\}\right)\\
    &= \sum_{j\geq 1} \sum_{\ell=1}^{N_j}\mathcal{H}^m\left(\left\{x\in \Gamma: \tjl^{*} \tau(x)>j^{2r-2}\lambda c_j\right\}\right).
\end{align*}
By \eqref{e:L^2 boundedness} and \cite[Theorem 20.26]{mattila1999geometry}, the operators $\tjl^{*} \tau$ satisfy the following weak $(1,1)$ bound:
\begin{align*}
    \mathcal{H}^m & \left(\left\{x\in \Gamma: T^*_{j,\ell}\tau(x)>j^{2r-2}\lambda c_j\right\}\right)
    \lesssim 2^{J} j^{\frac{n-1}{2} + 1 + J} \frac{ \tau(\R^{n})}{j^{2r-2}  c_{j}}\lambda^{-1}
\end{align*}
where the constant depends on $m,n$ and the Lipschitz constant defining $\Gamma$. Continuing the main estimate from above with the weak (1,1) estimate yields 
\begin{align}\label{e:003}
\mathcal{H}^m(D_{\alpha})& \lesssim_{m,n,J} \lambda^{-1} \tau(\R^{n}) \sum_{j \ge 1} N_{j} j^{\frac{n-1}{2} + 1 + J -(2r-2) + (n-2)}
\end{align}
Recall $\tau(\R^{n}) < \gamma$. Since $N_{j} = O(j^{n-2})$, there exists $r\geq \frac{5n+2J+1}{4}$ so that $\frac{n-1}{2} + 1 + J - (2r-2) + 2(n-2) \le -2$. Defining $M/2$ to be the largest choice of $r$ made in all steps of the proof, there then exists some $c= c(m,n,J,\|K\|_{C^{M}(\St^{n-1})})$ satisfying
$\cH^{m}(D_{\alpha}) < \frac{c \gamma}{\alpha}.$ Choosing $\gamma = \frac{\alpha^{2}}{c}$ confirms $\cH^{m}(D_{\alpha}) < \alpha$, completing the proof of the existence of $D_{\alpha}$ satisfying \eqref{e:prop of D_alpha}. 

To complete the proof of \eqref{e: pv off Lip graph}, let $D \defeq \bigcap_{k=1}^{\infty}\bigcup_{i=k}^{\infty}D_{2^{-i}}$. Then $\hg(D) = 0$ and for $x\in (\Gamma\cap A_k)\setminus D$, $\lim_{\epsilon \downarrow 0} T_{\epsilon} \sigma(x)$ exists, verifying \eqref{e: pv off Lip graph}.

Combining \eqref{e:001} with \eqref{e:002}  and \eqref{e:003} gives the bound independent of $\epsilon$ for \eqref{e:independentofeps}, which in turn justifies the use of Fubini \eqref{e: pv Fubini no limit} and also the interchange of limit with the double sum in \eqref{e: pv Fubini} confirming that the existence of limits in \eqref{e: pv spherical harmonic geo hiccup} and \eqref{e: pv Fubini} are equivalent. By \eqref{e: pv on Lip graph} and \eqref{e: pv off Lip graph}, the limit in \eqref{e: pv Fubini} exists, completing the proof.
\end{proof}

Now we record several new characterizations of rectifiable measures assuming the lower density is positive.

\begin{theorem} \label{t:centeredchar}
    Let $\mu$ be a Radon measure on $\R^{n}$ satisfying $0 < \theta^{m}_{*}(\mu,x)$ for $\mu$-a.e. $x$. Then, the following are equivalent
    \begin{enumerate}
        \item $\mu$ is $m$-rectifiable.
        \item Any of the following hold for $\mu$ a.e. $x$:
        \begin{enumerate}
            \item $\theta^{m,*}(\mu,x) < \infty$ and $\int_{0}^{1} \oalpha_{\mu}(x,r)^{2} \frac{dr}{r} < \infty$.
            \item $\lim_{r \to 0} \alpha_{\mu}(x,r) = 0$. 
            \item $\lim_{r \to 0} \oalpha_{\mu}(x,r) = 0$.
        \end{enumerate}
    \end{enumerate}
\end{theorem}

We remark the lower-density assumption is in Theorem \ref{t:centeredchar} can sometimes be restrictive. For instance, in \cite{azzam2020characterization} they describe a measure $\mu$ achieved as the weak-$*$ limit of probability measures defined on the approximations of a modified Koch snowflake for which $\lim_{r \to 0} \alpha_{\mu}(x,r) = 0$ for every $x \in K$, but $K$ is not rectifiable. Nonetheless, we expect that the ease of checking Theorem \ref{t:centeredchar}(2b),(2c) compared to the a priori stronger  square function characterizations make this theorem a new useful list of sufficient conditions for rectifiability.

\begin{proof}
That (1) and (2a) are equivalent is an immediate consequence of Proposition \ref{p:centeredalpha} and the equivalences between Theorem \ref{t:characterizations}(1,2c,2e). By definition $\alpha_{\mu} \le \oalpha_{\mu}$ so (2c) implies (2b). That (1) implies (2b) is well-known, see \cite[Equation 1.5]{azzam2020characterization}. By the doubling of $\oalpha_{\mu}$, (2a) implies (2c). 

Thus, it only remains to show that (2b) implies (1). Suppose $x$ is such that $\alpha_{\mu}(x,r) \to 0$. Let $\nu \in \Tan(\mu,x)$ and $0 < \theta^{m}_{*}(\mu,x)$. Then there are $c_{i} > 0$ and $r_{i} \downarrow 0$ so that $c_{i} T_{x,r_{i}}[\mu] \xrightharpoonup{*} \nu$. By the lower-density assumption and \eqref{e:tanmeas1}, $\limsup_{i \to \infty} c_{i} r_{i}^{m} < \infty$. Thus, it suffices to prove that for all $R > 0$, 
\begin{equation} \label{e:alphanu}
\alpha_{\nu}(0,R) \le \limsup_{i} c_{i} r_{i}^{m} \alpha_{\mu}^{m}(x, r_{i} R).
\end{equation}
Indeed, since $\limsup_{i} c_{i} r_{i}^{m} < \infty$, the fact that $\lim_{i \to \infty} \alpha_{\mu}^{m}(x,r_{i} R) = 0$ implies $\alpha_{\nu}(0,R) = 0$ for all $R > 0$. In particular, $\nu \in \cM_{n,m}$.  Now Theorem \ref{t:preiss}(2.ii) implies $\mu$ is $m$-rectifiable confirming (1).

To verify \eqref{e:alphanu}, we choose $\sigma_{i} \in \cM_{n,m}$ so that $c_{i}^{-1} T_{x,r_{i}}^{-1} [\sigma_{i}] \in \cM_{n,m}$ satisfy $\alpha_{\mu}(x,r_{i}R) = (r_{i}R)^{-(m+1)} F_{B(x,r_{i}R)}(\mu, c_{i}^{-1} T_{x,r_{i}}^{-1}[ \sigma_{i}])$. By the continuity of $F_{K}(\cdot, \cdot)$ with respect to weak-$*$ convergence and the scaling of $F_{K}$, i.e.,  \eqref{e:fscaling}:
\begin{align*}
    \alpha_{\nu}(0,R) & \le \limsup_{i \to \infty} R^{-(m+1)} F_{B(0,R)}(\nu,\sigma_{i}) \\
    & = \limsup_{i \to \infty} R^{-(m+1)} F_{B(0,R)}(c_{i}T_{x,r_{i}}[\mu], \sigma_{i}) \\
    & = R^{-(m+1)}\frac{c_{i}}{r_{i}} F_{B(x,r_{i}R)}(\mu, c_{i}^{-1} T_{x,r_{i}}^{-1}[\sigma_{i}]) \\
    & = \limsup_{i \to \infty} c_{i} r_{i}^{m} \alpha_{\mu}(x,r_{i}R).
\end{align*}
\end{proof}
We now prove Theorem \ref{t:existenceofpv}.
\begin{proof}[Proof of Theorem \ref{t:existenceofpv}]
Let $\mu$ be a finite $m$-rectifiable measure. Then, by \cite{preiss1987geometry}, $0 < \theta^{m}_{*}(\mu,x) \le \theta^{m,*}(\mu,x) < \infty$ for $\mu$-a.e. $x$. By Theorem \ref{t:centeredchar}, $\lim_{r \to 0} \oalpha_{\mu}(x,r) = 0$ for $\mu$-a.e. $x$. Let $A_{0}$ be the set of full measure from Lemma \ref{l: pv exist geo hiccup}. Then,
$$
A = \{ x \in A_{0} : 0 < \theta^{m}_{*}(\mu,x) \le \theta^{m,*}(\mu,x) < \infty \quad \text{and} \quad \lim_{r \to 0} \oalpha_{\mu}(x,r) = 0 \}
$$
is a set of full measure. Fix a kernel $K \in C^{M}(\St^{n-1})$ and a norm $\| \cdot \|$. Combining Lemmas \ref{l:shapechange} and \ref{l: pv exist geo hiccup}, verifies \eqref{e:pvexists} holds for every $x \in A$. 
\end{proof}

\subsection{Proof of Theorems \ref{t:fixedpv} and \ref{t:pvchars}}
\label{s:fixedpv}

\begin{proposition}[Symmetry of $\Lambda$-tangents] \label{p:symmetry}
		Suppose that $\mu$ is a finite Borel measure over $\mathbb{R}^{n}$ such that for $\mu$-a.e. $ a \in \mathbb{R}^n$, $\theta^{m}_*(\mu, a)>0$. If for all $0 < r < R < \infty$,
    \begin{equation}
      \label{e:boundedratio}  \lim_{\epsilon \downarrow 0} \int_{\epsilon r \le |\Lambda(a)^{-1}(y-a)| 
    \le \epsilon R} \frac{\Lambda(a)^{-1}(y-a)}{|\Lambda(a)^{-1}(y-a)|^{m+1}} d \mu(y) = 0 \quad \mu-a.e. ~ a 
    \end{equation}
    then for $\mu$-almost every $a$, every $\nu \in \mbox{Tan}_\Lambda(\mu, a)$ satisfies
			\begin{align} \label{e:symgoal}
				\int_{r\leq |y-x| \leq R}\frac{y-x}{|y-x|^{m+1}} \,d\nu(y) = 0 \qquad \forall x \in \spt \nu
			\end{align}
			for all $ 0 < r< R< \infty$. That is, $\Tan_{\Lambda}(\mu,a) \subset \cS_{n}$ for $\mu$-a.e. $a \in \R^{n}$.

   In particular, if $T^{m}_{\Lambda}\mu(a)$ exists $\mu$-a.e., $\Tan_{\Lambda}(\mu,a) \subset \cS_{n}$ for $\mu$-a.e. $a$. 
	\end{proposition}

	\begin{proof}
		Consider $A$ to be the set of points $a \in \rn$ satisfying 
		\begin{enumerate}
  		\item[A1)]  $\theta^{m}_{*}(\mu,a) > 0$
			\item[A2)]  For all $0 < r < R < \infty$ \eqref{e:boundedratio} holds,
			\item[A3)] For all $\nu \in \Tan_{\Lambda}(\mu,a)$, and all $x \in \spt \nu$, $T_{x,1}[\nu] \in \Tan_{\Lambda}(\mu,a)$.
		\end{enumerate}
        By hypothesis, (A1) and (A2) hold almost everywhere. By Theorem \ref{t:tan2ltan}, (A3) also holds almost everywhere, so $A$ is a set of full measure. Suppose $a \in A$ and $\nu = \lim_{i} c_{i} T^{\Lambda}_{a,r_{i}}[\mu] \in \Tan_{\Lambda}(\mu,a)$. Then for $0 < r < R,$ 
		\begin{align*}
			\bigg|  & \int  _{r \le |y| \le R }  \frac{y}{|y|^{m+1}} d \nu(y) \bigg|  = \left| \lim_{i \to \infty}   c_{i}  \int_{r < |y| < R} \frac{y}{|y|^{m+1}} d T_{a,r_{i}}^{\Lambda}[\mu](y) \right| \\
			& = \left| \lim_{i \to \infty} c_{i} r_{i}^{m}\int_{r < |T_{a,r_{i}}^{\Lambda}(y)| < R} \frac{\Lambda(a)^{-1}(y-a)}{|\Lambda(a)^{-1}(y-a)|^{m+1}} d \mu(y) \right| \\
			& \le \limsup_{i \to \infty} c_{i} r_{i}^{m} \left| \int_{r r_{i} \le |\Lambda(a)^{-1}(y-a)| \le R r_{i} } \frac{\Lambda(a)^{-1}(y-a)}{|\Lambda(a)^{-1}(y-a)|^{m+1}} d \mu(y) \right|.
		\end{align*}
	(A2) implies this final line is well-defined and zero so long as  $\limsup_{i} c_{i} r_{i}^{m} < \infty$. Since $x \mapsto \Lambda(a) x$ is a linear isomorphism from $\rn \to \rn$, $\theta^{m}_{*}(\mu,a) > 0$ if and only if $\theta^{m}_{\Lambda,*}(\mu,a) > 0$. Now, (A1) and \eqref{e:tanmeas1} imply $\limsup_{i} c_{i} r_{i}^{m} < \infty$ verifying \eqref{e:symgoal} when $x = 0$ for all $\nu \in \Tan_{\Lambda}(\mu,a)$. Finally, (A3) says $T_{x,1}[\nu] \in \Tan_{\Lambda}(\mu,a)$ for all $x \in \spt \nu$. Since,
	$$
		\int_{r  \le |y| < R} \frac{y}{|y|^{m+1}} d T_{x,1}[\nu](y) = \int_{r \le |y-x| \le R} \frac{y-x}{|y-x|^{m+1}} d \nu(y),
	$$
	\eqref{e:symgoal} follows. By Lemma \ref{l:symchar}, this verifies the symmetry of $\nu$. Since $a \in A$ and $\nu \in \Tan_{\Lambda}(\mu,a)$ are arbitrary and $\mu(\R^{n} \setminus A) = 0$ this verifies the claimed consequences of \eqref{e:boundedratio}.
    In particular, if $T^{m}_{\Lambda} \mu(a)$ exists $\mu$-a.e. then \eqref{e:boundedratio} holds almost everywhere, verifying $\Tan_{\Lambda}(\mu,a) \subset \cS_{n}$ for $\mu$-a.e. $a$.

	\end{proof}

The next lemma provides the final step to prove Theorem \ref{t:fixedpv}. As it is interesting in its own right, we state it separately.

\begin{lemma} \label{c:sym2flat}
    Fix $\Lambda : \R^{n} \to GL(n,\R)$. Suppose $\mu$ is a Radon measure so that at almost every $a$, $0 < \theta^{m}_{*}(\mu,a)$ and $\Tan_{\Lambda}(\mu,a) \subset \cS_{n}$. Then for almost every $a$, $\Tan_{\Lambda}(\mu,a) \subset \cM_{n}$. In particular, if $\theta^{m}_{*}(\mu,a) < \infty$ almost everywhere, $\mu$ is $m$-rectifiable.
\end{lemma}

\begin{proof}
Lemma \ref{l:symprops}(2,3) imply that for all $i=0,1, \dots, n$ the $d$-cones $\cF_i = \cM_{n,i}$ and $\cM = \cS_{n}$ satisfy \eqref{e:pi} of Corollary \ref{c:jams}. Since $\Tan_{\Lambda}(\mu,a) \subset \cS_{n}$, Lemma \ref{l:symprops}(1) and Theorem \ref{t:tan2ltan} imply $\Tan_{\Lambda}(\mu,a) \cap \cM_{n} \neq \emptyset$ for $\mu$-a.e. $a$. So, Corollary \ref{c:jams} verifies $\Tan_{\Lambda}(\mu,a) \subset \cM_{n}$ for almost every $a$. If additionally $\theta^{m}_{*}(\mu,a) < \infty$, Theorem \ref{t:preiss}(iii) implies rectifiability.
\end{proof}

We are now ready to prove Theorems \ref{t:fixedpv} and \ref{t:pvchars}.

\begin{proof}[Proof of Theorem \ref{t:fixedpv}]
    Since (1) implies (2), we only verify that (2) implies $\Tan_{\Lambda}(\mu,a) \subset \cM_{n}$. By Proposition \ref{p:symmetry}, (2) implies $\Tan_{\Lambda}(\mu,a) \subset \cS_{n}$ for almost every $a$. By hypothesis, $0 < \theta^{m}_{*}(\mu,a)$ almost everywhere, so we can apply Lemma \ref{c:sym2flat} to confirm Theorem \ref{t:fixedpv}.
\end{proof}

\begin{proof}[Proof of Theorem \ref{t:pvchars}]
Fix $\mu, \Lambda$ as in the theorem statement. Since (2) implies (3). We prove (1) implies (2).

Suppose $\mu$ is $m$-rectifiable. Note, for all $x \in \R^{n}$, the kernels
$$
K_{x}(z) \defeq \frac{\Lambda(x)^{-1}z}{|\Lambda(x)^{-1}z|^{m+1}}
$$
are odd, $-m$-homogeneous, and in $C^{\infty}(\R^{n} \setminus \{0\})$. Thus, by Theorem \ref{t:existenceofpv}, there exists some set $A$ of full measure with the property that for any norm, (in particular the norm $\| z\| \defeq | \Lambda(x)^{-1} z|$ the following limit exists in $\R^{n}$
$$
\lim_{\epsilon \downarrow 0} \int_{\|y-a\|\ge \epsilon} \frac{\Lambda(x)^{-1}(y-a)}{|\Lambda(x)^{-1}(y-a)|^{m+1}} d \mu(y) \qquad \forall a \in A, ~ \forall x \in \rn.
$$
For all $x \in A$, choosing $a =x$ verifies that (1) implies (2).

We now show (3) implies (1). Suppose (3) holds. By Proposition \ref{p:symmetry}, $\Tan_{\Lambda}(\mu,a) \subset \cS_{n}$ for $\mu$-a.e., $a \in \rn$. Since $0 < \theta^{m}_{*}(\mu,a) < \infty$, Lemma \ref{c:sym2flat} ensures (1) holds, verifying the theorem.
\end{proof}

\subsection{Proof of Theorem \ref{t:dmo}} \label{s:dmo}
To prove Theorem \ref{t:dmo} we show that for suitable matrix-valued functions $A$, and suitable Radon measures $\mu$, \eqref{e:starstar} implies that $\mu$ a.e., if $\nu \in \Tan_{\Lambda}(\mu,a)$ then \eqref{e:symgoal} holds. This implies $\Tan_{\Lambda}(\mu,a) \subset \cS_{n}$, which implies flat tangents and rectifiability by Lemma \ref{c:sym2flat}. We achieve this first step by adapting the estimates in \cite[Lemma 3.12, 3.13]{molero2021l2} to prove that $\nabla_{1} \Gamma_{A}(x,y)$ and $\nabla_{1} \Theta(x,y;A(x))$ are sufficiently close at small scales, see Lemma \ref{l:integrated}.

We first introduce some terminology and notation from \cite{molero2021l2}. Please note that the setting in  \cite{molero2021l2} is $\R^{n+1}$, while here we adapt to the setting of $\R^{n}$. 

A Lebesgue measurable function $\theta : [0, \infty] \to [0, \infty]$ is called $\kappa$-doubling if $\theta(t) \le \kappa \theta(s)$ for all $s \in [t/2, t]$. A $\kappa$-doubling function $\theta$ is in $DS(\kappa)$ (resp. $DL_{d}(\kappa)$) if $\int_{0}^{1} \theta(t) \frac{dt}{t} < \infty$ (resp. $\int_{1}^{\infty} \theta(t) \frac{dt}{t^{d+1}} < \infty$). These spaces are the Dini spaces for $\kappa$-doubling functions at small (resp. large) scales. Note that if $d_{1} < d_{2}$ then $DL_{d_{1}}(\kappa) \subset DL_{d_{2}}(\kappa)$. Given a matrix-valued function $A : \R^{n} \to \R^{n \times n}$, for any $x \in \R^{n}$ and $r > 0$ define 
\begin{align*}
    \overline{A}_{x,r} := \dashint_{B(x,r)} A(z) dz
\end{align*} 
and 
$$
\omega_{A}(r) = \sup_{x \in \R^{n}} \dashint_{B(x,r)} \|A(z) - \overline{A}_{x,r}\|dz.
$$
Further, denote
\begin{align*}
& \fL^{d}_{\theta}(r) = r^{d} \int_{r}^{\infty} \theta(t) \frac{dt}{t^{d+1}} \quad \text{and} \quad  \fI_{\theta}(r) = \int_{0}^{r} \theta(t) \frac{dt}{t}.
\end{align*}
The matrix-valued function $A$ is said to be in $\DMOs$  if, for some $\kappa< \infty$, $\omega_{A} \in DS(\kappa)$. $A$ is said to be in $\DMOl$  if, for some $\kappa< \infty$, $\omega_{A} \in DL_{n-2}(\kappa)$. We say that $A \in \DDMOs$ if both $A \in \DMOs$ and 
\begin{equation} \label{e:fi1}
\int_{0}^{1} \fI_{\omega_{A}}(r) \frac{dr}{r} = \int_{0}^{1} \int_{0}^{r} \omega_{A}(t) \frac{dt}{t} \frac{dr}{r} < \infty.
\end{equation}
The spaces $\DMOs$ (resp. $\DMOl$) stand for Dini mean oscillation at small scales (resp. at large scales) and $\DDMOs$ stands for double Dini mean oscillation at small scales. We write $A \in \DMO$ if $A \in \DDMOs \cap \DMOl$. 
\begin{remark} \label{r:unifcont} It is known that if $A \in \DMO$, then there exists $\tilde{A}$ such that $\tilde{A}$ is uniformly continuous and $A= \tilde{A}$ almost everywhere, \cite[Appendix A]{hwang2020green}. If we consider $A, \tilde{A} \in \DMO$ so that $A = \tilde{A}$  Lebesgue a.e., then for all $y \in \R^{n}$ the fundamental solutions with pole at $y$ corresponding to each differential equation are the same, that is if $L_{A}\Gamma_{A}(\cdot,y) = \delta_{y}$ then  $L_{\tilde{A}}\Gamma_{A}(\cdot,y) = \delta_{y}$ (and vice versa). Indeed, since by assumption $L_{A}\Gamma_{A}(\cdot,y) = \delta_{y}$ we have that for all $\varphi \in C_{c}^{\infty}(\R^{n})$
$$
    \int_{\R^{n}} \langle \tilde{A}(x) \nabla_{1} \Gamma_{A}(x,y), \nabla \varphi(x) \rangle d x =\int_{\R^{n}} \langle A(x) \nabla_{1} \Gamma_{A}(x,y), \nabla \varphi(x) \rangle d x = \varphi(y).
$$
 Similarly, $\omega_{A} (\cdot) = \omega_{\tilde{A}}(\cdot)$ on $[0,\infty]$. In particular, there is no loss in generality in assuming that $A \in \DMO$ is uniformly continuous, even when studying measures $\mu$ which are mutually singular with respect to the Lebesgue measure.
\end{remark} 
Finally, we define 
\begin{equation*} 
\tau_{A}(r) := \fI_{\omega_{A}}(r) + \fL^{n-1}_{\omega_{A}}(r) = \int_{0}^{r} \omega_{A}(t) \frac{dt}{t} + r^{n-1} \int_{r}^{\infty} \omega_{A}(t) \frac{dt}{t^{n}}
\end{equation*}
and
$$
\widehat{\tau}_{A}(R) := \fI_{\omega_{A}}(R) + \fL^{n-2}_{\omega_{A}}(R) = \int_{0}^{R} \omega_{A}(t) \frac{dt}{t} + R^{n-2} \int_{R}^{\infty} \omega_{A}(t) \frac{dt}{t^{n-1}}.
$$

\begin{remark} \label{r:tauzero}
    In \cite[p. 7]{molero2021l2} it is observed that $A \in \DMO$ implies both $\fI_{\tau_{A}}(1) < \infty$ and $\widehat{\tau}_{A}(R) < \infty$ for all $R>0$. In particular, this means that when $A \in \DMO$, $\lim_{r \to 0} \tau_{A}(r) = 0$. In fact, if $A \in  \DMOl \cap \DMOs$ then $\lim_{r \to 0} \tau_{A}(r) + \widehat{\tau}_{A}(r) = 0$. Indeed, \cite[Remark 2.2]{molero2021l2} says that if $\omega_{A} \in DS(\kappa) \cap DL_{n-2}(\kappa)$, then $\cL^{n-2}_{\omega_{A}}(R) \to 0$ as $R \to 0$. The fact that $\omega_{A} \in DS(\kappa) \cap DL_{n-2}(\kappa)$ is precisely the statement $A \in \DMOl \cap \DMOs$. 
    Thus we also know $\lim_{R \to 0} \widehat{\tau}_{A}(R) = 0$ when $A \in \DMOl \cap \DMOs$.
\end{remark}

Throughout this section, we will always let $A$ denote a uniformly elliptic matrix valued function from $\R^{n} \to \R^{n \times n}$ with uniform ellipticity constant $\Lambda_{0}$. That is, $|\xi|^{2} \Lambda_{0}^{-1} \le \langle A(x) \xi, \xi \rangle$ and $\langle A(x) \xi, \eta \rangle \le \Lambda_{0} |\xi| |\eta|$ for all $x, \xi, \eta \in \R^{n}$. We also fix $\kappa < \infty$ so that $\omega_{A}$ is $\kappa$-doubling. 

Some ideas behind the "frozen coefficient method" in the next two lemmas are already present in \cite{kenig2011layer,conde2019failure}, but the next Lemma comes directly from \cite[Lemma 3.12]{molero2021l2}.

\begin{lemma}[Lemma 3.12 from \cite{molero2021l2}] \label{312} 
Suppose $A \in \DMOs \cap \DMOl$ and $n \ge 3$. For $R_0 > 0$, there exists $C_0= C_0(n, \Lambda_{0}, R_{0}) > 0$ such that for $x,y \in \R^{n}$ and $0 < |y-x| < R < R_{0}$, 
$$
\left| \nabla_{1} \Gamma_{A}(x,y) - \nabla_{1} \Theta\left(x,y; \overline{A}_{x,\frac{|y-x|}{2}} \right) \right| \le C_0 \frac{ \tau_{A}(\frac{|y-x|}{2})}{|y-x|^{n-1}} + C
_0\frac{\widehat{\tau}_{A}(R)}{R^{n-1}} 
$$
\end{lemma}

In \cite[Lemma 3.13]{molero2021l2} the authors estimate the difference between $\nabla_{1} \Theta(x,y; \overline{A}_{x,\delta})$  and $\nabla_{1} \Theta(x,y; \overline{A}_{x,\rho})$. In our next lemma, we take advantage of the without loss of generality assumption that $A$ is uniformly continuous to similarly directly estimate the difference between $\nabla_{1} \Theta(x,y; \overline{A}_{x,\delta})$ and $\nabla_{1}\Theta(x,y;A(x))$.
\begin{lemma} \label{313}
Let $A \in \DMOs$, $n \ge 3$, $0 < \delta < 1$, and $x \in \R^{n}$. Then
    \begin{equation} \label{e:313}
        \left| \nabla_{1} \Theta(x,y; \overline{A}_{x, \delta} )-\nabla_{1} \Theta(x, y; A(x))  \right| \lesssim_{n, \Lambda_{0}} \frac{1}{|y-x|^{n-1}} \int_{0}^{\delta} \omega_{A}(t) \frac{dt}{t}.
    \end{equation}
\end{lemma}

\begin{proof}

First recall that the matrix inverse and determinant functions are locally Lipschitz on $GL(n,\R)$. In particular, when restricted to $\Lambda_{0}$-uniformly elliptic matrices, these functions are Lipschitz with constant depending on $n, \Lambda_{0}$. This implies that the matrix-valued mappings $ A(x)^{-1}$ and $\det(A(x))^{1/n}A(x)^{-1}$ inherit the modulus of mean oscillation of $A(x)$ up to a constant depending on $\Lambda_{0},n$. In particular, defining $L_{x, \delta} = \det(\overline{A}_{x, \delta})^{1/n} \overline{A}_{x, \delta}^{-1}$ and $L(x) = \det(A(x))^{1/n} A(x)^{-1}$ it follows that for $N \ge 0$
\begin{align*}
        \|A(x)^{-1}-\overline{A}^{-1}_{x,2^{-N}\delta}\| &\lesssim_{\Lambda_0,n} \|A(x)-\overline{A}_{x, 2^{-N}\delta}\| \xrightarrow{N \to \infty}  0\\
         \|L(x) - L_{x,2^{-N} \delta}\|&\lesssim_{\Lambda_0,n} \|A(x)-\overline{A}_{x, 2^{-N}\delta}\| \xrightarrow{N \to \infty}  0 \\
         \|\overline{A}^{-1}_{x, 2^{-N}\delta}-\overline{A}^{-1}_{x, 2^{-(N+1)}\delta}\| &\lesssim_{\Lambda_0,n} \|\overline{A}_{x, 2^{-N}\delta}-\overline{A}_{x, 2^{-(N+1)}\delta}\| \leq \omega_A(2^{-N}\delta) \\
        \|L_{x,2^{-N}\delta} - L_{x,2^{-(N+1)}} &\lesssim_{\Lambda_0,n} \|\overline{A}_{x, 2^{-N}\delta}-\overline{A}_{x, 2^{-(N+1)}\delta}\| \leq \omega_A(2^{-N}\delta)
        \end{align*}
In particular,
    \begin{align}
        &\nonumber\|A(x)^{-1}-\overline{A}^{-1}_{x, \delta}\|  
        \leq  \sum_{N=0}^{\infty} \|\overline{A}^{-1}_{x, 2^{-N}\delta}-\overline{A}^{-1}_{x, 2^{-(N+1)}\delta}\| \\
        & \label{eq interdini}  + \lim_{N \to \infty} \|A(x)^{-1}-\overline{A}^{-1}_{x,2^{-N}\delta}\| \lesssim_{\Lambda_0,n,\kappa} \int^{\delta}_0\omega_A(t) \frac{dt}{t}
    \end{align}
    and similarly
    \begin{equation} \label{e:Ldini}
        \|L(x) - L_{x,\delta}\| \lesssim_{\Lambda_{0},n,\kappa} \int_{0}^{\delta} \omega_{A}(t) \frac{dt}{t}.
    \end{equation}
We now continue with a straightforward computation
    \begin{align*}
    & |\nabla_{1} \Theta(x,y; \overline{A}_{x, \delta}) - \nabla_{1} \Theta(x,y; A(x))| \\
    &= c_{n} \left| \frac{\overline{A}_{x, \delta}^{-1} (y-x)}{\det(\overline{A}_{x, \delta})^{1/2} \langle \overline{A}_{x, \delta}^{-1}(y-x), y-x \rangle^{n/2}} - \frac{ A(x)^{-1}(y-x)}{\det(A(x))^{1/2} \langle A(x)^{-1}(y-x), y-x \rangle^{n/2}} \right| \\
    & \le c_{n} \left| \frac{\overline{A}_{x, \delta}^{-1} (y-x)}{\det(\overline{A}_{x, \delta})^{1/2} \langle \overline{A}_{x, \delta}^{-1}(y-x), y-x \rangle^{n/2}}  - \frac{A(x)^{-1} (y-x)}{\det(\overline{A}_{x, \delta})^{1/2} \langle \overline{A}_{x, \delta}^{-1}(y-x), y-x \rangle^{n/2}} \right| \\
    & \quad + c_{n} \left| \frac{A(x)^{-1} (y-x)}{\det(\overline{A}_{x, \delta})^{1/2} \langle \overline{A}_{x, \delta}^{-1}(y-x), y-x \rangle^{n/2}}  -  \frac{ A(x)^{-1}(y-x)}{\det(A(x))^{1/2} \langle A(x)^{-1}(y-x), y-x \rangle^{n/2}}\right| \\
    & = c_{n} \left| \frac{(\overline{A}_{x, \delta}^{-1} - A(x)^{-1})(y-x)}{\det(\overline{A}_{x, \delta})^{1/2} \langle \overline{A}_{x, \delta}^{-1}(y-x), y-x \rangle^{n/2}} \right| \\
    & \quad + c_{n} \left| A(x)^{-1}(y-x) \frac{ \Langle L(x)(y-x), y-x \Rangle^{n/2} - \Langle L_{x, \delta}(y-x), y-x \Rangle^{n/2}}{\Langle L(x)(y-x),y-x \Rangle^{n/2} \Langle L_{x, \delta}(y-x),y-x)\Rangle^{n/2}} \right| =: I + II
    \end{align*}
    where $L_{x, \delta} = \det(\overline{A}_{x, \delta})^{1/n} \overline{A}_{x, \delta}^{-1}$ and $L(x) = \det(A(x))^{1/n} A(x)^{-1}$. 
    By uniform ellipticity, we estimate the first term:
    \begin{align*}
        I &  \le c_{n} \frac{\|\overline{A}_{x, \delta}^{-1} - A(x)^{-1}\|}{\det(\overline{A}_{x, \delta})^{1/2}} \frac{|y-x|}{(\Lambda_{0}^{-1})^{n/2} |y-x|^{n}} \lesssim_{\Lambda_{0},n} \frac{\|\overline{A}_{x, \delta}^{-1}- A(x)^{-1}\|}{|y-x|^{n-1}} \lesssim_{\Lambda_{0},n,\kappa} \frac{\displaystyle \int_0^{\delta} \omega_A(t)\,\frac{dt}{t}}{|y-x|^{n-1}}.
    \end{align*}
For the second term we have
    \begin{align*}
         II & \lesssim_{\Lambda_{0},n}  \frac{|A(x)^{-1}(y-x)|}{|y-x|^{2n}} \left| \Langle L(x)(y-x),y-x \Rangle^{n/2} - \Langle L_{x, \delta}(y-x), y-x \Rangle^{n/2} \right| \\
        & \lesssim_{\Lambda_{0},n} \frac{1}{|y-x|^{2n-1}}\left|\Langle (L(x)-L_{x, \delta})(y-x),y-x \Rangle\right| \Langle (L(x) + L_{x, \delta})(y-x),y-x \Rangle^{n/2-1} \\
        &\lesssim_{\Lambda_{0},n} \frac{|y-x|^{2} |y-x|^{n-2}}{|y-x|^{2n-1}} \|L(x)-L_{x, \delta}\| = \frac{\|L(x)-L_{x, \delta}\|}{|y-x|^{n-1}}.
    \end{align*}
    since  $a, b > 0$  and $n\geq 3$ implies $|a^{\frac{n}{2}} - b^{\frac{n}{2}}| \lesssim_{n} |a-b|(a+b)^{\frac{n}{2}-1}$. In the last inequality we used the fact that $L(x), L_{x, \delta}$ are uniformly elliptic, $\|L(x) +L_{x, \delta}\| \simeq_{\Lambda_{0},n} 1$. 
   Using inequality \eqref{e:Ldini}, we have 
    \begin{align*}
         II & \lesssim_{\Lambda_{0},n, \kappa}  \frac{1}{|y-x|^{n-1}} \int_0^{\delta} \omega_A(t)\,\frac{dt}{t},
    \end{align*}
    Thus verifying \eqref{e:313}.
\end{proof}

The following is a consequence of the double Dini condition that follows from $A \in \DMO$. This estimate will be crucial to estimating the right-hand-side of inequality \eqref{e:integrated} when it appears in the proof of Theorem \ref{t:dmo}.

\begin{lemma}[c.f., \cite{molero2021l2} Lemma 2.1]  \label{l:21}
Suppose $A \in \DMO$ and $n \ge 3$. Then 
    \begin{align}
        \int_0^1 \fL^{n-1}_{\omega_{A}}(r) \,\frac{dr}{r} \leq \frac{1}{n-1}\tau_A(1) 
    \end{align}
\end{lemma}
\begin{proof}
    \begin{align*}
        \int_0^1 \fL^{n-1}_{\omega_{A}}(r) \,\frac{dr}{r}&=\int_0^1 r^{n-1} \int_r^{\infty} \omega_A(s) s^{-n}\,ds\frac{dr}{r}\\
        &=\int_0^1 r^{n-1} \int_r^{1} \omega_A(s) s^{-n}\,ds\frac{dr}{r}+\int_0^1 r^{n-1} \int_1^{\infty} \omega_A(s) s^{-n}\,ds\frac{dr}{r}
    \end{align*}
Observe that Fubini implies
    \begin{align*}
        \int_0^1 \int_r^1 f(r, s)\,dsdr = \int_0^1 \int_0^sf(r,s) \,drds.
    \end{align*}
Therefore,
    \begin{align*}
        \int_0^1 r^{n-2} \int_r^{1} \omega_A(s) s^{-n}\,dsdr &= \int_0^1\int_0^{s} r^{n-2}  \omega_A(s) s^{-n}\,drds\\
        &= \frac{1}{n-1}\int_0^1  s^{n-1}  \omega_A(s) s^{-n}\,drds \leq \frac{1}{n-1} \fI_{\omega_{A}}(1)
    \end{align*}
The second integral is easier to handle
    \begin{align*}
        \int_0^1 r^{n-1} \int_1^{\infty} \omega_A(s) s^{-n}\,ds\frac{dr}{r}\leq \frac{1}{n-1} \int_1^{\infty} \omega_A(s) s^{-n}\,ds= \frac{1}{n-1} \fL^{n-1}_{\omega_{A}}(1)
    \end{align*}
\end{proof}

The next lemma is a consequence of integrating Lemmas \ref{312} and \ref{313}.

\begin{lemma} \label{l:integrated}
    Let $A : \R^{n} \to \R^{n \times n}$ be a uniformly elliptic matrix-valued function, $n \ge 3$, satisfying $A \in \DMO$. Let $\mu$ be a finite Borel measure. For each $R>r > 0$, define a Borel set $E_{R, r}$ such that
    \begin{align*}
        E_{R, r} \subset B(0, R) \setminus B(0, r).
    \end{align*}
    Then for every $x \in \R^{n}$ and $R \le 1$,
    \begin{align}
        \nonumber \int_{y-x \in E_{R, r}} &\left|  \nabla_{1} \Gamma_{A}(x,y) - \nabla_{1} \Theta \left(x,y; A(x) \right) \right| d \mu(y)  \\
        \label{e:integrated} &\lesssim \left(\sup_{\rho>0}\frac{\mu(B(x, \rho))}{\rho^{n-1}}\right)\left( \int_{r}^Rt^{-1}\tau_{A}(t)\,dt + \widehat{\tau}_{A}(R)\right).
    \end{align}
    Moreover, if $A \in \DDMOs$, then
    \begin{equation} \label{e:ddmos}
        \lim_{R \to 0} \sup_{0 < r < R} \left| \int_{r}^{R} \frac{\tau_{A}(t)}{t} dt + \widehat{\tau}_{A}(R) \right| = 0.
    \end{equation}
\end{lemma}

\begin{proof}
We write,
\begin{align}
\nonumber & \left|  \nabla_{1} \Gamma_{A}(x,y) - \nabla_{1} \Theta \left(x,y; A(x) \right) \right| \le \left|  \nabla_{1} \Gamma_{A}(x,y) - \nabla_{1} \Theta\left(x,y; \overline{A}_{x, \frac{|y-x|}{2}} \right) \right| \\
\label{e:iandii}& \hspace{1.5cm}+ \left| \nabla_{1} \Theta\left(x,y; \overline{A}_{x, \frac{|y-x|}{2}} \right) - \nabla_{1} \Theta \left(x,y; A(x) \right) \right| \eqdef I + II
\end{align}
If $y-x \in E_{R, r}$ then 
\begin{equation} \label{e:er}
\frac{1}{2} r \le r_{y} := \frac{|y-x|}{2} \le \frac{1}{2} R.
\end{equation}
it follows from Lemma \ref{312} that 
\begin{equation*} 
    I \lesssim_{C_{0}} \frac{\tau_{A} \left( r_y\right) }{\left(r_y\right)^{n-1}} + \frac{ \widehat{\tau}_{A}\left( R\right)}{ R^{n-1}}.
\end{equation*}
Recalling $E_{R, r} \subset B(0,  R) \setminus B(0,  r)$ it follows 
\begin{align} \label{e:312i}
\nonumber \int_{E_{R}} I ~ d \mu(y)& \lesssim_{C_{0},  \kappa} \sum_{r< 2^{-k} <R}\frac{\mu(B(x,2^{-k}))}{(2^{-k})^{n-1}}  \tau_{A}(2^{-k})+\frac{\mu(B(x,R))}{R^{n-1}} \widehat{\tau}_{A}(R)\\
\nonumber &\le \left(\sup_{r \le \rho \le R }\frac{\mu(B(x,\rho))}{\rho^{n-1}} \right) \left(\sum_{r< 2^{-k} <R}\tau_{A}(2^{-k})+  \widehat{\tau}_{A}(R)\right) \\
& \lesssim \left( \sup_{r \le \rho \le R} \frac{\mu(B(x,\rho))}{\rho^{n-1}} \right) \left(\int_{r}^{R} \tau_{A}(t) \frac{d t}{t} + \widehat{\tau}_{A}(R)  \right)
\end{align}

The second term decomposes similarly
    \begin{align}
        &\int_{y-x \in E_{R, r}}\left| \nabla_{1} \Theta\left(x,y; \overline{A}_{x, \frac{|y-x|}{2}} \right) - \nabla_{1} \Theta \left(x,y; A(x) \right) \right|d \mu(y) \nonumber\\
        &\leq \sum_{r <2^{-k} < R} \int_{y \in B(x, 2^{-k}) \setminus B(x, 2^{-(k+1)})}\left| \nabla_{1} \Theta\left(x,y; \overline{A}_{x, \frac{|y-x|}{2}} \right) - \nabla_{1} \Theta \left(x,y; A(x) \right) \right|d \mu(y)\nonumber\\
        &\lesssim \sum_{r < 2^{-k} < R} \int_{y \in B(x, 2^{-k}) \setminus B(x, 2^{-(k+1)})} \frac{ \fI_{\omega_{A}}(|x-y|)}{|x-y|^{n+1}}d \mu(y)\lesssim \sum_{r < 2^{-k} < R} \frac{\mu(B(x,2^{-k}))}{2^{-k(n-1)}} \fI_{\omega_{A}}(2^{-k})\nonumber\\
        &\lesssim  \left(\sup_{r \le \rho \le R }\frac{\mu(B(x,\rho))}{\rho^{n-1}} \right) \int_{r}^{R} \frac{\fI_{\omega_{A}}(\rho)}{\rho} d \rho.\label{eq 312ii}
    \end{align}
where we have used Lemma \ref{313}.

Using the $\kappa$-doubling properties of $\tau, \fI_{\omega_{A}}$ in combination with \eqref{e:iandii}, \eqref{e:312i}, and \eqref{eq 312ii} implies
    \begin{align}
        \nonumber    \int_{y-x \in E_{R, r}}&  \left|  \nabla_{1} \Gamma_{A}(x,y) - \nabla_{1} \Theta \left(x,y; A(x) \right) \right| d \mu(y)  \le \int_{y-x \in E_{R,r}} I + II \, d \mu(y) \\
        \label{e:iii}    & \lesssim_{C_{0},n,\Lambda_{0}, \kappa} \left( \sup_{r \le \rho \le R} \frac{\mu \left(B(x,\rho) \right)}{\rho^{n-1}} \right) \left( \widehat{\tau}_{A}(R) + \int_{r}^{R} \frac{\tau_{A}(\rho)+\fI_{\omega_{A}}(\rho)}{\rho} d \rho \right). 
    \end{align} 
Since $\tau_{A} = \fI_{\omega_{A}} + \fL^{n-1}_{\omega_{A}}$, this verifies \eqref{e:integrated}. Moreover, If $A \in \DDMOs$,
    \begin{align*}
        \int_{r}^{R} \frac{\tau_{A}(\rho)}{\rho} d \rho\le 2 \int_{r}^{R} \frac{\fI_{\omega_{A}}(\rho)+\cL^{n-1}_{\omega_{A}}(\rho)}{\rho} d \rho
    \end{align*}
To confirm this converges to $0$ as $R \to 0$, with bound independent of $r$, note Lemma \ref{l:21} guarantees $\int_{0}^{1} \frac{\mathfrak{L}^{n-1}_{\omega_{A}}(\rho)}{\rho} d \rho < \infty$ implying $\lim_{R \to 0} \int_{0}^{R}  \frac{\fL^{n-1}_{\omega_{A}}(\rho)}{\rho} d \rho = 0$. On the other hand, the fact that $\lim_{R \to 0}   \widehat{\tau}_{A}(R)+\int_{0}^{R} \frac{\fI_{\omega_{A}}(\rho)}{\rho} d \rho = 0$ is the definition of $A \in \DDMOs$, verifying \eqref{e:ddmos} holds.
\end{proof}

\begin{proof}[Proof of Theorem \ref{t:dmo}]
    Note that (2) implies (3). We first show that (3) implies (1). By Theorem \ref{t:pvchars} it suffices to prove that (3) implies that for $\mu$-a.e., $a \in \rn$,
    \begin{equation} \label{e:zero2}
        \lim_{\epsilon \downarrow 0}   \int_{\epsilon r \le |\Lambda(a)^{-1}(y-a)| \le \epsilon R} \nabla_{1} \Theta(a,y;A(a)) d \mu(y)  = 0.
    \end{equation}
     To verify \eqref{e:zero2}, suppose without loss of generality that $A$ is continuous, see Remark \ref{r:unifcont}. Let $G$ be the collection of $x \in \R^{n}$ such that $\theta^{n-1}_{*}(\mu,x) > 0$, $\theta^{n-1,*}(\mu,x) < \infty$, \eqref{e:starstar} holds, and the conclusion of Theorem \ref{t:tan2ltan} holds. Then by assumption, and Theorem \ref{t:tan2ltan}, $\mu(\R^{n} \setminus G) = 0$. Fix $a \in G$.
    By the triangle inequality,
    \begin{align*}
    \bigg| \int&_{B_{\Lambda}(a, R r_{i}) \setminus  B_{\Lambda}(a, rr_{i})}  \nabla_{1} \Theta(a,y;A(a)) d \mu(y) \bigg|  \\
    \le &\bigg| \int_{B_{\Lambda}(a, R r_{i}) \setminus  B_{\Lambda}(a, rr_{i})} \nabla_{1} \Gamma_{A}(a,y) d \mu(y) \bigg| \\
    & + \int_{B_{\Lambda}(a, R r_{i}) \setminus  B_{\Lambda}(a, rr_{i})} \big|  \nabla_{1} \Theta(a,y; A(a)) - \nabla_{1} \Gamma_{A}(a,y) \big| d \mu(y)=I+II.
    \end{align*}    
   By \eqref{e:boundeddmoversion} and $a \in G$, the term $I$ tends to zero as $i \to \infty$. Choosing $E_{R,r} =  E_{i} \defeq B_{\Lambda}(0, Rr_{i}) \setminus B_{\Lambda}(0, r r_{i}) \subset B(0, \Lambda_{0}R r_{i}) \setminus B(0, \Lambda_{0}^{-1} r r_{i})$ in Lemma \ref{l:integrated} implies
    $$
    II \lesssim \left( \sup_{\rho> 0} \frac{\mu(B(a, \rho))}{\rho^{n-1}} \right) \left(  \widehat{\tau}_{A}(\Lambda_{0} Rr_{i}) + \int_{\Lambda_{0}^{-1} r r_{i}}^{\Lambda_{0} R r_{i}} \tau(\rho) \frac{d \rho}{\rho} \right)
    $$
    which tends to zero as $i \to \infty$ since $a \in G$ implies $\theta^{n-1,*}(\mu,a) < \infty$ and $A \in \DMOs \cap \DMOl$, see Remark \ref{r:tauzero}. This confirms \eqref{e:zero2} as desired.

    We now confirm that (1) implies (3). By Theorem \ref{t:pvchars}(2), for $\mu$-a.e. $a \in \R^{n}$ 
    \begin{equation} \label{e:zero}
        \lim_{r < R \downarrow 0}   \int_{ r \le |\Lambda(a)^{-1}(y-a)| \le R} \nabla_{1} \Theta(a,y;A(a)) d \mu(y)  = 0.
    \end{equation}
    In particular, \eqref{e:zero2} holds.  
    But then, for any $0 <r < R < \infty$, we estimate
    \begin{align*}
         \bigg| \int&_{\epsilon r \le |\Lambda(a)^{-1}(y-a)|\le \epsilon R} \nabla_{1} \Gamma_{A}(x,y) d \mu(y) \bigg|  \le    \left| \int_{ \epsilon r \le |\Lambda(a)^{-1}(y-a)| \le \epsilon R} \nabla_{1} \Theta(a,y;A(a)) d \mu(y) \right|  \\
        & + \int_{ \epsilon r \le |\Lambda(a)^{-1}(y-a)| \le \epsilon R}\left| \nabla_{1} \Gamma_{A}(x,y) - \nabla_{1} \Theta(a,y;A(a)) \right| d \mu(y) = I + II.
        \end{align*}
        \eqref{e:zero2} says precisely that $\lim_{\epsilon \to 0} I = 0$. So we now estimate $II$ using Lemma \ref{l:integrated} with the set $E_{R,r} = \{ \epsilon r \le | \Lambda(a)^{-1}(y-a)| \le \epsilon R\}$:
        \begin{align*}
            II  & \lesssim_{\Lambda_{0},n,\kappa} \left( \sup_{\rho > 0} \frac{ \mu \left( B(a,\rho) \right)}{\rho^{n-1}} \right)  \left(  \widehat{\tau}(\Lambda_{0} R \epsilon) + \int_{\Lambda_{0}^{-1}r \epsilon}^{\Lambda_{0}R\epsilon} \frac{ \tau_{A}(\rho)}{\rho} d \rho \right) \\
        & \le \left( \sup_{\rho > 0} \frac{ \mu \left( B(a,\rho) \right)}{\rho^{n-1}} \right) \left( \widehat{\tau}_{A}(\Lambda_{0} R \epsilon) + \log \left( \frac{\Lambda_{0}^{2}R}{r} \right) \tau_{A}\left( \Lambda_{0} R \epsilon \right) \right),
        \end{align*}

    where the final line uses that $\tau_{A}$ is a non-decreasing function. In particular, since $\theta^{n-1,*}(\mu,a) < \infty$, Remark \ref{r:tauzero} ensures $\lim_{\epsilon \downarrow 0} II = 0$,  verifying (3) for any $0 < r < R$.

    Now we show (1) implies (2). Hence, we suppose $A \in \DMO$. Then as in the previous computation, by choosing $E_{R,r} = \{ r \le \Lambda(a)^{-1}(y-a)| \le R\}$ when applying Lemma \ref{l:integrated}, we deduce
        \begin{align*}
         \bigg| \int_{ r \le |\Lambda(a)^{-1}(y-a)|\le  R}& \nabla_{1} \Gamma_{A}(x,y) d \mu(y) \bigg|  \le    \left| \int_{  r \le |\Lambda(a)^{-1}(y-a)| \le R} \nabla_{1} \Theta(a,y;A(a)) d \mu(y) \right|  \\
        & + \left| \int_{ r \le |\Lambda(a)^{-1}(y-a)| \le  R} \nabla_{1} \Gamma_{A}(x,y) - \nabla_{1} \Theta(a,y;A(a)) d \mu(y) \right| \\
        & \lesssim_{\Lambda_{0},n,\kappa} \left( \sup_{\rho > 0} \frac{ \mu \left( B(a,\rho) \right)}{\rho^{n-1}} \right) \left(  \widehat{\tau}(\Lambda_{0} R ) + \int_{\Lambda_{0}^{-1}r}^{\Lambda_{0}R} \frac{ \tau_{A}(\rho)}{\rho} d \rho \right). 
    \end{align*}
    and the bound is independent of $r$, whenever $r < R$. Therefore, Lemma \ref{l:integrated} implies this converges to zero as $r < R \to 0$, verifying (2). 

    Finally, the proof that (1) implies (4) is similar to the proof that (1) implies (2). This time we use Theorem \ref{t:existenceofpv} to deduce that for $\mu$-a.e. $a \in \rn$
    \begin{align*}
\lim_{r < R \downarrow 0} \int_{|y-a| \ge \epsilon} \nabla_{1} \Theta_{A}(x,y;A(x)) d \mu(y) = 0.
    \end{align*}
    Then applying Lemma \ref{l:integrated} with $E_{R,r} = \{ r \le |y-a| \le R \}$ we deduce as before that
    \begin{align*}
        &\lim_{r < R \downarrow 0} \left| \int_{r \le |y-a| \le R} \nabla_{1} \Gamma_{A}(x,y) d \mu(y) \right| \\
        &\le \lim_{r < R \downarrow 0} \left( \sup_{\rho > 0} \frac{ \mu \left( B(a,\rho) \right)}{\rho^{n-1}} \right) \left( \widehat{\tau}(R) + \int_{r}^{R} \frac{\tau_{A}(\rho)}{\rho} d \rho \right) = 0, 
    \end{align*}
    confirming (4).
\end{proof}

\bibliographystyle{alpha}
\bibdata{references}
\bibliography{references}

\end{document}